\documentclass[dvips,aihp]{imsart}

\usepackage{amsthm,amsmath}
\usepackage{amssymb}
\usepackage{amsopn}

\usepackage{examplep,booktabs}
\usepackage{mathrsfs}
\RequirePackage[colorlinks,citecolor=blue,urlcolor=blue]{hyperref}

\arxiv{0804.0676}

\startlocaldefs
\newtheorem{theorem}{Theorem}
\newtheorem{proposition}{Proposition}
\newtheorem{lemma}{Lemma}
\newtheorem{definition}{Definition}
\newtheorem{corollary}{Corollary}

\theoremstyle{definition}

\newtheorem{remark}{Remark}

  \def\diag{\mathop{\rm diag}}

  \def\Ex{{\bf E}}
  \def\Pb{{\bf P}}
  \def\var{{\rm Var}}

  \def\tr{{\rm Tr}}
  \def\emptyset{\varnothing}

  \def\1{\boldsymbol{1}}

\def\sfT{\textsf{T\!}}

\def\mX{\mathscr X}

\def\mfc{\mathfrak c}

\def\sfh{\textsf{h}}
\def\sfr{\textsf{r}}

\def\mP{\mathscr P}

\def\mO{\mathcal O}
\def\eps{\varepsilon}
\def\cal{\mathcal}
\def\xxi{{\boldsymbol{\xi}}}
\def\zzeta{{\boldsymbol{\zeta}}}
\def\llambda{{\boldsymbol{\lambda}}}
\def\mB{\boldsymbol{B}}

\def\T{\top}
\def\bds#1{\boldsymbol{#1}}

\newcommand{\bbA}{{\mathbb A}}
\newcommand{\bbB}{{\mathbb B}}

\newcommand{\bbD}{{\mathbb D}}

\newcommand{\bbG}{{\mathbb G}}
\newcommand{\bbH}{{\mathbb H}}

\newcommand{\bbN}{{\mathbb N}}

\newcommand{\bbR}{{\mathbb R}}

\newcommand{\bbZ}{{\mathbb Z}}

\def\E{{\bf E}}

\def\P{{\bf P}}

\def\T{{\bf T}}

\def\cale{{\cal E}}
\def\calf{{\cal F}}

\def\calk{{\cal K}}

\def\caln{{\cal N}}

\def\calp{{\cal P}}

\def\calt{{\cal T}}

\def\calv{{\cal V}}

\def\calx{{\cal X}}
\def\caly{{\cal Y}}
\def\calz{{\cal Z}}
\def\ii{{\!\:\rm i}}
\def\ep{\epsilon}
\def\half{\frac{1}{2}}
\def\iku{\rightarrow}

\def\nn{\nonumber}
\def\be{\begin{equation}}
\def\ee{\end{equation}}
\def\bea{\begin{eqnarray}}
\def\eea{\end{eqnarray}}
\def\beas{\begin{eqnarray*}}
\def\eeas{\end{eqnarray*}}

\def\KIJ{K_{I\!J}}
\def\l{\left}
\def\r{\right}
\def\msE{\mathscr E}
\def\mX{\boldsymbol{X}}
\def\mbeta{\boldsymbol{\beta}}
\def\msigma{\boldsymbol{\sigma}}

\endlocaldefs

\begin{document}

\begin{frontmatter}

\title{Second-order asymptotic expansion for a non-synchronous covariation estimator}

\runtitle{Expansion for the HY-estimator}

\begin{aug}

\author{\fnms{Arnak} \snm{Dalalyan}\thanksref{a}\ead[label=e1]{dalalyan@imagine.enpc.fr}}
\and
\author{\fnms{Nakahiro} \snm{Yoshida}\thanksref{b}\ead[label=e2]{nakahiro@ms.u-tokyo.ac.jp}}
\address[a]{LIGM/IMAGINE, Ecole des Ponts ParisTech, Universit\'e Paris-Est, \printead{e1}}
\address[b]{University of Tokyo and Japan Science and Technology Agency, \printead{e2}}

\runauthor{Dalalyan and Yoshida}
\end{aug}

\begin{abstract}
In this paper, we consider the problem of estimating the covariation of two diffusion processes when observations are 
subject to non-synchronicity. Building on recent papers \cite{Hay-Yos03, Hay-Yos04}, we derive second-order asymptotic expansions for 
the distribution of the Hayashi-Yoshida estimator in a fairly general setup including random sampling schemes and non-anticipative 
random drifts. The key steps leading to our results are a second-order decomposition of the estimator's distribution in the Gaussian 
set-up, a stochastic decomposition of the estimator itself and an accurate evaluation of the Malliavin covariance. To give a concrete 
example, we compute the constants involved in the resulting expansions for the particular case of sampling scheme generated by two 
independent Poisson processes.
\end{abstract}

\begin{abstract}[language=french]
Dans cet article, nous consid\'erons le probl\`eme d'estimation de la covariation de deux processus de diffusion
observ\'es de fa\c{c}on asynchrone. Nous nous pla\c{c}ons dans le cadre pr\'esent\'e dans \cite{Hay-Yos03, Hay-Yos04} et \'etablissons un d\'eveloppement
asymptotique au second ordre de la loi de l'estimateur de Hayashi-Yoshida. Ce d\'eveloppement est valable pour les drifts al\'eatoires non-anticipatifs 
et pour des pas d'\'echantillonnage irr\'eguliers, \'eventuellement al\'eatoires, mais ind\'ependant des processus observ\'es. L'approche utilis\'ee pour obtenir
les principaux r\'esultats peut \^etre d\'ecompos\'ee en trois \'etapes. La premi\`ere consiste \`a \'etablir un d\'eveloppement au second-ordre de la loi de l'estimateur
dans le cadre Gaussien. La deuxi\`eme est l'obtention d'une d\'ecomposition stochastique de l'estimateur lui-m\^eme et la dernière est l'\'evaluation de la covariance
de Malliavin. A titre d'exemple, nous calculons les constantes du d\'eveloppement  au second ordre dans le cas o\`u l'\'echantillonnage est obtenu par deux processus
de Poisson ind\'ependants.
\end{abstract}




\end{frontmatter}

\section{Introduction}

In the last decade, studies on covariance estimation has attracted considerable attention thanks to 
the applications in mathematical finance and econometrics; see \textit{e.g.} Andersen and Bollerslev 
\cite{AB(98)}, Comte and Renault \cite{CR98}, Andersen \textit{et al.}\ \cite{ABDE01,ABDL01}, Barndorff-Nielsen and 
Shephard \cite{BNS04a}. All these papers consider the situation where two diffusion processes are observed at 
the same discrete instants. In contrast with this, covariance estimation under a ``non-synchronous'' sampling 
scheme has rarely been treated theoretically in spite of its importance in the analysis of high-frequency financial data 
\cite{Shanken,Lo,Scholes}. 
The first contributions to the statistical inference for covariance estimation with non-synchronous data have been 
made by Hayashi and  Yoshida~\cite{Hay-Yos03, Hay-Yos04}. They proposed an estimator of the covariation and explored 
its statistical properties such as the consistency and the asymptotic normality. Interestingly, it follows from the 
results  in~\cite{Hay-Yos04} that the drifts of the observed diffusions do not affect the asymptotic variance of the 
covariance estimator. The aim of the present paper is to complement the results in \cite{Hay-Yos03,Hay-Yos04} by 
establishing a second-order asymptotic expansion for the distribution of the covariance estimator. In particular, we 
get explicit expressions that have the advantage of reflecting the impact of drifts on the asymptotic distribution 
of the estimator.

One common approach to cope with non-synchronicity is the following. First, two regularly spaced time series are generated 
by interpolating the observed non-synchronous data. Then the realized covariance estimator is computed for the 
interpolated time series. However, it is known that such a synchronization technique causes estimation bias, which is often 
referred to as the {\it Epps effect} \cite{Epps}. Another estimator of the covariance, based on the harmonic analysis, has been 
proposed by Malliavin and Mancino~\cite{Mal-Man02}. In the case where in addition to the non-synchronicity the data is contaminated
by a microstructure noise, estimators of the covariance have been proposed by Palandri \cite{Palandri}, Barndorff-Nielsen 
\textit{et al.} \cite{BNH} and Zhang \cite{Zhang06}. A detailed account on covariance estimation for 
non-synchronous data can be found in~\cite{Hay-Yos06} and \cite{Zhang06}. 

In order to present the framework and to describe our contributions, we need some notation.
Let $\mX=(X_1,X_2)$ be a two dimensional diffusion process given by
\bea\label{071219-1}
d\mX_t=\mbeta_t\,dt+\diag(\msigma_t)\,d\mB_t,
\eea
where
$\mB=((B_{1,t},B_{2,t})^\T,\;t\ge 0)$ is a two dimensional Gaussian process
with independent increments, zero mean and covariance matrix
\begin{equation*}
\Ex[\mB_t\cdot \mB^\T_t]=
\begin{pmatrix}
t & \int_0^t \rho_s\,ds\\
\int_0^t \rho_s\,ds & t
\end{pmatrix},\qquad \forall t\ge 0.
\end{equation*}
In (\ref{071219-1}), $\mbeta=(\beta_1,\beta_2)^\T$ is a progressively measurable process, $\msigma=(\sigma_1,\sigma_2)^\T$ is a deterministic function and 
$\diag(\msigma)$ stands for the diagonal matrix having $\sigma_i$ as $i^{\text{th}}$ diagonal entry, $i=1,2$. In what follows, 
we restrict our attention to the case when $\sigma_1$, $\sigma_2$ and $\rho$ are deterministic functions; the functions $\sigma_i$, 
$i=1,2$ take positive values while $\rho$ takes values in the interval $[-1,1]$. Note that the marginal processes $B_{1}$ and 
$B_2$ are Brownian motions (BM). Moreover, we can define a process $B^*_t$ such that $(B_{1,t},B^*_t)_{t\ge 0}$  is a two-dimensional BM
and $dB_{2,t}=\rho_t dB_{1,t}+\sqrt{1-\rho_t^2}\,dB^*_{t}$ for every $t\geq 0$.

We will assume that the processes $X_1$ and $X_2$ are observed respectively at
the time instants $0=S^0<S^1<\ldots<S^{N_1}=T$ and $0=T^0<\ldots<
T^{N_2}=T$. Let us denote $I^{i}=(S^{i-1},S^{i}]$ and $J^j=(T^{j-1},T^{j}]$.
The families
$\Pi^1=\{I^i,i=1,\ldots,N_1\}$ and $\Pi^2=\{J^j,j=1,\ldots,N_2\}$ are partitions
of the interval $[0,T]$. We will also use the notation
$\Delta_i X_1=X_{1,S^{i}}-X_{1,S^{i-1}}$ and
$\Delta_j X_2=X_{2,T^{j}}-X_{2,T^{j-1}}$.

In this paper, we are concerned with the problem of estimating the parameter
\begin{equation*}
\theta=\int_0^T \rho_t\sigma_{1,t}\sigma_{2,t}\,dt=\langle X_{1},X_{2}\rangle_T
\end{equation*}
based on the observations $(X_{1,S^i},X_{2,T^j},i=0,\ldots,N_1,
j=0,\ldots,N_2)$. The parameter $\theta$ represents the covariance between the martingale parts
of $X_1$ and $X_2$. Therefore, it can be used to evaluate the correlation between the two
BMs $B_1$ and $B_2$.

If the processes $X_1$ and $X_2$ are synchronously observed, the sum of cross products
$\sum_{i=1}^{N_1}\Delta_iX_1\cdot\Delta_i X_2$ is a natural
estimator of $\theta$. Indeed, it converges in probability
to $\theta$ when the maximum lag of the sampling times tends to $0$
in probability. In the field of statistical inference for stochastic processes,
this fact has been applied to estimating the volatility and the covariation between 
semimartingales. The asymptotic distributions are well investigated; see Dacunha-Castelle and Florens-Zmirou 
\cite{Dac-Flo86}, Florens-Zmirou \cite{Flo91}, Prakasa Rao \cite{Pra83,Pra88}, 
Yoshida \cite{Yos92b}, Genon-Catalot and Jacod \cite{Gen-Jac93}, Kessler \cite{Kes97}, 
and Mykland and Zhang \cite{MZ05a}.

An estimator of $\theta$, which is unbiased when the drift $\mbeta$ is identically zero, has been 
proposed in  \cite{Hay-Yos03}. Henceforth called HY-estimator, it is defined as follows:
\begin{equation}\label{est}
\hat\theta=\sum_{i=1}^{N_1}\sum_{j=1}^{N_2} \Delta_iX_1\cdot\Delta_j X_2
\cdot\mathbf{1}(I^i\cap J^j\not=\emptyset).
\end{equation}
It is established in \cite{Hay-Yos03} that under mild assumptions, $\hat\theta$ is consistent as the maximum lag of the
sampling times tends to $0$ in probability. Kusuoka and Hayashi \cite{Kus-Hay04} extended the 
consistency result to a more general sampling scheme. Asymptotic normality of the HY-estimator was proved in 
Hayashi and Yoshida \cite{Hay-Yos04} under the assumption that the sampling times are independent of the process $\mX$. 
For related literature, see Hoshikawa \textit{et al.\ }\cite{HKNN06}, Griffin and Oomen \cite{GO06}, Robert and 
Rosenbaum \cite{RR08} and Voev and Lunde \cite{VL06}. 
The general case of a sampling scheme depending on the process $\mX$ has been studied in Hayashi and 
Yoshida~\cite{Hay-Yos06,Hay-Yos08}, where a stochastic analytic proof of the asymptotic mixed normality of the 
HY-estimator is presented. An estimator for the variance of the HY-estimator under the assumption that the observed 
process $\mX$ has no drift has been recently proposed by Mykland~\cite{Myk06}.

In the present work, the main emphasis is put on the higher-order asymptotic 
behavior of the HY-estimator. Note that the theory of asymptotic expansions is one of chapters of 
statistics that received a revival of interest owing to its usefulness for exploring properties
of bootstrap-based statistical methods. For a comprehensive introduction to this subject we refer
the reader to Hall~\cite{Hall92}. Results on asymptotic expansions in other contexts can be
found in Bose \cite{Bose88}, Mykland \cite{Myk93}, Koul and Surgailis \cite{KS97}, Bertail and 
Cl\'emen\c{c}on \cite{BC04}, Zhang et al.\ \cite{ZMS06}, Fukasawa \cite{Fuka08} and the references therein. 

Section \ref{asy.exp.dist} contains an asymptotic expansion of the distribution of the HY-estimator. 
As a first step for deriving asymptotic expansions for the distribution of the HY-estimator, we 
give in Section \ref{ch.f.} a representation of the cumulants of $\hat{\theta}$ as 
functionals of the sampling times, and obtain asymptotic estimates for them.
This is used to derive a second-order asymptotic expansion of the characteristic function of the estimator 
while the asymptotic normality is also proved as an application of those estimates. 

The application of these results in the setup of Poisson sampling schemes is presented in Section \ref{Sec4}. 
We assume that the Poisson processes generating the sampling times have constant intensities
$np_1$ and $np_2$, where $n$ is a parameter guaranteeing the high-frequency of the observations ($n\to\infty$).  
This setup has the advantage of making it possible to compute all the quantities involved in the asymptotic expansion. 
We show that the residual term in the proposed asymptotic expansion of the distribution of  $\sqrt{n}(\hat\theta_n-\theta)$ 
behaves nearly like $n^{-1}$, as $n$ goes to infinity.

When there are (possibly random) drift terms in the stochastic differential equation
of $\mX_t$, some additional terms appear in the asymptotic expansion.
In order to identify these terms, we derive in Section~\ref{071219-2} a stochastic decomposition
of the HY-estimator and explore the asymptotic behavior of the variables appearing in the second-order terms.
Since the asymptotics we get is non-Gaussian, the classical techniques leading to Edgeworth expansions
can not be used. Instead, our arguments rely on the limit theory for semimartingales. 

The asymptotic expansion  of the distribution of the HY-estimator is carried out in Section~\ref{071219-3} using a perturbation method.
We apply the Malliavin calculus first to ensure the regularity of 
the distribution of the principal part---a quadratic form of Gaussian random variables---and then to extend this property 
to the model under the perturbation. To enhance the legibility, we postpone the most technical proofs to the last three 
sections.

\section{Elementary properties of $\hat\theta$}\label{S2}

As noticed by Mykland \cite{Myk06}, the estimator $\hat\theta$ is the Maximum
Likelihood Estimator (MLE) of $\theta$. Let us present here some computations
that not only show that $\hat\theta$ is the MLE of $\theta$, but also
give some interesting insight concerning the efficiency properties of
the HY-estimator $\hat\theta$. Let us deal with a slightly more general setup. Assume that
$\xxi\in\bbR^N$ is a random vector having centered Gaussian
distribution with unknown covariance matrix $\Sigma$. The entries of
the matrix $\Sigma$ are $\sigma_{\ell,\ell'}=
E[\xi_\ell\xi_{\ell'}]$ for $\ell,\ell'=1,\ldots,N$. We want to
estimate a linear combination
\begin{equation*}
\theta=\sum_{\ell,\ell'=1}^N a_{\ell,\ell'} \sigma_{\ell,\ell'},
\end{equation*}
where $a_{\ell,\ell'}\in\bbR$, $\ell,\ell'=1,\ldots,N$ are some
known numbers verifying $a_{\ell,\ell'}=a_{\ell',\ell}$.

In order to use results on the exponential family, it is convenient to consider
the parametrization by the entries of the inverse, denoted by $V=\Sigma^{-1}$, 
of the covariance matrix $\Sigma$. Set $p=(N^2+N)/2$ and write
\begin{equation*}
V=\begin{pmatrix}
v_1 &v_2 & \ldots & v_N\\
v_2 &v_{N+1}&\ldots & v_{2N-1}\\
\vdots&\vdots&\ddots&\vdots\\
v_N &v_{2N-1}&\ldots&v_p
\end{pmatrix}.
\end{equation*}
The log-likelihood function can now be written as follows:
\begin{equation}\label{logl}
\ell(V)=\frac12\log|V|-\frac12\sum_{k=1}^p v_k\sfT_k(\xxi),
\end{equation}
where $|V|$ denotes the determinant of the matrix $V$ and
$\sfT\,(\xxi)=(\sfT_1(\xxi),\sfT_2(\xxi),\ldots)$ is defined by
\begin{equation*}
\sfT_1(\xxi)=\xi_1^2,\quad \sfT_2(\xxi)=2\xi_1\xi_2,\quad \sfT_3(\xxi)=2\xi_1\xi_3,
\quad\ldots,\quad \sfT_p(\xxi)=\xi_N^2.
\end{equation*}
It follows from (\ref{logl}) that the distribution $P_V$ of the Gaussian
vector $\xxi\sim \mathcal N_N(0,V^{-1})$ belongs to the (simple) exponential
family. This implies that the statistic $\sfT\,(\xxi)$ is the MLE of the
parameter
$
\tau=E[\sfT\,(\xxi)]=(\sigma_{11},2\sigma_{12},\ldots,\sigma_{NN})^\T$. 
Hence, the MLE of $\theta=\sum_{\ell,\ell'}a_{\ell,\ell'}\sigma_{\ell,\ell'}$
is $\hat\theta=\sum_{\ell,\ell'}a_{\ell,\ell'} \xi_\ell\xi_{\ell'}$.
It is easily seen that this estimator is unbiased.
Furthermore, since $\sfT\,(\xxi)$ is a complete sufficient statistic,
the MLE $\hat\theta=\sum_{\ell,\ell'}
a_{\ell,\ell'} \xi_\ell\xi_{\ell'}$ is the best unbiased estimator
of $\theta$ in the sense that any other unbiased estimator will have
a variance at least as large as that of $\hat\theta$.

We can now return to our model. The vector
\begin{equation*}\xxi=(\Delta_1 X_1,\ldots,\Delta_{N_1}X_1,\Delta_1 X_2,\ldots,
\Delta_{N_2}X_2)^\T
\end{equation*}
is drawn from an $N=N_1+N_2$ dimensional centered
Gaussian distribution. In addition, the parameter $\theta=Cov(X_{1,T},X_{2,T})$
can be represented in the form $\sum_{\ell,\ell'}a_{\ell,\ell'}
\sigma_{\ell,\ell'}$ with
\begin{equation*}
a_{\ell,\ell'}=\frac12\,\1(\ell\leq N_1,\ell'>N_1, I^\ell\cap
J^{\ell'-N_1}\not=\emptyset)
\end{equation*}
for every $\ell\leq \ell'$ and $a_{\ell,\ell'}=a_{\ell',\ell}$ for $\ell>\ell'$.
Therefore, the arguments presented above yield the following result.

\begin{proposition} The estimator
$\hat\theta$ defined by (\ref{est}) is the MLE of $\theta$. Moreover,
it is the estimator having the smallest quadratic risk among all
unbiased estimators of $\theta$.
\end{proposition}

This proposition advocates for using the HY-estimator in the case where $\mbeta\equiv 0$. 
If the latter condition is not satisfied, $\hat\theta$ is not necessarily unbiased, but under 
very mild assumptions it is consistent \cite{Hay-Yos03} and asymptotically normal \cite{Hay-Yos04} 
as the maximum lag of the sampling times tends to $0$. This explains the popularity of the
HY-estimator motivating our interest in its second-order asymptotic expansion. At a heuristical level,
the construction of the HY-estimator can be derived from the decomposition $\theta=\sum_{i,j} \1(I^i\cap J^j\not=\emptyset)\int_{I^i\cap J^j} 
\sigma_{1,t}\sigma_{2,t}\rho_t\,dt$. Indeed, each term of that decomposition is nearly equal  to the covariance 
of the increments $\Delta_i X_1$ and $\Delta_j X_2$, since the martingale part of a small increment of a 
semi-martingale dominates the increment of the bounded-variation part. Hence, if $I^i$ and $J^j$ are small, 
it is reasonable to estimate $\int_{I^i\cap J^j} \sigma_{1,t}\sigma_{2,t}\rho_t\,dt$ by the product
$\Delta_i X_1\cdot\Delta_j X_2$ and, therefore, to estimate $\theta$ by the HY-estimator $\hat\theta$.

\section{Asymptotic expansion of the distribution in Gaussian setup}\label{asy.exp.dist}

\subsection{Notation and main results}
In this section, we will derive the second-order asymptotic expansion
of the distribution of $b_n^{-1/2}(\hat\theta_n-\theta)$, where $b_n$ is a suitably chosen normalization
factor, for the model (\ref{071219-1}) without drifts. We will treat a model with drifts in Section~\ref{071219-2}, 
where we will resort to the Malliavin calculus for dealing with general nonlinear Wiener functionals.


Given positive numbers $M$ and $\gamma$,
let $\cale(M,\gamma)$ denote the set of measurable functions
$f:\bbR\iku\bbR$ satisfying
$|f(x)|\leq M(1+|x|^\gamma)$ for all $x\in\bbR$.
For positive numbers ${\sf C}$, $\eta$, $r_0$ and $\mfc^*$ we set
\beas
\cale^0=\cale^0({\sf C},\eta,r_0,\mfc^*)
=\Big\{f:\>
\int_\bbR\bar{\omega}_f(z,r)\phi(z;\mfc^*)dz\leq{\sf C}r^{\eta},\ \forall r\le r_0\Big\},
\eeas
where
\beas
\bar{\omega}_f(z,r)=\sup_{x:|x|\le r}|f(z+x)-f(z)|
\eeas
and $\phi(z;\Sigma)$ is the density of the centered normal distribution with variance $\Sigma$.
Note that this class is large enough to contain most functions that are encountered
in practice. In particular, all functions satisfying the generalized H\"older condition
$|f(z+x)-f(z)|\le F(z)|x|^\eta$ with some function $F$ such that $\int F(z)\phi(z;\mfc^*)\,dz\le \sf C$
belong to $\cale^0({\sf C},\eta,\infty,\mfc^*)$. It is also easy to check that the set of all
indicator functions of intervals of $\bbR$ is included in $\cale^0(\sqrt{2\pi} \mfc^*,1,\infty,\mfc^*)$ for any $\mfc^*>0$.

Our aim is now to get uniformly in $f\in\cale^*$ an asymptotic expansion for the sequence 
$\Ex[f(b_n^{-1/2}(\hat\theta_n-\theta))]$ with $\cale^*=\cale(M,\gamma)\cap\cale^0({\sf C},\eta,r_0,\mfc^*)$. To this end,
define $h_r(z;\Sigma)$ as the $r$-th Hermite polynomial given by
\beas
h_r(z;\Sigma)=(-1)^r\phi(z;\Sigma)^{-1}
\partial_z^r\phi(z;\Sigma),\qquad \forall z\in\bbR.
\eeas
In particular, $h_2(z;\Sigma)=(z^2-\Sigma)/\Sigma^2$ and 
$h_3(z;\Sigma)=(z^3-3\Sigma z)/\Sigma^3$. Along with the Hermite polynomials, 
it is customary to express the second-order asymptotic expansion of a distribution
in terms of the first-order and the second-order cumulants. To define this quantities
in the present framework, let us denote, for any Borel set $S\subset\bbR$, 
\begin{align}\label{vS}
v(S)=\int_S \rho_t\sigma_{1,t}\sigma_{2,t}\,dt,\quad
v_1(S)=\int_S \sigma_{1,t}^2\,dt,\quad
v_2(S)=\int_S \sigma_{2,t}^2\,dt,
\end{align}
and introduce
\begin{align}
\mu_{2}&= \frac12\Big\{\sum_{I\!,J}v_1(I)v_2(J)\KIJ +
\sum_{I\in\Pi^1} v(I)^2 +\sum_{J\in\Pi^2}v(J)^2-
\sum_{I\!,J} v(I\cap J)^2\Big\},\label{mu221}\\
\mu_{3}&=\frac14\Big\{\sum_{I\in\Pi^1} v(I)^3+\sum_{J\in\Pi^2}v(J)^3+2\sum_{I\!,J}v(I\cap J)^3
+3\sum_{I\!,J} v_1(I)v_2(J)v(I\cup J)\KIJ\nonumber\\ &\quad\quad
-3\sum_{I\!,J} [v(I\cap J)^2(v(I)+v(J))-v(I\cap J)v(I)v(J)]\Big\},\label{mu321}
\end{align}
where $\KIJ =\1(I\cap J\not=\emptyset)$ and
$\sum_{I\!,J}=\sum_{I\in\Pi^1}\sum_{J\in\Pi^2}$. 
Since we are dealing with the asymptotics of high frequency data, we will assume that
all the intervals $I^i=I_n^i$ and $J^j=J_n^j$ depend on some parameter $n$---representing the frequency 
of the sampling---that is large. To make the dependence on $n$ explicit, we will write $\mu_{2,n}$ and $\mu_{3,n}$ 
instead of $\mu_2$ and $\mu_3$. Furthermore, as the time interval $[0,T]$ is fixed, the maximal sampling step 
$r_n=[(\max_{i} |I_{n}^i|)\vee(\max_j |J_{n}^j|)]$ is assumed to tend to zero as $n\to\infty$. 
Using this notation, we define 
\begin{align}\label{bar_lambda}
\bar{\lambda}_{2,n}= 2 \>b_n^{-1}\mu_{2,n},\qquad\text{and}\qquad\bar{\lambda}_{3,n}= 8 \>b_n^{-2}\mu_{3,n},
\end{align}
for some deterministic sequence $b_n$, tending to zero as $n\to\infty$. To some extent, one can think of $b_n$ as the rate  
of convergence of $\mu_{2,n}$ to zero. This point will become clearer in Section~\ref{Sec4}, where the concrete
example of the Poisson sampling scheme is analyzed.

We introduce a $\sigma[\Pi]$-dependent random signed-measure $\Psi^\Pi_n$ on $\bbR$
by the density
\beas
p_{3,n}(z)&=&
\phi(z;\bar{\lambda}_{2,n})\Big[1
+\frac{b_n^{1/2}}{6}\bar{\lambda}_{3,n} \> h_3(z;\bar{\lambda}_{2,n})\Big].
\eeas
It is not hard to check that the Fourier transform of $\Psi^\Pi_n$ is given by
\beas
\hat{\Psi}^\Pi_n(u)=
e^{-\half \bar{\lambda}_{2,n}u^2} \Big[ 1
+\frac{b_n^{1/2}}{6}\bar{\lambda}_{3,n}({\ii}u)^3 \Big].
\eeas
In the case where no assumption on the convergence of $\mu_{2,n}$ is made, the measure $\Psi_n^\Pi$
will serve as the second-order approximation to the distribution of $\calx_n=b_n^{-1/2}(\hat\theta_n-\theta)$. However,
for many sampling schemes one can prove the convergence of $\bar\lambda_{2,n}$ to some constant $\mfc$, 
implying that the estimator $\hat\theta_n$ is asymptotically normal with asymptotic variance $\mfc$. 
It is therefore natural to address the issue of approximating the distribution of $\calx_n$ by a measure similar 
to $\Psi_n^\Pi$ but based on the Gaussian density with variance $\mfc$. To this end, we define 
the signed measure $\tilde{\Psi}^\Pi_n$ on $\bbR$ by the density
\beas
\tilde{p}_{3,n}(z)=
\phi(z;\mfc)\Bigl[1
+\half (\bar{\lambda}_{2,n}-\mfc)h_2(z;\mfc)
+\frac{b_n^{1/2}}{6}\bar{\lambda}_{3,n} \> h_3(z;\mfc)\Bigr].
\eeas

The following result, the proof of which is deferred to Section~\ref{App0}, asserts that
$p_{3,n}$ and $\tilde p_{3,n}$ are good approximations to the density of $(\hat\theta_n-\theta)/\sqrt{b_n}$. 

\begin{theorem}\label{191214-1}
Let $M,\gamma,\eta,{\sf C},r_0,\mfc^*>0$ be the parameters describing the set of functions of interest. For 
$a\in(\frac{3}{4},1)$ and $\mfc,\mfc_0,\mfc_1\in(0,\mfc^*)$ set
\begin{align*}
\calp_n(\mfc_0,\mfc_1,a)&=\{\> \mfc_0<\bar{\lambda}_{2,n}<\mfc_1,\>r_n\leq b_n^a\>\},\\
A_n(a) &= \{\>(\bar{\lambda}_{2,n}-\mfc)^2\leq b_n^{2a-1},\ r_n\leq b_n^a\>\},
\end{align*}
where $r_n$ is the maximal lag of the sampling times and $\bar\lambda_{2,n}=2b_n^{-1}\mu_{2,n}$. 
Then, there exists a sequence
$\ep_n=\ep_n(M,\gamma,\eta,{\sf C},r_0,a,\mfc_0,\mfc_1)$ such that
$\ep_n=O(b_n^{2a-1})$ and 
the inequalities
\begin{align}
\sup_{f\in\cale(M,\gamma)\cap\cale^0({\sf C},\eta,r_0,\mfc^*)} \l| \>\Ex^\Pi[f(\calx_n)]-\Psi^\Pi_n[f]\>\r| &\leq
\ep_n,\qquad \forall \Pi_n\in\calp_n(\mfc_0,\mfc_1,a),\label{thm1a}\\
\sup_{f\in\cale(M,\gamma)\cap\cale^0({\sf C},\eta,r_0,\mfc^*)} \l| \>\Ex^\Pi[f(\calx_n)]-\tilde\Psi^\Pi_n[f]\>\r| &\leq
\ep_n,\qquad \forall \Pi_n\in A_n(a)\label{thm1b},
\end{align}
hold true, where $\calx_n=b_n^{-1/2}(\hat\theta_n-\theta)$.
\end{theorem}

\begin{remark}
The approximating measure $\Psi_n^\Pi$ provided by Theorem \ref{191214-1} contains the
Gaussian density with variance $\bar{\lambda}_{2,n}$, which depends on $n$.
One can easily deduce from that result that the distribution of
$(b_n\bar{\lambda}_{2,n})^{-1/2}(\hat{\theta}_n-\theta)$ can be approximated
by the measure
$$
\Big[\,1+\frac{\sqrt{b_n}\bar\lambda_{3,n}}{6}\Big(z^3-\frac{3}{\bar{\lambda}_{2,n}}z\Big)\Big]\phi(z;1)\>dz.
$$
\end{remark}

The following result is an immediate consequence of~(\ref{thm1b}) and provides
an unconditional asymptotic expansion for the distribution of $\calx_n=b_n^{-1/2}(\hat\theta_n-\theta)$.

\begin{theorem}\label{ThmAE3} Under the notation of Theorem~\ref{191214-1}, if\/  $\Pb(A_n(a)^c)=o(b_n^{p})$ for every $p>1$, and  $\Ex[\bar\lambda_{2,n}-\mfc]=O(b_n^{2a-1})$,
then
\begin{equation}\label{ae3}
\sup_{f\in\cale(M,\gamma)\cap\cale^0({\sf C},\eta,r_0,\mfc^*)} \l| \>\Ex[f(\calx_n)]-\int_\bbR f(z)\,p_n^*(z)\,dz\>\r|=O(b_n^{2a-1}),
\end{equation}
where $p_n^*(z)=\phi(z;\mfc)\Bigl[1
+\frac{b_n^{1/2}}{6}\Ex[\bar{\lambda}_{3,n}] \> h_3(z;\mfc)\Bigr]$. Moreover, if $\sup_{n\in\bbN}\Ex[\bar{\lambda}_{3,n}]<\infty$, then
relation (\ref{ae3}) holds with $p_n^*$ replaced by
$$
p_n^+(z)=\frac{\max(0,p_n^*(z))}{\int_\bbR \max(0,p_n^*(u))\,du}\ ,
$$
which is a probability density.
\end{theorem}

{


\subsection{Gaussian analysis and expansion of the characteristic function}\label{ch.f.}

The goal of this section is to prepare the ground for the proof of Theorem~\ref{191214-1}. 
To this end, we present in Section~\ref{SSec2} general results on the characteristic function
of a random variable that can be written as a quadratic functional of a standard Gaussian vector. 
As usual, this characteristic function involves the cumulants that take a simplified form in
the context of the HY-estimator. Section~\ref{SSec3} is devoted to proving that the second and 
the third cumulants for the HY-estimator can be computed using formulae~(\ref{mu221}) and~(\ref{mu321}).
These results lead to a second-order expansion of the characteristic function of the HY-estimator,
which is rigorously stated and proved in Section~\ref{SSec4}. Finally, the proof of Theorem~\ref{191214-1}
is presented in Section~\ref{SSec5}.

\subsubsection{General Gaussian setup}\label{SSec2}
In order to determine the asymptotic expansion of the distribution of
$\hat\theta$, we start with expanding its characteristic function. It will be useful
for our purposes to consider the more general setup defined via
Gaussian vector $\xxi$ and the matrix
$A=(a_{\ell,\ell'})_{\ell,\ell'=1}^N$, see Section~\ref{S2}.

Recall that
$$
\hat\theta=\xxi^\T A\xxi\qquad \text{and}\qquad
\xxi\sim \mathcal N_N(0,\Sigma).
$$
In other terms, $\hat\theta$ is a quadratic form of a centered Gaussian vector. 
The aim of the present subsection is twofold. Firstly, we compute the cumulants of any
quadratic form $Q$ of a Gaussian vector $\xxi$ as functions of the matrix associated 
to the quadratic form $Q$ and the covariance matrix of $\xxi$. Among other things, 
this computation allows us to give a simple condition implying the weak convergence
of a series of quadratic forms of Gaussian vectors. The second goal of the present
subsection is to show that the tails of the characteristic function of a quadratic
form of a Gaussian vector have at least polynomial decay. To achieve this second goal, 
we establish an explicit upper bound for the characteristic function of interest.
It should be pointed out that most results and conditions are stated in terms of 
the spectral characteristics of the matrix $\Sigma^{1/2}A\Sigma^{1/2}$.

Since $A$ is a symmetric matrix, the $N$-by-$N$ matrix $\Sigma^{1/2}A
\Sigma^{1/2}$ is symmetric and therefore diagonalizable. Let $\Lambda$ and
$U$ be respectively the $N$-by-$N$ diagonal and orthogonal matrices
such that $\Sigma^{1/2}A\Sigma^{1/2}=U^\T\Lambda U$. Let
$\zzeta$ be a Gaussian $\mathcal N_N(0,I_N)$ vector such that
$\xxi=\Sigma^{1/2}\cdot U^\T\zzeta$.
Such a vector exists always and it is unique if $\Sigma$ is invertible.
In this notation, we have
$$
\hat\theta=\zzeta^\T\Lambda\zzeta=\sum_{\ell=1}^N \lambda_\ell \zeta_\ell^2,
$$
where $\lambda_1,\ldots,\lambda_N$ are the eigenvalues of the matrix
$\Sigma^{1/2}A\Sigma^{1/2}$ and $\zeta_1,\ldots,\zeta_N$ are independent
Gaussian random variables. This implies that $\zeta_\ell^2$s are independent
and distributed according to the $\chi^2_1$ distribution. Hence
$\Ex[e^{{\ii}u\zeta_\ell^2}]=(1-2{\ii}u)^{-1/2}$ and
$$
\varphi_{\hat\theta}(u):=\Ex[e^{{\ii}u\hat\theta}]=
\prod_{\ell=1}^N (1-2\ii \lambda_\ell u)^{-1/2}.
$$
By taking the logarithm and using its Taylor series we get
$$
\log\varphi_{\hat\theta}(u) =-\frac12\sum_{\ell=1}^N\log(1-2\ii \lambda_\ell u)
=\frac12\sum_{\ell=1}^N\sum_{k=1}^\infty\frac{(2\ii\lambda_\ell u)^k}{k},
$$
as soon as $|u|< 1/(2\max_{\ell} |\lambda_\ell|)$.
Since all the series in the above formula are absolutely convergent,
we can change the order of summation. This yields
\begin{align}\label{logfi}
\log\varphi_{\hat\theta}(u)
=\sum_{k=1}^\infty\frac{(2\ii u)^k}{2k}\mu_{k},
\qquad |u|< 1/(2\|\llambda\|_\infty),
\end{align}
with $\|\llambda\|_\infty=\max_{\ell} |\lambda_\ell|$ and
$
\mu_k=\sum_{\ell=1}^N \lambda_\ell^k=\tr[(\Sigma^{1/2}A\Sigma^{1/2})^k]
=\tr[(\Sigma\cdot A)^k]$, where the last equality follows from the property 
$\tr(M_1\cdot M_2)= \tr(M_2\cdot M_1)$ provided that both products are well 
defined. Separating the first two terms in the RHS of (\ref{logfi}),
we arrive at
\begin{align}\label{logfi1}
\log\varphi_{\hat\theta}(u)
=\ii\theta u-u^2\mu_2+\sum_{k=3}^\infty\frac{(2\ii u)^k}{2k}\mu_{k},
\qquad |u|< 1/(2\|\llambda\|_\infty).
\end{align}
Let us define $\bar\alpha=\|\llambda\|_\infty/\|\llambda\|_2$. Using simple inequalities, one checks that
$|\mu_{k}|\le {\bar\alpha}^{k-2}\mu_{2}^{k/2}$ for every $k\geq 3$.
Therefore,
\begin{align*}
\bigg|\sum_{k=3}^\infty \frac{(2{\ii}u)^{k}\mu_{k}}
{2k}\bigg|&\le
 2\mu_2|u|^2\sum_{k\ge 0} \frac{(2|u|{\bar\alpha}\sqrt{\mu_2})^{k+1}}{k+1}=-2\mu_2|u|^2\log(1-2|u|{\bar\alpha}\sqrt{\mu_2}),
\end{align*}
for every $u$ satisfying $|u|<(2\bar\alpha\sqrt{\mu_2})^{-1}$. This leads to the inequality
\begin{align}\label{phi2}
|\log\varphi_{\hat\theta-\theta}(v/\sqrt{2\mu_2})+\frac{v^2}{2}|\le -v^2\log(1-\sqrt2|v|{\bar\alpha}),
\end{align}
for every $|v|<(\sqrt2\bar\alpha)^{-1}$. As a first application of our approach,
we obtain a central limit theorem for $\hat{\theta}_n$.

\begin{proposition}\label{WCONV}
Suppose that the matrices $A=A_n$ and $\Sigma=\Sigma_n$ as
well as the number $N=N_n$ depend on $n\in\bbN$. If
$\lambda_{1,n},\ldots,\lambda_{N,n}$, the eigenvalues of
$\Sigma_n^{1/2}A_n\Sigma_{n}^{1/2}$, satisfy
$\lim_{n\to\infty} \|\llambda_{n}\|_\infty^2/\mu_{2,n}=0$,
then
$$
\frac{\hat\theta_n-\theta_n}
{\sqrt{2\mu_{2,n}}}
\xrightarrow[n\to\infty]{D}
\mathcal N(0,1),
$$
where $\hat\theta_n=\xxi^\T A_n\xxi$, $\theta_n=\Ex[\hat\theta_n]
=\tr[\Sigma_n A_n]$, $\mu_{2,n}=\tr[(\Sigma_n A_n)^2]$
and $\xrightarrow{D}$ stands for the convergence in distribution.
\end{proposition}

\begin{proof}
Set $\mu_{k,n}=\tr[(\Sigma_n A_n)^k]=\sum_\ell \lambda_{\ell,n}^k$
and  $\eta_n=(\hat\theta_n-\theta_n)/\sqrt{2\mu_{2,n}}$. The inequality
(\ref{phi2}) and the condition $\lim_{n\to\infty}\|\llambda_n\|_\infty^2/\mu_{2,n}=0$ imply that
the characteristic function of $\eta_n$ converges pointwise to the characteristic function of a standard Gaussian
distribution. This completes the proof of the proposition.
\end{proof}

This result states that the distribution of the estimator $\hat\theta_n$ is well approximated by
a Gaussian distribution.  In order to give a more precise sense to this approximation and to obtain more
accurate approximations, we focus our attention on a second-order asymptotic expansion of the distribution
of $\hat\theta_n$. To this end, we prove first that the tails of this distribution are sufficiently
small.

\begin{lemma}\label{lemm3}
If for some $p\in\bbN$ the inequality $\|\llambda\|_\infty^2\le \mu_2/(2p)$ holds, then
for every $j\in\bbN$
$$
\bigg|\frac{d^j}{du^j}\,\Ex[e^{{\ii}u(\hat\theta-\theta)}]\bigg|\le j! (2N\|\llambda\|_\infty+|\theta|)^j (p/2)^{p/4}(1+\mu_2u^2)^{-p/4},\qquad\forall\,u\in\bbR.
$$
\end{lemma}

\begin{proof}
Thanks to the fact that $\zeta_\ell^2$ is distributed according to the $\chi_1^2$ distribution, one easily checks that
$\big|\varphi_{\hat\theta}(u)\big|=\big|\prod_{\ell=1}^N(1-2{\ii}u \lambda_{\ell})^{-1/2}\big|
=\prod_{\ell=1}^{N} (1+4u^2\lambda_{\ell}^2)^{-1/4}$. 
In view of the assumptions of the lemma, for every $i=1,\ldots,p$, there exists an integer $\ell_i$ verifying
$\mu_2^{-1}\sum_{\ell=1}^{\ell_i} \lambda_{\ell}^2<i/p$ and $\mu_2^{-1}\sum_{\ell=1}^{\ell_{i}+1} \lambda_{\ell}^2\ge i/p$.
For this sequence $\ell_i$, we get $\mu_2^{-1}\sum_{\ell=\ell_i+1}^{\ell_{i+1}} \lambda_{\ell}^2\ge (i+1)/p-1/(2p)-i/p=1/(2p)$
and therefore
\begin{align}\label{j=0}
\prod_{\ell=1}^{N} (1+4u^2\lambda_{\ell}^2)^{-1/4}&\le
\prod_{i=1}^{p} \Big(1+4u^2\sum_{\ell=\ell_i+1}^{\ell_{i+1}}\lambda_{\ell}^2\Big)^{-1/4}\le (p/2)^{p/4}(1+\mu_2u^2)^{-p/4}.
\end{align}
This gives the desired estimate in the case where $j=0$.

For $j> 0$, the explicit form of $\varphi_{\hat\theta}$ allows one to check that
$$
\varphi_{\hat\theta}^{(j)}(u)=\sum_{j_1+\ldots+j_N =j} \frac{j!}{j_1!\ldots j_N!} \prod_{\ell=1}^N\frac{d^{j_\ell}}{du^{j_\ell}}(1-2{\ii}u\lambda_{\ell})^{-1/2}.
$$
Simple computations yield
\begin{align*}
\Big|\frac{d^{j_\ell}}{du^{j_\ell}}(1-2{\ii}u\lambda_{\ell})^{-1/2}\Big|&\le
\Big|\frac{j_\ell!\;(2\ii\lambda_{\ell})^{j_\ell}}{(1-2{\ii}u\lambda_{\ell})^{j_\ell+1/2}}\Big|
\le \frac{j_\ell!\|2\lambda\|_\infty^{j_\ell}}{(1+4u^2\lambda_{\ell}^2)^{1/4}}.
\end{align*}
Therefore,
$$
\Big|\frac{d^j}{du^j}\,\varphi_{\hat\theta}(u)\Big|\le j!  (2N\|\llambda\|_\infty)^{j}\prod_{\ell=1}^N (1+4u^2\lambda_{\ell}^2)^{-1/4}
$$
and the desired inequality for $\theta=0$ follows from (\ref{j=0}). For $\theta$ different from zero, it suffices to use
the relation $|\varphi_{\hat\theta-\theta}^{(j)}(u)|\le\sum_{k=0}^j C_j^k |\ii \theta|^k|\varphi_{\hat\theta}^{(j-k)}(u)|$
and the obtained estimate for $|\varphi_{\hat\theta}^{(j-k)}(u)|$.
\end{proof}

\begin{remark}
We will use the result of Lemma~\ref{lemm3} in the asymptotic setup described in Proposition~\ref{WCONV}, essentially for bounding
the tails of the derivatives of the characteristic function $\varphi_{\hat\theta-\theta}(u)$ of $\hat\theta-\theta$, when the absolute value of $u$ is larger than $N^{q_0}/\sqrt{\mu_2}$ for some $q_0>0$. As we see later, in the asymptotic setup, the
ratio $\|\llambda\|_\infty^2/\mu_2$ tends to zero under mild assumptions on the sampling schemes. This will allow us to take the parameter $p$ of Lemma~\ref{lemm3} large enough to guarantee suitable 
decay properties for the tails of the derivatives of $\varphi_{\hat\theta-\theta}$.
\end{remark}

\subsubsection{Computation of $\mu_k$ in our setup}\label{SSec3}

We showed in the previous subsection that the asymptotic expansion of
the characteristic function of $\hat\theta$ involves the
traces of integer powers of the matrix $\Sigma\cdot A$. In our setup, both
matrices $A$ and $\Sigma$ have special forms. In particular, they contain
only a small number of nonzero entries and, therefore, the expression of $\mu_k$
takes a simplified form.

Prior to presenting the formula for $\mu_k$, we need a definition.
Let $k>0$ be an integer.

\begin{definition}
We call chain of length $k$, any vector
$(\boldsymbol{i},\boldsymbol{j})\in \{1,\ldots,N_1\}^k
\times\{1,\ldots,N_2\}^k$ such that
$I^{i_p}\cap J^{j_p}\not=\emptyset$ and
$J^{j_p}\cap I^{i_{p+1}}\not=\emptyset$ for all $p\in\{1,\ldots,k\}$
with the convention $i_{k+1}=i_1$.
The set of all chains of length $k$ will be denoted by  $\mathscr C_k$.
\end{definition}

In the definition of $\mathscr C_k$, $i_p$ (resp.\ $j_p$)
stands for the $p$th coordinate of $\boldsymbol{i}$ (resp.\ $\boldsymbol{j}$).

\begin{proposition}\label{propmu}
The coefficients $\mu_2$ and $\mu_3$ can be computed by the formulae
\begin{align*}
\mu_2&=\frac12\sum_{(\boldsymbol i,\boldsymbol j)\in\mathscr C_{2}}
\prod_{p=1}^2 v(I^{i_p}\cap J^{j_p})+
\frac12\sum_{(i,j)\in\mathscr C_{1}}
v_1(I^{i})v_2(J^{j}),\\
\mu_3&=\frac14\sum_{(\boldsymbol i,\boldsymbol j)\in\mathscr C_{3}}
\prod_{p=1}^3 v(I^{i_p}\cap J^{j_p})+
\frac34\sum_{(\boldsymbol i,\boldsymbol j)\in\mathscr C_{2}}
v_1(I^{i_1})v_2(J^{j_1})v(I^{i_2}\cap J^{j_2}),
\end{align*}
where $v,v_1$ and $v_2$ are defined by (\ref{vS}).
\end{proposition}
\begin{proof}We give only the proof of the second formula. The proof of the
first  formula is analogous but simpler, therefore it is omitted.
Since $\mu_3=\tr[(\Sigma\cdot A)^3]$, we have
\begin{align}\label{mu3}
\mu_3=\sum_{\ell_1,\ldots,\ell_6=1}^N \sigma_{\ell_1\ell_2}
a_{\ell_2\ell_3}\sigma_{\ell_3\ell_4}a_{\ell_4\ell_5}
\sigma_{\ell_5\ell_6}a_{\ell_6\ell_1}.
\end{align}
In our setup, the entries of the matrix $A$ are
\begin{align}
a_{\ell,\ell'}&=
\frac12\cdot\1(\ell\le N_1,\ell'>N_1, I^\ell\cap J^{\ell'-N_1}\not=\emptyset)
+\frac12\cdot\1(\ell>N_1,\ell'\le N_1, I^{\ell'}\cap
J^{\ell-N_1}\not=\emptyset),\label{a}
\end{align}
and those of $\Sigma$ are
\begin{align}
\sigma_{\ell,\ell'}&=
\begin{cases}
v(I^\ell\cap J^{\ell'-N_1}), &\text{if}\ \ell\le N_1,\ell'>N_1,\\
v(I^{\ell'}\cap J^{\ell-N_1}), &\text{if}\ \ell'\le N_1,\ell>N_1,\\
v_1(I^\ell),&\text{if}\ \ell=\ell'\le N_1,\\
v_2(J^{\ell-N_1}),&\text{if}\ \ell=\ell'> N_1,\\
0,&\text{otherwise}.\label{sigma}
\end{cases}
\end{align}
To compute the sum in the right hand side of (\ref{mu3}), we consider
different cases separately.\\
\underline{Case A: $\ell_1\le N_1$} Our aim now is to compute
$$
\mu_{3,A}=
\sum_{\ell_1\le N_1}\sum_{\ell_2,\ldots,\ell_6=1}^N \sigma_{\ell_1\ell_2}
a_{\ell_2\ell_3}\sigma_{\ell_3\ell_4}a_{\ell_4\ell_5}
\sigma_{\ell_5\ell_6}a_{\ell_6\ell_1}.
$$
This can be done by considering the following four subcases:
\medskip

\hglue20pt Case A.1 $\ell_1\not=\ell_2$ and $\ell_3\not=\ell_4$,
\hglue20pt Case A.2 $\ell_1=\ell_2$ and $\ell_3=\ell_4$,

\hglue20pt Case A.3 $\ell_1\not=\ell_2$ and $\ell_3=\ell_4$,
\hglue20pt Case A.4 $\ell_1=\ell_2$ and $\ell_3\not=\ell_4$.

\medskip

In the case A.1, in order that the corresponding term in (\ref{mu3})
be nonzero, the indices $\ell_i,i\le 6$, should satisfy
$\ell_1\le N_1$, $\ell_2>N_1$, $\ell_3\le N_1$, $\ell_4>N_1$, $\ell_5\le
N_1$ and $\ell_6>N_1$. Moreover, if we set $\bds i=(\ell_1,\ell_3,\ell_5)$
and $\bds j=(\ell_2,\ell_4,\ell_6)$, then $(\bds i,\bds j)$ should belong
to $\mathscr C_3$. Therefore,
$\sigma_{i_pj_p}=v(I^{i_p}\cap J^{j_p})$ for $p=1,2,3$ and
\begin{align}\label{T1}
\sigma_{\ell_1\ell_2}
a_{\ell_2\ell_3}\sigma_{\ell_3\ell_4}a_{\ell_4\ell_5}
\sigma_{\ell_5\ell_6}a_{\ell_6\ell_1}=
\frac18\1((\bds i,\bds j)\in\mathscr C_3)\prod_{p=1}^3 v(I^{i_p}\cap J^{j_p}).
\end{align}

In the case A.2, in order to get nonzero term in (\ref{mu3}), the
indices $\ell_i,i\le 6$, should satisfy
$\ell_1=\ell_2\le N_1$, $\ell_3=\ell_4>N_1$, $\ell_5\le N_1$
and $\ell_6>N_1$. Moreover, if we set $\bds i=(\ell_1,\ell_5)$
and $\bds j=(\ell_3,\ell_6)$, then $(\bds i,\bds j)$ should belong
to $\mathscr C_2$. Therefore,
\begin{align}\label{T2}
\sigma_{\ell_1\ell_2}
a_{\ell_2\ell_3}\sigma_{\ell_3\ell_4}a_{\ell_4\ell_5}
\sigma_{\ell_5\ell_6}a_{\ell_6\ell_1}&=
\sigma_{i_1i_1}a_{i_1j_1}\sigma_{j_1j_1}a_{j_1i_2}
\sigma_{i_2j_2}a_{j_2i_1}\nonumber\\
&=\frac18\1((\bds i,\bds j)\in\mathscr C_2)
v_1(I^{i_1})v_2(J^{j_1})v(I^{i_2}\cap J^{j_2}).
\end{align}

In the cases A.3 and A.4, it is easily seen that the corresponding summand
in the right hand side of (\ref{mu3}) is $\not=0$ only if
$\ell_5=\ell_6$. Using the symmetry of $a_{\ell\ell'}$s and
$\sigma_{\ell\ell'}$s, we infer that the results in these cases are
equal and equal to the result of the case A.2.

\underline{Case B: $\ell_1>N_1$} We want to evaluate the term
$$
\mu_{3,B}=
\sum_{\ell_1>N_1}\sum_{\ell_2,\ldots,\ell_6=1}^N \sigma_{\ell_1\ell_2}
a_{\ell_2\ell_3}\sigma_{\ell_3\ell_4}a_{\ell_4\ell_5}
\sigma_{\ell_5\ell_6}a_{\ell_6\ell_1}.
$$
In view of the symmetry of matrices $A$ and $\Sigma$, we can rewrite
$\mu_{3,B}$ in the form
$$
\mu_{3,B}=
\sum_{\ell_1>N_1}\sum_{\ell_2,\ldots,\ell_6=1}^N \sigma_{\ell_6\ell_5}
a_{\ell_5\ell_4}\sigma_{\ell_4\ell_3}a_{\ell_3\ell_2}
\sigma_{\ell_2\ell_1}a_{\ell_1\ell_6}.
$$
Since $a_{\ell_1\ell_6}\not=0$ and $\ell_1>N_1$ entails $\ell_6\le N_1$,
and $a_{\ell_1\ell_6}\not=0$ and $\ell_6\le N_1$ entails $\ell_1> N_1$,
we get
$
\mu_{3,B}=
\sum_{\ell_6\le N_1}\sum_{\ell_1,\ell_2,\ldots,\ell_5=1}^N
\sigma_{\ell_6\ell_5}a_{\ell_5\ell_4}\sigma_{\ell_4\ell_3}a_{\ell_3\ell_2}
\sigma_{\ell_2\ell_1}a_{\ell_1\ell_6}.
$
By reordering the indices we get $\mu_{3,B}=\mu_{3,A}$ and the
assertion of the proposition follows.
\end{proof}
\begin{corollary}
The terms $\mu_2$ and $\mu_3$ may alternatively be computed by formulae (\ref{mu221})-(\ref{mu321}).
\end{corollary}

\begin{proof}
Let us prove the second equality.
Let us denote by $T_1$ and $T_2$ respectively the first and the second sums
in the expression of $\mu_3$ given in Proposition~\ref{propmu}.
In this notation, $4\mu_3=T_1+3T_2$.

On the one hand,
$(\bds i,\bds j)\in\mathscr C_2$ implies that both $I^{i_1}$ and $I^{i_2}$
have non-empty intersections with each of $J^{j_1}$ and $J^{j_2}$. This
obviously implies that $i_1=i_2$ or $j_1=j_2$. Therefore,
\begin{align*}
T_2&=\sum_{(\boldsymbol i,\boldsymbol j)\in\mathscr C_{2}}
v_1(I^{i_1})v_2(J^{j_1})v(I^{i_2}\cap J^{j_2})\\
&=\sum_{I\!,J,J'} v_1(I)v_2(J)v(I\cap J')\KIJ
+\sum_{I,I'\!\!,J} v_1(I)v_2(J)v(I'\cap J)\KIJ 
-\sum_{I\!,J} v_1(I)v_2(J)v(I\cap J),
\end{align*}
the last term resulting from the fact that the terms with $i_1=i_2$
and $j_1=j_2$ are present both in the first and in the second sums
of the right hand side. Since the set of intervals $\Pi^2=\{J^j\}$
forms a partition of $[0,T]$, we have $\sum_{J'} v(I\cap J')=v(I)$.
Similarly, $\sum_{I'} v(I'\cap J)=v(J)$. Therefore
\begin{equation}\label{T2'}
T_2=\sum_{I\!,J} v_1(I)v_2(J)[(v(I)+v(J))\KIJ
-v(I\cap J)]=v(I\cup J)\KIJ.
\end{equation}
To compute the term $T_1$, we decompose the sum
$\sum_{(\bds i,\bds j)\in\mathscr C_3}$ into the sum of
three terms
$$
T_{1q}=\sum_{\substack{(\bds i,\bds j)\in\mathscr C_3\\
\#\{j_1,j_2,j_3\}=q}}
\prod_{p=1}^3 v(I^{i_p}\cap J^{j_p}),\ q=1,2,3.
$$
If $q=1$, then $J^{j_1}=J^{j_2}=J^{j_3}:=J$ and
using the same arguments as for evaluating $T_2$, we get
$T_{11}=\sum_{J} v(J)^3$. If $q=2$, then $j_1=j_2\not=j_3$
or $j_1=j_3\not=j_2$ or $j_1\not=j_2=j_3$. Because of the symmetry,
it suffices to consider one of these cases. Let $j_1=j_2\not=j_3$
and set $J=J^{j_1}$ and $J'=J^{j_3}$. The relations
$(\bds i,\bds j)\in\mathcal C_3$ implies that both $J$ and $J'$ have
non-empty intersections with both $I^{i_1}$ and $I^{i_3}$. Therefore,
$I^{i_1}=I^{i_3}:=I$ and setting $I^{i_2}=I'$ we get
\begin{align*}
T_{12}&=3\!\!\sum_{J\not=J',I,I'}
v(I\cap J)v(I'\cap J)v(I\cap J')
=3\!\!\sum_{J\not=J',I}v(I\cap J)v(J)v(I\cap J')\\
&=3\sum_{I\!,J}v(I\cap J)v(J)[v(I)-v(I\cap J)].
\end{align*}
In the case when all indices $j_1,j_2$ and $j_3$ are different,
it is easily seen that $(\bds i,\bds j)\in\mathscr C_3$ entails
$i_1=i_2=i_3$. Therefore,
\begin{align*}
T_{13}&=\sum_{\substack{I\!,J,J',J''\\\#\{J,J'J''\}=3}}
v(I\cap J)v(I\cap J')v(I\cap J'')\\
&=\sum_{I\!,J,J',J\not=J'}
v(I\cap J)v(I\cap J')[v(I)-v(I\cap J')-v(I\cap J)]\\
&=\sum_{I\!,J,J',J\not=J'}
v(I\cap J)v(I\cap J')v(I)-2\sum_{I\!,J\not=J'}v(I\cap J)^2v(I\cap J').
\end{align*}
Using the identity $\sum_{J':J\not=J'} v(I\cap J')=v(I)-v(I\cap J)$ we
get $T_{13}=\sum_I v(I)^3-\sum_{I\!,J}v(I\cap J)^2[3v(I)-2v(I\cap J)]$.
Summing up the terms $T_{11},T_{12},T_{1,3}$
and $T_2$ we get equality (\ref{mu321}).
Equality (\ref{mu221}) can be proved along the same lines.
\end{proof}

\begin{remark}
If the observations are synchronous, that is $\Pi^1=\Pi^2=\Pi$, then
$\mu_2$ and $\mu_3$ have the following simple expressions:
\begin{align*}
2\mu_2&=\sum_{I\in\Pi}\ [v(I)^2+v_1(I)v_2(I)],\qquad
4\mu_3=\sum_{I\in\Pi}\ [v(I)^3+3v_1(I)v_2(I)v(I)].
\end{align*}
\end{remark}

\begin{lemma}\label{lambdasup}
Assume that we are given two sequences of partitions
$\Pi^1_n=\{I_{n}^i,i\le N_{1,n}\}$ and $\Pi^2_n=\{J_{n}^j,j\le N_{2,n}\}$
of the interval $[0,T]$. Define the matrices $A_n$ and $\Sigma_n$
by (\ref{a}) and (\ref{sigma}). If the functions $\sigma_1$ and
$\sigma_2$ are bounded on $[0,T]$ by some constant $\sigma$, then
$$
\max_\ell \lambda_{\ell,n}^2=\|(\Sigma_n^{1/2}A_n\Sigma_n^{1/2})^2\|
\le 3\sigma^4 r_n^2,
$$
where $r_n=[(\max_{i} |I_{n}^i|)\vee(\max_j |J_{n}^j|)]$.
\end{lemma}

\begin{proof}
Let us define a new partition $\tilde\Pi^1_n$ as follows:
$I\in\tilde\Pi^1_n$ if and only if either $I\in\Pi^1_n$ and it has
non-empty intersection with two distinct intervals from $\Pi^2_n$ or
there is $J\in\Pi^2_n$ such that $I$ is the union of all intervals
from $\Pi^1_n$ included in $J$. The partition $\tilde\Pi^2_n$ is
defined analogously. It is easy to check that the estimator
$\hat\theta_n$ based on $(\tilde\Pi^1_n,\tilde\Pi^2_n)$ is equal to
the one based on $(\Pi^1_n,\Pi^2_n)$. It follows that
$\mu_{p,n}=\tilde\mu_{p,n}$ for every $p\in\bbN$. Therefore, the
relation $\max_{\ell} \lambda_{\ell,n}^2=\lim_{p\to\infty}
\mu_{2p,n}^{1/p}$ implies that $\max_{\ell}
\lambda_{\ell,n}^2=\max_{\ell}\tilde\lambda_{\ell,n}^2$. It is clear
that $r_n=\tilde r_n$, but the advantage of working with
$(\tilde\Pi^1_n,\tilde\Pi^2_n)$ is that
\begin{align}\label{20}
\max_{J\in\tilde\Pi^2}\sum_{I\in\tilde\Pi^1} \KIJ \le 3,\quad
\max_{I\in\tilde\Pi^1}\sum_{J\in\tilde\Pi^2} \KIJ \le 3.
\end{align}
In the remaining of this proof, without loss of generality we assume
that (\ref{20}) is fulfilled for partitions $(\Pi^1,\Pi^2)$. The
estimate $\|(\Sigma_n^{1/2}A_n\Sigma_n^{1/2})^2\|\le
\|\Sigma_n\|^2\|A_n\|^2$ implies that it suffices to estimate
$\|A_n\|$ and $\|\Sigma_n\|$. To bound from above $\|A_n\|^2$, we
use  $\|A_n\|^2=\max_{u:|u|=1} |A_nu|^2$ and
\begin{align*}
|A_nu|^2&=\frac14
\sum_{i}\bigg(\sum_{j} K_{I^iJ^j}u_{N_1+j}\bigg)^2+\frac14
\sum_{j}\bigg(\sum_{i} K_{I^iJ^j}u_{i}\bigg)^2.
\end{align*}
Applying the Cauchy-Schwarz inequality and changing the order of
summation, we get the inequalities $\sum_{i}\big(\sum_{j}
K_{I^iJ^j}u_{N_1+j}\big)^2\le \frac{3}4 \sum_{j} u_{N_1+j}^2$ and
$\sum_{j}\big(\sum_{i} K_{I^iJ^j}u_{i}\big)^2\le \frac{3}4\sum_{i}
u_{i}^2$, which imply that $\|A_n\|^2\le 3/4$.

On the other hand,
\begin{align*}
\|\Sigma_n\|
&=\max_{u:|u|=1}\sum_{\ell,\ell'=1}^N \sigma_{\ell,\ell'} u_\ell u_{\ell'}
=\max_{u:|u|=1} \Big(\sum_{\ell=1}^{N} \sigma_{\ell,\ell} u_\ell^2+
2\sum_{i,j}v(I_{n}^i\cap J_{n}^j) u_iu_{N_1+j}\Big).
\end{align*}
Since $\sigma_{\ell,\ell'}$s are given by (\ref{sigma}),
the first sum in the right hand side is
bounded by $\sigma^2(\max_{i} |I_{n}^i|)\vee(\max_j |J_{n}^j|)$,
whereas the second sum can be bounded using the inequality relating the
geometrical and the arithmetical means :
\begin{align*}
2\sum_{i,j}v(I_{n}^i\cap J_{n}^j) u_iu_{N_1+j}&\le
\sum_{i,j}v(I_{n}^i\cap J_{n}^j) u_i^2+
\sum_{i,j}v(I_{n}^i\cap J_{n}^j) u_{N_1+j}^2\\
&=\sum_i v(I_{n}^i) u_i^2+\sum_j v(J_{n}^j) u_{N_1+j}^2\\
&\le |u|^2\sigma^2 (\max_{i} |I_{n}^i|)\vee(\max_j |J_{n}^j|).
\end{align*}
This completes the proof of the lemma.
\end{proof}

As a by-product of the preceding lemma, we give below a simple sufficient
condition for the asymptotic normality of $\hat\theta_n$.

\begin{corollary}\label{cor2}
In the notation of Lemma~\ref{lambdasup}, if
\begin{align}\label{IJ}
\lim_{n\to\infty} \frac{r_n^2} {\mu_{2,n}}=0,
\end{align}
then $(\hat\theta_n-\theta)/\sqrt{2\mu_{2,n}}$ converges in distribution
to a standard Gaussian random variable.
\end{corollary}

\begin{proof}
According to Proposition \ref{WCONV}, it is enough to show that
$
\lim_{n\to \infty}\frac{\|(\Sigma_n^{1/2}A_n\Sigma_n^{1/2})^2\|}{
\tr[(\Sigma_nA_n)^2)]}=0.
$
This convergence follows from assumption (\ref{IJ}) and
Lemma~\ref{lambdasup}.
\end{proof}

\subsubsection{Expansion of the characteristic function for random sampling schemes}\label{SSec4}
We assume now that the partitions $\Pi^1_n$ and $\Pi^{2}_n$
are random and independent of 
$\{X_{1,t}-X_{1,0},X_{2,t}-X_{2,0}\}_{t\in[0,T]}$.
We denote by $\Ex^{\Pi}$ the conditional
expectation given $\Pi_n$, where $\Pi_n=(\Pi^1_n,\Pi^2_n)$.
Since in this setup the quantities $r_n$ and $\mu_{2,n}$ ---
introduced in Lemma~\ref{lambdasup} and in Proposition~ \ref{WCONV}, respectively ---
are random, Corollary~\ref{cor2} can not be applied directly.
The following result gives a sufficient condition for the
convergence in distribution of Corollary~\ref{cor2} to hold
in the setup of random sampling scheme.

\begin{proposition}\label{prop3}
Let $r_n$ be defined as in Lemma~\ref{lambdasup}. If
$r_n^2/\mu_{2,n}$ tends to zero in probability as $n\to\infty$, then
$(\hat\theta_n-\theta_n)/\sqrt{2\mu_{2,n}}$ converges in
distribution to a standard normal random variable.
If moreover, $2\mu_{2,n}/b_n\xrightarrow[n\to\infty]{P} \mfc$ for some
deterministic sequence $\{b_n\}$ and some
positive constant $\mfc$, then
$(\hat\theta_n-\theta)/\sqrt{b_n}\xrightarrow[n\to\infty]{D}
\mathcal \mathcal N(0,\mfc)$.
\end{proposition}

\begin{proof} Denote $\sigma[\Pi]=\sigma[\Pi_n,\ n\in\bbN]$. 
Our aim is to show that for every $u\in\bbR$ the convergence
$
\E\Big[\exp\Big({\ii}u(\hat\theta_n-\theta_n)/\sqrt{2\mu_{2,n}}\Big)\Big]
\xrightarrow[n\to\infty]{} \E[e^{-\half u^2}]
$
holds. Let us denote $a_n=\E\Big[\exp\Big({\ii}u(\hat\theta_n-\theta_n)/\sqrt{2\mu_{2,n}}\Big)\Big]$ 
and $a=\E[e^{-\half u^2}]$. To show the desired convergence, it suffices to check that every convergent
subsequence of $\{a_n\}$ converges to $a$. For checking this property, one can simply remark that for
any subsequence $\{a_{n_k}\}$, there is a sub-subsequence $\{n_{k_j}\}$ such that $r_{n_{k_j}}^2/\mu_{2,n_{k_j}}$ 
converges almost surely. Then, Corollary~\ref{cor2} implies that $a_{n_{k_j}}$ converges to $a$ as $j\to\infty$. 
Therefore, $a$ is also the limit of the sequence $\{a_{n_k}\}$ and 
the first assertion of the proposition follows. The second assertion follows from the first one by a simple application
of the Slutsky lemma.
\end{proof}


{From now on}, we assume that the assumptions of Proposition~\ref{prop3} are fulfilled and
aim at finding the asymptotic expansion of the distribution of the random variable $\calx_n=(\hat\theta_n-\theta)/\sqrt{b_n}$ as $n\to\infty$. The first step in deriving the asymptotic
expansion of a distribution is the expansion of the characteristic function. As usual, the desired expansion
involves the $r$-th conditional cumulant of  $\calx_n$ given $\Pi$, henceforth denoted by $\kappa_r[\calx_n]$.
Let $\bar{\lambda}_{r,n}$ be the normalized $r$-th  conditional cumulant of $\calx_n$:
\beas
\bar{\lambda}_{r,n}= b_n^{-\frac{r-2}{2}}\kappa_r[\calx_n]
=2^{r-1}(r-1)! \>b_n^{-r+1}\mu_{r,n}.
\eeas
Note that this notation is consistent with those introduced in (\ref{bar_lambda}).
\begin{lemma}\label{071217-1}
For every positive integer $r$, we have
\bea\label{boundmuk}
|\mu_{r,n}|\le \sum_\ell |\lambda_{\ell,n}|^r\leq \max_\ell
|\lambda_\ell|^{r-2} \mu_{2,n}\le
(\alpha_n\sqrt{b_n})^{r-2}\mu_{2,n},
\eea
where $\alpha_n=\sqrt{3}\sigma^2 r_nb_n^{-1/2}$. In terms of 
the conditional cumulants, this is equivalent to 
$
|\kappa_r[\calx_n]|\leq
c_r
\alpha_n^{r-2}\>\bar{\lambda}_{2,n}$,
where $c_r=2^{r-2}(r-1)!$.
\end{lemma}
\proof
This is an immediate consequence of Lemma \ref{lambdasup}. \qed

\begin{proposition}\label{prop4} Let the sequence $\{b_n\}$
be as in Proposition~\ref{prop3}. For some fixed $\mfc_1>0$, let
\begin{align*}
\calp_n(\delta)&=\big\{\Pi\,:\,
\alpha_n<\delta, \>\bar{\lambda}_{2,n}<\mfc_1\big\},\qquad \forall\delta>0.
\end{align*}
Then, for every $j\in\bbZ_+$, there exist some positive constants $C$ and $q$ such that
\begin{align*}
\frac{d^j}{du^j}\,\Big(\Ex^{\Pi}[e^{{\ii}u\calx_n}]\Big)&=
\frac{d^j}{du^j}\,\Big\{e^{-\frac{\bar{\lambda}_{2,n}u^2}2}
\bigg(1+\frac{({\ii}u)^3b_n^{1/2}}{6}\bar{\lambda}_{3,n} \Big)\Big\}+\mO(\delta^2)
(1+|u|^q)e^{-\frac{\bar{\lambda}_{2,n}}{2}u^2}
\end{align*}
for every $u$ satisfying $|u|\le C\delta^{-1/3}$ and for every $\Pi_n\in\calp_n(\delta)$.
In this formula,  $\mO(\delta^p)$ stands for a random variable depending only on partitions
$\Pi_n=(\Pi_n^{1},\Pi_n^{2})$ and satisfying the condition
$\limsup_{\delta\to 0}\sup_{n}\sup_{\Pi_n\in\calp_n(\delta)}
|\mO(\delta^p)|\delta^{-p}<\infty$.
\end{proposition}

\begin{proof}
Let us define $a_0(u)=-{\bar\lambda_{2,n}}u^2/2$,
$a_{1,n}(u)=\frac{({\ii}u)^3b_n^{1/2}\bar\lambda_{3,n}}{6}$
and $\sfr_n(u)=\sum_{k=4}^\infty \frac{(2{\ii}u)^k\mu_{k,n}}{2kb_n^{k/2}}$. 
Using (\ref{logfi}) and the fact that in our setup $\max_\ell
|\lambda_\ell|$ is bounded by $\sqrt{3}\sigma^2r_n$, we get
\begin{align*}
\Ex^\Pi[e^{{\ii}u\calx_n}]
&=\exp\bigg\{\sum_{k=2}^\infty \frac{(2{\ii}u)^k\mu_{k,n}}{2kb_n^{k/2}}\bigg\}=
\exp\bigg\{-\frac{\bar\lambda_{2,n}u^2}2+\frac{({\ii}u)^3b_n^{1/2}\bar\lambda_{3,n}}{6}+\sfr_n(u)\bigg\}
\end{align*}
for every $u\in\bbR$ such that $|u|< 1/(2\delta)$.  One easily checks that
\begin{align}\label{a0a1}
\Ex^\Pi[e^{{\ii}u\calx_n}]
-e^{a_0(u)}(1+a_{1,n}(u))&=e^{a_0(u)}(a_{1,n}(u)+\sfr_n(u))^2\int_0^1\int_0^1
ve^{tv(a_{1,n}(u)+\sfr_n(u))}\,dtdv\nonumber\\
&\qquad+\sfr_n(u)e^{a_0(u)}.
\end{align}
Inequalities (\ref{boundmuk})
imply that there exists some constant $C>0$ such that for every $\ell\le j$ and for every $\Pi_n\in
\calp_n(\delta)$, it holds that
$$
\bigg|\frac{d^\ell \sfr_n(u)}{du^\ell}\bigg|\le
C\frac{(1+u^{4})\alpha_n^2\mu_{2,n}}{b_n}\le C_1(1+u^4)\delta^2,
$$
as soon as $|u|\le 1/(4\alpha_n)$. Similarly, for every $\ell\in\mathbb N$,
$$
\Big|\frac{d^\ell}{du^\ell}\, a_{1,n}(u)\Big|\le C_2(1+|u|^3)\frac{\alpha_n\mu_{2,n}}
{3b_n}\le C_2(1+|u|^3)\delta,
\quad \text{if}\ \Pi_n\in
\calp_n.
$$
These inequalities in conjunction with Eq.\ (\ref{a0a1}) yield the estimate
\begin{align*}
\frac{d^j}{du^j}\Big(\Ex^\Pi[e^{{\ii}u\calx_n}]-e^{a_0(u)}(1+a_{1,n}(u))\Big)=
\mO(\delta^2)(1+|u|^q)e^{a_0(u)}.
\end{align*}
This completes the proof of the proposition
\end{proof}
\begin{remark} As usual in asymptotic expansions, the coefficient of the second order term (i.e., the coefficient of $({\ii}u)^3b_n^{1/2}$)
in the obtained decomposition is given by the normalized third cumulant divided by $6$. It also admits the following representations:
\begin{align*} \frac{b_n^{1/2}}{6}\bar{\lambda}_{3,n}
=\frac{1}{6}\kappa_3[\calx_n]
=\frac{4\mu_{3,n}}{3b_n\sqrt{b_n}}
\end{align*}
where $\mu_{3,n}$ is defined by (\ref{mu321}).
\end{remark}
}

\subsection{Proof of Theorem \ref{191214-1}} \label{SSec5}
Let us start by proving relation (\ref{thm1a}).
Let $\sfh(x)=1+|x|^{\gamma}$.
Let $\calk$ be a probability density on $\bbR$ such that the Fourier transform $\hat{\calk}$ of $\calk$ is compactly supported, $\int_\bbR |x|^{\gamma+2}\calk(x)\,dx<\infty$  and 
$\int_{-1}^1\calk(x)\,dx\ge 2/3$. Let $K>0$.
For $\ep>0$, define the measure $\calk_\ep$ by $\calk_\ep(x)=\calk(\ep^{-1}x)$ for all $x\in\bbR$. 
Using the modified version of the Sweeting lemma~\cite{Swe77} stated in Babu and Singh \cite[Lemma 1]{Babu}, we get:
\bea\label{191113-1}
|\Ex^\Pi[f(\calx_n)]-\Psi^\Pi_n[f]|
&\leq&
9^\gamma M (P^{\calx_n|\Pi}+|\Psi^\Pi_n|)[\sfh]\,\big( \bbA_0+\bbA_1+\bbA_2\big)
+\bbA_3,
\eea
where 
\begin{align*}
\bbA_0 &= \int_{\bbR}\sfh(x)\big|\calk_{b_n^{K}}*(P^{\calx_n|\Pi}-\Psi^\Pi_n)\big|(dx),\quad
\bbA_1 = b_n^{K} \int_{\bbR} |x|^{\gamma+2}\calk(x)\,dx,\quad\bbA_2 = 2^{1-b_n^{-K/4}}\\
\bbA_3 &=\sup_{|x|\le b_n^{K}}\int_\bbR \omega_f(x-y,2b_n^{K})|\Psi^\Pi_n|(dy).
\end{align*}
As we already mentioned, the Rosenthal inequality
yields that $P^{\calx_n|\Pi}[\sfh]=1+\Ex^\Pi[|\calx_n|^{\gamma}]$ is bounded uniformly in $n$.
Furthermore, it is obvious that the term $|\Psi_n^\Pi|[\sfh]$ is bounded uniformly in $n$.

If $n$ is sufficiently large, $[x-y-2b_n^{K},x-y+2b_n^{K}]\subset [-y-3b_n^{K/4},-y+3b_n^{K/4}]$ and therefore
\begin{align*}
\bbA_3 &\leq 2\int_\bbR\omega_f(-y,3b_n^{K/4})|\Psi^\Pi_n|(dy)
\le{\sf C}^\circ\int_\bbR\omega_f(y,3b_n^{K/4})\phi(y;\mfc^*)\,dy\le {\sf C}b_n^{K\eta/4}.
\end{align*}
On the other hand, $\bbA_0$ admits the estimate
\beas
\bbA_0\leq
\sum_{\alpha=0}^{2+\gamma}
\int_\bbR \Bigl| \partial_u^\alpha
\Big[\l(\varphi_{\calx_n}^\Pi(u)-\hat{\Psi}^\Pi_n(u)\r)
\hat{\calk}(b_n^{K}u)\Big]\Bigr|\>du,
\eeas
where
$\varphi_{\calx_n}^\Pi(u)=\Ex^\Pi[e^{{\ii}u\calx_n}]$. Let $\delta_n=b_n^{a-1/2}$.
By virtue of Proposition \ref{prop4}
and Lemma~\ref{lemm3},
we have
\begin{align*}
\int_\bbR \Bigl| \partial_u^\alpha
\Big[\Big(\varphi_{\calx_n}^\Pi(u)&-\hat{\Psi}^\Pi_n(u)\Big)
\hat{\calk}(b_n^{K}u)\Big]\Bigr|\>du\\
&\leq\int_{u:|u|\leq C\delta_n^{-1/3}} \Big| \partial_u^\alpha
\Big[\Big(\varphi_{\calx_n}^\Pi(u)-\hat{\Psi}^\Pi_n(u)\Big)
\hat{\calk}(b_n^{K}u)\Big]\Big|\>du\\
&\qquad+\int_{u:|u|>C\delta_n^{-1/3}}\Bigl| \partial_u^\alpha
\Big[\l(\varphi_{\calx_n}^\Pi(u)-\hat{\Psi}^\Pi_n(u)\r)
\hat{\calk}(b_n^{K}u)\Big]\Bigr|\>du\\
&\leq\int_{u:|u|\leq C\delta_n^{-1/3}}
\mO(\delta_n^2)(1+|u|^q)e^{-{\mu_{2,n}u^2}/2}\,du+
\int_{u:|u|>C\delta_n^{-1/3}} \frac{C_2}{1+|u|^L}\>du\\
&\qquad+\sum_{\alpha'=0}^{2+\gamma}
\int_{u:|u|>C\delta_n^{-1/3}}
|\partial_u^{\alpha'}\hat{\Psi}^\Pi_n(u)|\>du\\
&\leq C_3 [ \mO(\delta_n^2) +\delta_n^{(L-1)/3} ]\leq C_4 \delta_n^2,
\end{align*}
where $L$ can be chosen as large as we need, therefore
$\bbA_0\leq C_5 \delta_n^2$. Combining all these estimates, we get
$$
|\Ex^\Pi[f(\calx_n)]-\Psi^\Pi_n[f]|\le {\sf C}\big(b_n^{2a-1}+b_n^K+2^{-b_n^{-K/4}}+b_n^{K\eta/4}\big).
$$
Choosing $K>\max(2a-1,4(2a-1)/\eta)$, we get the relation stated in (\ref{thm1a}).

To prove (\ref{thm1b}), we notice that $|\bar{\lambda}_{3,n}(\bar{\lambda}_{2,n}-\mfc)|= O(b_n^{-1}r_n\times b_n^{a-1/2})
=O(b_n^{2a-\frac{3}{2}})=o(1)$ uniformly on the event $A_n$.
Expanding $\phi(z;\bar{\lambda}_{2,n})$ in $\Psi^\Pi_n$
around $\mfc$ we get the desired result.

\section{Poisson sampling scheme}\label{Sec4}

As an application of previous results let us consider the case when
the partitions $\Pi^1_n$ and $\Pi^2_n$ are generated by Poisson
point processes. Let $\mP^{i,n}=(\mP^{i,n}_{t},t\ge 0)$, $i=1,2$, be
two independent homogeneous Poisson processes with intensities
$np_i$, $i=1,2$. Moreover, assume that these processes are
independent of $\mB$. Let the sampling times
$S^1,\ldots,S^{N_1}$ and $T^1,\ldots,T^{N_2}$ be the time instants corresponding to the
jumps of $\mP^{1,n}$ and $\mP^{2,n}$ occurred before the instant $T$. Note that
$S^i$s and $T^j$s depend also on $n$. However, for simplicity of exposition
this dependence will not be reflected in our notation.

Prior to stating the main result of this section, let us recall several notation. 
We denote by $h(t)$ the function
$\sigma_{1,t}\sigma_{2,t}\rho_t$ and by $x_+$ the positive part of a real $x$. Finally, we write  $g_1(z)\propto g_2(z)$
if for some $C_g\in\bbR$ the equality $g_1(z)=C_g g_2(z)$ holds for every $z$.

\begin{theorem}\label{expPois}
Let the sampling scheme be generated by two independent Poisson processes with intensities $np_1$ and
$np_2$, independent of the driving BM $\mB$. If the functions $\sigma_{1}$, $\sigma_{2}$ and $\rho$ are Lipschitz
continuous then, for every $a\in(\frac34,1)$, it holds that
\begin{equation}\label{ae4}
\sup_{f\in\cale(M,\gamma)\cap\cale^0({\sf C},\eta,r_0,\mfc^*)} \l| \>\Ex[f(\sqrt{n}(\hat\theta_n-\theta))]-\int_\bbR f(z)\,p_n^\circ(z)\,dz\>\r|=O(n^{1-2a}),
\end{equation}
where
$$
p_n^\circ(z)\propto \frac1{\sqrt{2\pi \mfc}}\Bigl[1
+\frac{2\kappa(z^3-3\mfc z)}{\sqrt{n}\,\mfc^3} \> \Bigr]_+e^{-z^2/(2\mfc)}
$$ is a probability density
with
\begin{align*}
\mfc&=\bigg(\frac2{p_1}+\frac{2}{p_2}\bigg)\int_0^T\sigma_{1,t}^2\sigma_{2,t}^2
(1+\rho_t^2)dt-\frac{2}{p_1+p_2}\int_0^T (\sigma_{1,t}\sigma_{2,t}\rho_t)^2dt,\\
\kappa&=\bigg(\frac{1}{p_1^2}+\frac{1}{p_2^2}\bigg) \int_0^T h(t)^3\,dt
+\frac{3p_1^2+2p_1p_2+3p_2^2}
{p_1^2p_2^2} \int_0^T \sigma_{1,t}^2\sigma_{2,t}^2h(t)\, dt.
\end{align*}
\end{theorem}

Before proceeding with the proof of this theorem, let us note that it extends the asymptotic normality result proved 
in Hayashi and Yoshida~\cite{Hay-Yos04}, providing the second-order
term in the asymptotic expansion of the distribution of $\hat\theta_n$. Note however that the price to pay for getting this
expansion is a slightly stronger assumption on the functions
$\sigma_{1}$, $\sigma_{2}$ and $\rho$. Indeed, we assume in Theorem~\ref{expPois} that these functions are Lipschitz, while 
in \cite{Hay-Yos04} only the continuity of these functions was required.

Remark also that the constant of proportionality in the definition of $p_n^\circ$ can be replaced by one. 
Indeed, $p_n^\circ(z)$ is the positive part of the function
\begin{equation}\label{func}
z\mapsto\frac1{\sqrt{2\pi \mfc}}\Bigl[1+\frac{2\kappa(z^3-3\mfc z)}{\sqrt{n}\,\mfc^3} \> \Bigr]e^{-z^2/(2\mfc)},
\end{equation}
whose integral over $\bbR$ is equal to one. Moreover, for some $c>0$, the function (\ref{func}) is positive on 
the interval $[-cn^{1/6},cn^{1/6}]$ and its absolute value is bounded by an exponentially decreasing function 
outside the interval $[-cn^{1/6},cn^{1/6}]$. This implies that the proportionality constant in the definition
of $p_n^\circ$ is $1+O(e^{-n^{1/3}/(4\mfc)})$ and, consequently, its exact value is unimportant.

\begin{proof}[Proof of Theorem~\ref{expPois}]
We want to apply Theorem~\ref{ThmAE3}. To this end, we have to accomplish the following tasks: 
\begin{itemize}
\item[{[T1]}] prove that $\bar\lambda_{2,n}=2n\mu_{2,n}$ is very
close to $\mfc$ in expectation and in probability,
\item[{[T2]}] check that the maximal sampling step $r_n$ is smaller than $b_n^a$ with high probability,
\item[{[T3]}] determine the asymptotic behavior of  $\E[\bar\lambda_{3,n}]$,
\end{itemize}  
with $b_n=1/n$ and some $a<1$. In fact, we will show that any $a<1$ can be used.

Concerning the task [T1], it is proved in \cite{Hay-Yos04} that $2n\mu_{2,n}$
converges in probability to $\mfc$. In the present work, we need
a result providing the rate of convergence of $2n\mu_{2,n}$ to $\mfc$.
It is done in the following

\begin{proposition}\label{propo5}
If the functions $\sigma_{1}$, $\sigma_{2}$ and $\rho$ are Lipschitz
continuous, then there exists a constant $C>2$ depending only on
$p_1$ and $p_2$ such that, for every $x>C\log n$ and for every $n\ge
2$, it holds that
\begin{equation}\label{exp_in_mu}
\Pb\bigg(|2n\mu_{2,n}-\mfc|>\frac{C\log^3 n}{n}+
\frac{x}{\sqrt{n}}\bigg)\le Cn e^{-x/C}.
\end{equation}
Furthermore, $\Ex[2n\mu_{2,n}]=\mfc+\mO(n^{-1}\log^3 n)$ as $n$ goes to infinity.
\end{proposition}

The proof of this proposition is deferred to Section~\ref{App0}.
 
The task [T2], consisting in bounding the probability of the event $r_n>b_n^a=n^{-a}$ is done using the following lemma.
\begin{lemma}\label{lem2}
There exists a constant $C$ depending only on $p_1$
and $p_2$ such that, for every $x>0$, the inequality $\Pb(nr_n> x)\le Cn
e^{-x/C}$ holds.
\end{lemma}
\begin{proof}
We start with bounding $\Pb(\max_{I\in\Pi^1_n} n|I|>x)$. According
to the Markov inequality, for every $u>0$,
$$
\Pb(\max_{I\in\Pi^1_n}n|I|>x)\le e^{-ux}\Ex\Big[\sum_{I\in\Pi^1_n} e^{un|I|}\Big].
$$
The last sum can be bounded by the sum of $N_1$ independent random
variables each of which has the same law as $e^{u\zeta/p_1}$, with
$\zeta$ being exponentially distributed with mean $1$. In view of the Wald
equation, this yields $\Ex\big[\sum_{I\in\Pi^1_n}
e^{un|I|}\big]=np_1T \Ex[e^{u\zeta/p_1}]$. Choosing $u$ smaller than
$p_1$ and repeating the same arguments for $\max_{J\in\Pi^2_n}
n|J|$, we obtain the desired result.
\end{proof}

Replacing $x$ by $n^{\frac12-a}$ in (\ref{exp_in_mu}) and by $n^{1-a}$ in Lemma~\ref{lem2}, 
we obtain that the probability of the event $A_n(a)^c$ is exponentially small as $n\to\infty$.
Therefore, $\P(A_n(a)^c)=o(b_n^p)=o(n^{-p})$ for every $p>0$. One also deduces from Proposition~\ref{propo5}
that $\E[\bar\lambda_{2,n}]-\mfc=o(n^{1-2a})$ as $n\to\infty$. Thus, it remains to accomplish the task [T3], 
which is done using the following 
proposition, the proof of this proposition is deferred to Section~\ref{App0}.
\begin{proposition}\label{prop6}
Under the assumptions of Theorem~\ref{expPois}, it holds that
$\Ex[\mu_{3,n}]=\frac32\kappa n^{-2}+\mO(\frac{\log^3 n}{n^3})$.
\end{proposition}
Combining these results, we get the assertion of Theorem~\ref{expPois}.
\end{proof}

\section{Stochastic decomposition for $\hat\theta_n$ in a model with drift terms}\label{071219-2}

So far we have considered a Gaussian system $(X_{1,t}-X_{1,0},X_{2,t}-X_{2,0})$
as the underlying model and
essentially finite dimensional Gaussian calculus served as a tool.
In this section, we will treat a system that has random drift terms.
It will be seen that the principal part of the estimator is the same
as in the case without drifts. Thus, the contribution of the principal part
to the asymptotic expansion of the estimator has already been assessed in the previous section.

Beyond being a useful tool for deriving asymptotic expansions of the distribution of $\hat\theta_n$,
the stochastic decomposition of the HY-estimator that we obtain below bridges the problem of estimating
the covariance and that of signal detection in Gaussian white noise. The latter problem has been extensively
studied in the statistical literature and we believe that the methodology developed for the problem of signal
detection may be of interest for our problem.

To state the main result of this section, let us recall that we deal with processes $X_1$ and $X_2$ given by
$$
\begin{cases}
dX_{1,t}=\beta_{1,t}\>dt+\sigma_{1,t}\,dB_{1,t}, &t\in[0,T],\\
dX_{2,t}=\beta_{2,t}\>dt+\sigma_{2,t}\,dB_{2,t}, &t\in [0,T],
\end{cases}
$$
where $\beta_{i,t}$ are progressively measurable processes
and assumed to be unknown to the observer. We will assume that these drift processes admit the following
stochastic decompositions:
$$
d\beta_{i,t}=\beta_{i,t}^{[0]}dt+\beta_{i1,t}^{[1]}\,dB_{1,t}+\beta_{i2,t}^{[1]}\,dB_{2,t},\ i=1,2,
$$
where $\beta_i^{[0]}$, $\beta_{ij}^{[1]}$, $i,j=1,2$ are progressively measurable processes with respect to
the filtration $\{\sigma(\mB_s,\,s\le t)\}_{t\in[0,T]}$.

In this section, we will separate the assumptions on the sampling scheme from those on
$\rho$ and on the drifts and volatilities of $X_1$ and $X_2$. For this reason,
let us introduce the following measures on
$([0,T]^2,\mathscr B_{[0,T]^2})$:
\begin{align*}
&\calv^I_{n}(\cdot)=b_n^{-1}|\cdot\cap\, \{\cup_I I\times I\}|,\quad
\calv^J_{n}(\cdot)=b_n^{-1}|\cdot\cap\, \{\cup_J J\times J\}|,\\
&\calv^{I\cap J}_{n}(\cdot)=b_n^{-1}|\cdot\cap\, \{\cup_{I\!,J} (I\cap
J)\times(I\cap J)\}|,\\
&\calv^{I\!,J}_{n}(\cdot)=b_n^{-1}\sum_{I\!,J}K_{I\!J}|\cdot\cap\, (I\times
J)|.
\end{align*}
Note that these measures depend on the sampling schemes and, therefore, they are random
if the sampling schemes are random. Similarly, let
$\calv_n^{I,I'\!\!,J}(\cdot)=b_n^{-2}|\cdot\cap \{\cup_{J} J\times I(J)\times I(J)\}|$,
$\calv_n^{I\!,J,J'}(\cdot)=b_n^{-2}|\cdot\cap \{\cup_{I} I\times J(I)\times J(I)\}|$ and
$\calv_n^{J(I),I(J),J\cap I}(\cdot)=b_n^{-2}|\cdot\cap \{\cup_{I\!,J} J(I)\times I(J)\times J\cap I\}|$
be (random) measures defined on $([0,T]^3,\mathscr B_{[0,T]^3})$.

\begin{description}
\item[\textbf{Assumption \textsf{P1}}] The random measures $\calv^I_{n}$,
$\calv^J_{n}$, $\calv^{I\cap J}_{n}$ and $\calv^{I\!,J}_{n}$ converge weakly
to some deterministic measures $\calv^I$, $\calv^J$, $\calv^{I\cap J}$ and
$\calv^{I\!,J}$ in probability, as $n\to\infty$. These measures are concentrated
on the diagonal $\mathcal D_T^2=\{(s,t)\in[0,T]^2:s=t\}$ and absolutely continuous
w.r.t.\ the Lebesgue measure on the line.

\item[\textbf{Assumption \textsf{P2}}] As $n\to\infty$, the random measures 
$\calv^{I,I'\!\!,J}_{n}$, $\calv^{I\!,J,J'}_{n}$ and  $\calv^{J(I),I(J),I\cap J}_{n}$  
converge weakly to some deterministic measures $\calv^{I,I'\!\!,J}$, $\calv^{I\!,J,J'}$ 
and  $\calv^{J(I),I(J),I\cap J}$ in probability. These measures are concentrated
on the diagonal $\mathcal D_T^3=\{(s,t,u)\in[0,T]^3:s=t=u\}$ and absolutely continuous
w.r.t.\ the Lebesgue measure on the line.
\end{description}

The weak convergence of $\calv^{I}_{n}$ to $\calv^{I}$  in
probability should be understood as follows: for every continuous
function $\varphi:[0,T]^2\to \bbR$, the sequence of random variables
$\int_{[0,T]^2}\varphi\,d\calv^{I}_{n}$ converges in probability to
$\int_{[0,T]^2}\varphi\,d\calv^{I}$ as $n$ tends to infinity.
For the purposes of the present work, it is probably possible to 
slightly relax Assumption \textsf{P2} by replacing the weak convergence
by the tightness condition. However, to avoid additional technicalities 
we assume that the weak convergence of measures stated in Assumption 
\textsf{P2} holds.

Recall that according to our assumptions $\Pi$ is independent of $\mB$, where $\Pi$ is the collection
of random intervals $I^i:=(S^{i-1}\wedge T,S^{i}\wedge T]$, $J^j:=(T^{j-1}\wedge T,T^{j}\wedge T]$ with $i=1,\ldots,N_1$ and
$j=1,\ldots,N_2$.
In what follows, the following notation will be used:
for two functions $f,g:[0,T]\to\bbR$, we denote by $f\cdot g$ the function $t\mapsto\int_0^t f_s\,dg_s$ and we
often write $I$ or $J$ instead of $\1_I$ or $\1_J$. Thus the estimator $\hat\theta_n$ can be rewritten as
\begin{align*}
\hat{\theta}_n=
\sum_{i=1}^{N_1}\sum_{j=1}^{N_2}\>K_{ij} \{I^i\cdot X_1\}_T\times \{J^j\cdot X_2\}_T.
\end{align*}
We want to derive an asymptotic expansion of the distribution of this estimator using a perturbation method 
based on a stochastic expansion of the estimator $\hat\theta_n$ itself. The main term in this stochastic expansion is 
\begin{align*}
M^n_T=b_n^{-1/2}\Big( \sum_{i,j}
K_{ij}\{(I^i\sigma_1)\cdot B_1\}_T\{(J^j\sigma_2)\cdot B_2\}_T
-\theta\Big).
\end{align*}
Note that the asymptotic expansion of the distribution of $M^n_T$ has already been obtained in preceding sections.
In this section, we will need a representation of $M^n$ as a stochastic integral with respect to the BM $(B_1,B_2)$ that
can be written---using the It\^o formula---as follows:
\begin{align}\label{M^n}
M^n& =\bbH^{1,n}\cdot B_1+\bbH^{2,n}\cdot B_2,
\end{align}
where
$\bbH^{1,n}=\sum_{I\!,J}b_n^{-1/2}\KIJ (J\sigma_2\cdot B_2)I\sigma_1$ and 
$\bbH^{2,n}=\sum_{I\!,J}b_n^{-1/2}\KIJ (I\sigma_1\cdot B_1)J\sigma_2$.

\begin{lemma}\label{lem-decomp}
Assume that $\sigma_1,\sigma_2$ and $\rho$ are bounded and $\beta_{ij}^{[\ell-1]}$s are bounded in $L^4$ uniformly
in $[0,T]$ for every $i,j,\ell\in\{1,2\}$. If $r_n^3=o_p(b_n^2)$, then
$$
b_n^{-1/2}(\hat{\theta}_n-\theta)=M^n_T+b_n^{1/2}(N^n_T+A^n_T)+o_p(b_n^{1/2}),
$$
where $dN^n_t= \bbG^{1,n}_tdB_{1,t}+\bbG^{2,n}_tdB_{2,t}$ is a local martingale with
\begin{align*}
\bbG^{1,n}&= b_n^{-1}\sum_{i,j}\> K_{ij}\{((J^j\beta_2)\cdot t)(I^i\sigma_1)\}
+b_n^{-1}\sum_{i,j}\>K_{ij}\{(T^{j}-T^{j-1}\vee\cdot)_+I^i\sigma_1\beta_{2,S^{i-1}}\},\\
\bbG^{2,n}&= b_n^{-1}\sum_{i,j}\> K_{ij}\{((I^i\beta_1)\cdot
t)(J^j\sigma_2)\}+ b_n^{-1}\sum_{i,j}\>K_{ij}
\{(S^{i}-S^{i-1}\vee\cdot)_+J^j\sigma_2\beta_{1,T^{j-1}}\},
\end{align*}
and $A^n$ is a bounded variation process defined by
\begin{align*}
A^n&= b_n^{-1}\sum_{i,j}\> K_{ij}\Bigl\{J^j\{[I^i\sigma_1(\beta^{[1]}_{21}+\beta^{[1]}_{22}\rho)]\cdot s\}
+I^i\{[J^j\sigma_2(\beta^{[1]}_{11}\rho+\beta^{[1]}_{12})]\cdot s \}\Bigr\}\cdot t
\\
&+ b_n^{-1}\sum_{i,j}\>K_{ij}\{(I^i\beta_1)\cdot t\} \times
\{(J^j\beta_2)\cdot t\}.
\end{align*}
\end{lemma}

Lemma~\ref{lem-decomp} provides a stochastic decomposition of the HY-estimator with a RHS
depending on $n$. Under the assumptions P1 and P2 of the convergence of random measures associated to
the sampling scheme, it is possible to obtain a refinement of this result with a RHS depending on $n$
exclusively through $b_n$. To this end, limit theorems for martingales will be used. An
important step for proving limit theorems for martingales is the computation of the limits of their
quadratic variations and covariations, which will be treated below.

\subsection{Convergence of quadratic variations and covariations}
To establish an asymptotic expansion of $b_n^{-1/2}(\hat\theta_n-\theta)$ that is more
explicit than the one given by Lemma~\ref{lem-decomp}, we need to identify the limiting distribution
of the martingale $(B_1,B_2,M^n,N^n)$ as $n$ goes to infinity. The convergence of
the quadratic variation-matrix is a classical tool for proving the convergence of a martingale.
Most results of the present section being quite technical, we postponed their proofs to Section~\ref{App0.5}. 
 
We start with the cross terms $\langle M^n,B_1\rangle$ and $\langle M^n,B_2\rangle$. In view of (\ref{M^n}), for $\nu=1,2$, we have
\begin{align*}
\langle M^n,B_\nu \rangle &=
\bbH^{1,n}\cdot\langle B_1,B_\nu \rangle +\bbH^{2,n}\cdot\langle
B_2,B_\nu \rangle\\
&=
\sum_{I\!,J}b_n^{-1/2}\KIJ \Big[\{(J\sigma_2\cdot
B_2)I\sigma_1\}\cdot\langle B_1,B_\nu \rangle+
\{(I\sigma_1\cdot B_1)J\sigma_2\}\cdot\langle B_2,B_\nu \rangle\Big].
\end{align*}

\begin{lemma}\label{lem5.6} If $\sigma_1$, $\sigma_2$ and $\rho$ are bounded in
$[0,T]$ and $r_n^2=o_p(b_n)$, then
\begin{align*}
\sup_{\nu=1,2}&\bigg|\sum_{I\!,J}b_n^{-1/2}\KIJ \big(\{(J\sigma_2\cdot
B_2)I\sigma_1\}\cdot\langle B_1,B_\nu
\rangle\big)_t\bigg|\xrightarrow[n\to\infty]{P} 0,\\
\sup_{\nu=1,2}&\bigg|\sum_{I\!,J}b_n^{-1/2}\KIJ \big(\{(I\sigma_1\cdot
B_1)J\sigma_2\}\cdot\langle B_2,B_\nu
\rangle\big)_t\bigg|\xrightarrow[n\to\infty]{P} 0,
\end{align*}
for every $t\in[0,T]$. As a consequence, for every $t\in[0,T]$,
$\max_{\nu=1,2}|\langle M^n,B_\nu \rangle_t|$ tends to zero in
probability as $n\to\infty$.
\end{lemma}

We study now the behavior of the quadratic variation
\begin{align}
\langle M^n,M^n \rangle_t &= \sum_{c,d=1}^2
(\bbH^{c,n}\bbH^{d,n})\cdot \langle B_c,B_d \rangle_t
\end{align}
as $n$ tends to infinity.
First, we note that
\begin{align*}
\bbH^{1,n}\bbH^{2,n} &=
\sum_{i,j,i'\!,j'}b_n^{-1}K_{ij}K_{i'j'}(J^j\sigma_2\cdot
B_2)I^i\sigma_1 (I^{i'}\sigma_1\cdot B_1)J^{j'}\sigma_2
\\
&=\sum_{i,j,i'\!,j'}b_n^{-1}K_{ij}K_{i'j'} (J^j\sigma_2\cdot
B_2)J^{j'}\sigma_21_{\{j\leq j'\}} (I^{i'}\sigma_1\cdot
B_1)I^i\sigma_11_{\{i'\leq i\}}
\end{align*}
Denote by $R^n(i,i'\!,j,j')$ the summand on the right-hand side of the
last equation. This term is different from zero only if the
conditions $I^i\cap J^j\not=\emptyset$, $I^i\cap
J^{j'}\not=\emptyset$, $I^{i'}\cap J^{j'}\not=\emptyset$, $j\le j'$
and $i'\le i$ are fulfilled. If $i'<i$, then these conditions are
fulfilled only if $j=j'$. Similarly, the terms with $j<j'$ are
non-zero only if $i=i'$. This leads to
\begin{align*}
\bbH^{1,n}\bbH^{2,n} &= \sum_{i,j,j':\>j\leq
j'}b_n^{-1}K_{ij}K_{ij'} (J^j\sigma_2\cdot B_2)J^{j'}\sigma_2
(I^i\sigma_1\cdot B_1)I\sigma_1
\\&
+\sum_{i'\!,j,i\,:\>i'\leq i}b_n^{-1}K_{ij}K_{i'j} (J^j\sigma_2\cdot
B_2)J^j\sigma_2 (I^{i'}\sigma_1\cdot B_1)I^i\sigma_1\\&
-\sum_{I\!,J}b_n^{-1}K_{ij} (J^j\sigma_2\cdot B_2)J^j\sigma_2
(I^{i}\sigma_1\cdot B_1)I^i\sigma_1.
\end{align*}
Sum them up in $j'$ and in $i$
respectively and use
\begin{align*}
(J^j\sigma_2\cdot B_2)I\sum_{j':\>j\leq j'}K_{ij'}J^{j'} &=
(J^j\sigma_2\cdot B_2)I\1_{[T^{j-1},T]}=(J^j\sigma_2\cdot B_2)I,
\\
(I^{i'}\sigma_1\cdot B_1)J\sum_{i:\>i'\leq i}K_{ij}I^i &=
(I^{i'}\sigma_1\cdot B_1)J\1_{[S^{i-1},T]}=(I^{i'}\sigma_1\cdot
B_1)J.
\end{align*}
to obtain
$\bbH^{1,n}\bbH^{2,n} =
b_n^{-1}\sum_{I\!,J}\sigma_1\sigma_2K_{I\!J}(J\sigma_2\cdot
B_2)(I\sigma_1\cdot B_1)(I+J-IJ).
$ This implies that
\begin{align*}
\bbH^{1,n}\bbH^{2,n}\cdot\langle B_1,B_2 \rangle_t &=
\int_0^t\sigma_{1,s}\sigma_{2,s}\sum_{I\!,J}\tilde{K}^n_{I\!J}(s)(J\sigma_2\cdot
B_2)_s(I\sigma_1\cdot B_1)_s\>d\langle B_1,B_2 \rangle_s,
\end{align*}
where $\tilde{K}^n_{I\!J}(t)=b_n^{-1}\KIJ (I_t+J_t-I_tJ_t)$.

\begin{lemma}\label{lem8} Assume that $r_n^3=o_p(b_n^2)$ and the functions $\sigma_1$, $\sigma_2$ and $\rho$ are continuous.
If Assumption \textsf{P1} is fulfilled then,  for any $t\in[0,T]$,
\begin{align*}
&\int_0^t\bbH^{1,n}_s\bbH^{2,n}_s\,d\langle B_1,B_2 \rangle_s
\xrightarrow[n\to\infty]{P} \frac12\int_0^t
h_s^2\,\{\calv^{I}(ds)+\calv^{J}(ds)-\calv^{I\cap
J}(ds)\},\\
&\int_0^t(\bbH^{1,n}_s)^2\,d\langle B_1,B_1
\rangle_s+\int_0^t(\bbH^{2,n}_s)^2\,d\langle B_2,B_2 \rangle_s
\xrightarrow[n\to\infty]{P} \int_0^t
\sigma_{1,s}^2\sigma_{2,s}^2\,\calv^{I\!,J}(ds)
\end{align*}
and consequently
$$
\langle M^n,M^n\rangle_t\xrightarrow[n\to\infty]{P}\int_0^t
h_s^2\,\{\calv^{I}(ds)+\calv^{J}(ds)-\calv^{I\cap
J}(ds)\}+\int_0^t
\sigma_{1,s}^2\sigma_{2,s}^2\,\calv^{I\!,J}(ds).
$$
\end{lemma}

Using the claims of two last lemmas, one can already derive the asymptotic distribution of 
the martingale $(B_1,B_2,M^n)$ as $n\to\infty$. However, for our purposes,
it is crucial to know the asymptotics of the joint distribution of the triplet $(B_1,B_2,M^n)$ with
the martingale $N^n$. 

\begin{lemma}\label{lem3.8}
If $\sigma_1,\sigma_2$ and $\rho$ are bounded, $\sup_{t\in[0,T]}
\Ex[\beta_{i,t}^2]<\infty$, $i=1,2$ and $r_n^4=o_p(b_n^3)$ as $n\to\infty$, then for any $t\in[0,T]$ the
sequence of random variables $\langle M^n,N^n\rangle_t$ tends
in probability to zero as $n$ tends to infinity.
\end{lemma}

An interesting fact revealed by this lemma is the orthogonality of $M^n$ and $N^n$ in terms
of quadratic covariation. This indicates that the limiting distribution of $(M^n,N^n)$ 
is that of two independent martingales. This statement will be rigorously proved at the end of 
this section. Prior to presenting that proof, we wish to investigate the structure of the
limiting distribution of $N^n$ and how it relates to the BM $\mB$.


\begin{lemma}\label{lem3.9} Assume that $r_n^3=o_p(b_n^2)$ and that $\sup_{t\in[0,T]}\Ex[(\beta_{ij,t}^{[\ell-1]})^2]<\infty$ for every $i,j,\ell\in\{1,2\}$. Then, under Assumption \textsf{P1}, for every fixed $t\in[0,T]$, we have
\begin{align*}
\langle N^n,B_1\rangle_t&\xrightarrow[n\to\infty]{P} \int_0^t (\beta_{2,s}\sigma_{1,s}+\beta_{1,s}\sigma_{2,s}\rho_s)\calv^{I\!,J}(ds),\\
\langle N^n,B_2\rangle_t&\xrightarrow[n\to\infty]{P} \int_0^t (\beta_{1,s}\sigma_{2,s}+\beta_{2,s}\sigma_{1,s}\rho_s)\calv^{I\!,J}(ds)
\end{align*}
\end{lemma}

This lemma describes the parts of the limit of $N^n$ that can be described or explained by 
$B_1$ and $B_2$. This is however not enough. One also needs to evaluate the limiting quadratic variation of the process $N^n$.


\begin{lemma}\label{lem3.10} If Assumption \textsf{P2} is fulfilled, then for every $t\in[0,T]$, we have
\begin{align*}
\langle N^n,N^n\rangle_t&\xrightarrow[n\to\infty]{P}
\int_0^t \beta_{2}^2\sigma_{1}^2\,d\calv^{I\!,J,J'}
+\int_0^t \beta_{1}^2\sigma_{2}^2\,d\calv^{I,I'\!\!,J}
+2\int_{0}^t \beta_{2}\beta_{1}\sigma_{1}\sigma_{2}\rho\,d\calv^{J(I),I(J),I\cap J}.
\end{align*}
\end{lemma}

The last step before stating the main result on the convergence of the processes involved in the
stochastic decomposition presented in Lemma~\ref{lem-decomp} is the proof of the convergence of
the bounded variation process $A^n$. Recall that the latter is defined by
\begin{align*}
A^n&= b_n^{-1}\sum_{I\!,J}\> \KIJ \Bigl\{
J\{[I\sigma_1(\beta^{[1]}_{21}+\beta^{[1]}_{22}\rho)]\cdot s\}
+I\{[J\sigma_2(\beta^{[1]}_{11}\rho+\beta^{[1]}_{12})]\cdot s \}\Bigr\}\cdot t
\\
&+ b_n^{-1}\sum_{I\!,J}\>\KIJ \{(I\beta_1)\cdot t\} \times
\{(J\beta_2)\cdot t\}.
\end{align*}
Obviously, it can be written as $A^{n}_t=A^{1,n}_t+A^{2,n}_t$, where
\begin{align*}
A^{1,n}_t&=b_n^{-1}\sum_{I\!,J} \KIJ\int_I\int_J \Big\{\sigma_{1,u}(\beta^{[1]}_{21,u}+\beta^{[1]}_{22,u}\rho_u)+
\sigma_{2,s}(\beta^{[1]}_{11,s}\rho_s+\beta^{[1]}_{12,s})\Big\}\1_{\{u\le s\le t\}}\,du\,ds\\
A^{2,n}_t&=b_n^{-1}\sum_{I\!,J} \KIJ\int_I\int_J \beta_{1,u}\beta_{2,s}\1_{\{u\vee s\le t\}}\,du\,ds=\int_{[0,t]^2}
\beta_{1,u}\beta_{2,s}\,\calv_{n}^{I\!,J}(du,ds).
\end{align*}
Using Assumption \textsf{P1} and the fact that the measures $\calv_{n}^{I\!,J}$ are concentrated on the diagonal of
the square $[0,t]^2$, we get $A^n_t=A^\infty_t+o_p(1)$ with
\begin{align}\label{An}
A^\infty_t&=\frac12\int_0^t\{\sigma_{1,u}(\beta^{[1]}_{21,u}+\beta^{[1]}_{22,u}\rho_u)+
\sigma_{2,u}(\beta^{[1]}_{11,u}\rho_u+\beta^{[1]}_{12,u})+2\beta_{1,u}\beta_{2,u}\}\calv^{I\!,J}(du).
\end{align}

\begin{proposition}\label{propWC} Assume that the functions $\sigma_1$, $\sigma_2$ and $\rho$ are continuous in $[0,T]$
and that $\sup_{t\in[0,T]} \Ex[(\beta_{ij}^{[\ell-1]})^4]<\infty$ for every $i,j,\ell\in\{1,2\}$. If assumptions
P and P1 are fulfilled, then the sequence of two dimensional processes $(M^n,N^n+A^n)$ converges weakly
to a process $(M^\infty,N^\infty+A^\infty)$. Furthermore, $N^\infty+A^\infty$ is independent of $M^\infty$.
\end{proposition}
\begin{proof}
We already did the major part of the proof by showing the convergence in probability of the sequences
of quadratic variations-covariations and that of $A^n_t$. Now, if we apply Theorem 2-1 from \cite{Jac97} to
the semimartingale $Z^n=(M^n,N^n+A^n)^\T$ with $\mB$ serving as a martingale of reference (denoted by $M^n$
in \cite{Jac97}), we obtain the weak convergence of $Z^n$ to a process $Z$. Moreover, it follows from (ii)
of the aforementioned theorem that $Z$ may be constructed on an enlargement of the original probability space
on which there is a two-dimensional Brownian  motion $\tilde\mB$ independent of $\mB$  such that
$$
Z_t=
\begin{pmatrix}
0\\
A^\infty_t
\end{pmatrix}+\int_0^t
\frac{d\calv^{I\!,J}}{dt}(s)
\begin{pmatrix}
0 & 0\\
\beta_{2,s}\sigma_{1,s} & \beta_{1,s}\sigma_{2,s}
\end{pmatrix}
\,d\mB_s+\int_0^t
\begin{pmatrix}
\mathfrak m_s & 0\\
0&\mathfrak w_s
\end{pmatrix}
d\tilde\mB_s,
$$
where
$$
\mathfrak m_s^2=h_s^2\,\Big\{\frac{d\calv^{I}}{ds}+\frac{d\calv^{J}}{ds}-\frac{d\calv^{I\cap
J}}{ds}\Big\}+\sigma_{1,s}^2\sigma_{2,s}^2\,\frac{d\calv^{I\!,J}}{ds}
$$
stands for the Radon-Nikodym derivative of $\lim_{n\to\infty}\langle M^n,M^n\rangle_t$ with respect to the Lebesgue
measure (cf.\ Lemma~\ref{lem8}) and $\mathfrak w_s$ is a predictable process (hence independent of $\tilde\mB$). If we denote
$(M^\infty,N^\infty)=Z^\T-(0, A^\infty)$, we get $M^\infty_t=\int_0^t\mathfrak m_s\,d\tilde B_{1,s}$ and $N^\infty_t=\int_0^t\mathfrak n_{1,s}\,dB_{1,s}+\int_0^t\mathfrak n_{2,s}\,dB_{2,s}+\int_0^t\mathfrak w_{1,s}\,d\tilde B_{2,s}$
with a predictable process $\mathfrak n_s=(\mathfrak n_1,\mathfrak n_2)$, and the assertion of the proposition follows.
\end{proof}
This result implies in particular that $\Ex[N^\infty_t+A^\infty_t |M^\infty_t]=\Ex[N^\infty_t+A^\infty_t]=\Ex[A^\infty_t]$ for every $t\in[0,T]$.
Therefore, using (\ref{An}), we get
\begin{align*}
{\sf A}&= \Ex[N^\infty_T+A^\infty_T|M^\infty_T]\\
&=\frac12\int_0^T \{\sigma_{1,u}\Ex(\beta^{[1]}_{21,u}
+\beta^{[1]}_{22,u}\rho_u)+\sigma_{2,u}\Ex(\beta^{[1]}_{11,u}\rho_u+\beta^{[1]}_{12,u})+2\Ex[\beta_{1,u}\beta_{2,u}]\}\calv^{I\!,J}(du).
\end{align*}
As we see in the next section, this expression of ${\sf A}$ appears in the asymptotic expansion of
the distribution function of $b_n^{-1/2}(\hat\theta_n-\theta)$. 

\section{Expansion of the distribution for a model with drift terms}\label{071219-3}

The aim of this section is to obtain an asymptotic expansion for the distribution of the HY-estimator
in the case where the diffusions $X_1$ and $X_2$ have non-zero drifts. As shows the stochastic expansion
of $\hat\theta_n$ obtained in Lemma~\ref{lem-decomp}, the main term in the expansion of $b_n^{-1/2}(\hat\theta_n-\theta)$ is independent of the drifts. Therefore, asymptotic expansions for
its distribution are already obtained in Sections~\ref{asy.exp.dist} and \ref{Sec4}. This indicates that
the influence of the drifts on the distribution of $\hat\theta_n$ can be regarded as a small perturbation
of the distribution in the case where there is no drift. Before stating the main result of this section,
let us give a theorem that allows to derive the second-order expansion of the distribution of a random
variable defined on the Wiener space in presence of a random perturbation.

\subsection{Perturbation}

Since the drift terms are possibly non-linear functionals of the Brownian motion $\mB$,
we need the Malliavin calculus to carry out computations on the infinite-dimensional
Gaussian space.

The basis of our arguments is a perturbation method for deriving asymptotic expansion.
It was used in \cite{Yos97} for the perturbation of
a martingale but the proof was written inseparably from the martingale structure.
In order to apply this methodology to the present situation,
we will begin with generalizing Theorem 2.1 of Sakamoto and Yoshida \cite{Sak-Yos02}.

We consider a probability space equipped with a differential calculus
in Malliavin's sense, an integration-by-parts formula and
the Sobolev spaces $\mathbb D_{p,\ell}$ equipped with the norm $\|\cdot\|_{p,\ell}$. For
positive numbers $M$ and $\gamma$, let $\cal{E}(M,\gamma)$ be the set of all measurable
functions $f:\bbR^d\rightarrow\bbR$ satisfying $|f(x)|\leq
M(1+|x|^\gamma)$ for all $x\in\bbR^d$. Let $\cale'$ be a subset of $\cal{E}(M,\gamma)$.

Let $\calx_n$ and $\caly_n$ be $\bbR^d$-valued Wiener functionals and put
\begin{align*}
\calz_n=\calx_n+s_n\caly_n
\end{align*}
for some sequence of positive numbers $s_n$ tending to $0$ as
$n\iku\infty$. We write $G_n(f)=\bar{o}(s_n)$ if
$s_n^{-1}\sup_{f\in\cale}|G_n(f)|\iku0$ as $n\iku\infty$.

\begin{theorem}\label{ThmSY}
Let $\ell$ be an integer such that $\ell>d+2$.
Suppose that the following conditions are satisfied:
\begin{itemize}
\item[(1)]
$\sup_n ||\calx_n||_{p,\ell}+\sup_n||\caly_n||_{p,\ell}<\infty$
for any $p>1$,
\item[(2)] $(\calx_n, \caly_n) \stackrel{D}{\to} (\calx_\infty, \caly_\infty)$
for some random variables $\calx_\infty$ and $\caly_\infty$.
\end{itemize}
In addition, assume that there exists a functional $\tau_n$ such that
\begin{itemize}
\item[(3)] $\sup_n ||\tau_n||_{p,\ell-1}<\infty$ for any $p>1$.
\item[(4)] $\P[|\tau_n|>1/2]=o(s_n^\alpha)$ for some $\alpha>1$.
\item[(5)] $\sup_n \E[1_{\{|\tau_n|<1\}}
(\det\sigma_{\calx_n})^{-p}]<\infty$ for any $p>1$.
\item[(6)]
There is a sequence of signed measures $\Psi_n$ on $\bbB_d$ such that
for any positive numbers $M$ and $\gamma$,
$\Ex[f(\calx_n)]=\Psi_n[f]+\bar{o}(s_n)$ as $n\iku\infty$
for $f\in\cale'$.
Moreover, for every 
polynomial $\pi(x)$ in $x$,
there exists a constant 
$c_\pi$
such that
$|\Psi_n[e^{{\ii}u\cdot x}\pi(x)]|\leq c_\pi (1+|u|^{\ell-1})^{-1}$
for all $u\in\bbR^d$.
\end{itemize}
Then $\calx_\infty$ has a density $p^{\calx_\infty}$ with respect to the Lebesgue measure and,  for any positive numbers $M$ and $\gamma$,
\begin{equation}
\Ex[f(\calz_n)]=\Psi_n[f]+s_n \int_{\bbR} f(x) g_\infty(x)\>dx
+\bar{o}(s_n)
\end{equation}
for $f \in \cal{E}'$, where $g_\infty(x)=-\mbox{{\rm div}}_x \bigl(\E[\caly_\infty\ | \
\calx_\infty=x ]\,p^{\calx_\infty}(x) \bigr)$.
\end{theorem}

\subsection{Asymptotic expansion of the distribution}\label{sec6.2}
We are now in a position to state and to prove the main result of this section,
which provides an unconditional asymptotic expansion of the distribution of the
HY-estimator. It is also possible to derive asymptotic expansions conditionally
to the processes generating the sampling times, but they have more complicated
form and are not presented here.

\begin{theorem}\label{200325-20}
Suppose that Assumptions {\sf P1}  and {\sf P2} are satisfied and
$$
\sup_{t\in[0,T]}\|\beta^{[l-1]}_{i,t}\|_{p,4}<\infty,\quad  \text{for all}\quad p>1\quad \text{and}\quad i,l\in\{1,2\}.
$$
Let us define
\begin{align*}
\mfc&=\int_0^T\sigma_{1,t}^2\sigma_{2,t}^2\calv^{I\!,J}(dt)+\int_0^T\sigma_{1,t}\sigma_{2,t}\rho_t\big\{\calv^{I}(dt)+\calv^{I}(dt)-\calv^{I\cap J}(dt)\big\},\\
{\sf A}&= \frac12\int_0^T \{\sigma_{1,u}\Ex(\beta^{[1]}_{21,u}
+\beta^{[1]}_{22,u}\rho_u)+\sigma_{2,u}\Ex(\beta^{[1]}_{11,u}\rho_u+\beta^{[1]}_{12,u})+2\Ex[\beta_{1,u}\beta_{2,u}]\}\calv^{I\!,J}(du).
\end{align*}
Under the notation of Theorem~\ref{191214-1}, if for some $a\in(3/4,1)$\/,  $\Pb(A_n(a)^c)=o(b_n^{p})$ for every $p>1$, and  $\Ex[2\mu_{2,n}-\mfc]=O(b_n^{2a-1})$,
then
\begin{equation}\label{ae5}
\sup_{f\in\cale(M,\gamma)\cap\cale^0({\sf C},\eta,r_0,\mfc^*)} \l| \>\Ex[f(b_n^{-1/2}(\hat\theta_n-\theta))]-\int_\bbR f(z)\,p_n^*(z)\,dz\>\r|=o(b_n^{1/2}),
\end{equation}
where
\begin{align*}
p_{n}^*(z)&=\frac{e^{-z^2/(2\mfc)}}{\sqrt{2\pi\mfc}}\Bigl[1
+\frac{b_n^{1/2}}{6\mfc^3}\big(\Ex[\bar{\lambda}_{3,n}](z^3-3\mfc z)+6{\sf A}\mfc^2z\big)\Bigr].
\end{align*}
Moreover, if $\sup_{n\in\bbN}\Ex[\bar{\lambda}_{3,n}]<\infty$, then
inequality (\ref{ae5}) holds with $p_n^*$ replaced by
$$
p_n^+(z)=\frac{\max(0,p_n^*(z))}{\int_\bbR \max(0,p_n^*(u))\,du}\ ,
$$
which is a probability density.
\end{theorem}
\begin{proof}
We apply Theorem~\ref{ThmSY} to $\calz_n=b_n^{-1/2}(\hat\theta_n-\theta)$ with $\ell=4$, $\calx_n=M^n_T$ and $\caly_n=b_n^{-1/2}(\calz_n-M^n_T)$.
Thus,  we need to check that all the 6 conditions of Theorem~\ref{ThmSY} are fulfilled. In view of Lemma~\ref{lem-decomp} and Proposition~\ref{propWC},
$(\calx_n,\caly_n)$ converges in distribution to some random vector $(\calx_\infty,\caly_\infty)$. Thus the second condition
of Theorem~\ref{ThmSY} is verified.

We have already seen in Section~\ref{SSec2} that the principal part
$\calx_n$ of $b_n^{-1/2}(\hat\theta_n-\theta)$ can be written in the
form
$
\calx_n=b_n^{-1/2}(\xxi^\T A\xxi-\theta)=b_n^{-1/2}\sum_{\ell=1}^{N}
\lambda_{\ell,n} (\zeta_{\ell,n}^2-1),
$
where
$$
\xxi=(\{I^1\sigma_1\cdot B_1\}_T,\ldots,\{I^{N_1}\sigma_1\cdot
B_1\}_T,\{J^1\sigma_2\cdot B_2\}_T,\ldots,\{J^{N_2}\sigma_2\cdot
B_2\}_T,)^\T\sim \caln_N(0,\Sigma)
$$
and the entries of the matrices $\Sigma$ and $A$ are given by
(\ref{sigma}) and (\ref{a}) respectively. Recall that the vector
$\zzeta\in\bbR^{N}$ is obtained as a linear transformation of $\xxi$
and is distributed according to $\caln(0,I)$.

Let $W=C_0([0,T],\bbR^2)$ be the Wiener space of continuous
functions from $[0,T]$ to $\bbR^2$ vanishing at the origin. Recall
that $W$ is a measurable space equipped with the Borel
$\sigma$-field induced by the uniform topology. The reference
measure on $W$ is the measure generated by the standard Wiener
process (in our case, the two-dimensional Brownian motion).

Let $w=(w_1,w_2)$ be the canonical process on $W$. Then, $(B_1,B_2)$
can be defined by
$$
B_{1,t}=w_{1,t},\qquad B_{2,t}=\int_0^t
\rho_s\,dw_{1,s}+\int_0^t\sqrt{1-\rho_s^2}\,dw_{2,s}.
$$
Obviously, for every $\ell=1,\ldots,N$, there is some function
$\phi^\ell\in L^{2}([0,T],\bbR^2)$ such that
$\zeta_{\ell,n}=\int_0^T \phi^\ell_{1,t}\,dw_{1,t}+\int_0^T
\phi^\ell_{2,t}\,dw_{2,t}:=w(\phi^\ell)$.

The process $w$ is an isonormal Gaussian process on
$H=L^2([0,T],\bbR^2)$ (see \cite[Def.\ 1.1.1]{Nual}) Using the
definition of the Malliavin derivative (see \cite[Def.\
1.2.1]{Nual}) and the chain rule~\cite[Prop.\ 1.2.3]{Nual}, we
get the following expression for the Malliavin derivative of
$\calx_n$:
$$
D_t\calx_n=2b_n^{-1/2}\sum_{\ell=1}^N\lambda_{\ell,n}\zeta_{\ell,n}\phi^\ell_t.
$$
Since the components of $\zzeta$ are non-correlated with variance
equal to one, the family $\{\phi^\ell\}_{\ell\le N}$ is
orthonormal. As a first consequence of this fact, we get that $\sup_n\|\calx_n\|_{p,4}<\infty$ for
every $p>1$. To show this, Rosenthal's inequality and the result of Lemma~\ref{071217-1} can be used.
As a second consequence, we obtain that	 the Malliavin covariance of $\calx_n$ is
\bea\label{071218-3}
\sigma_{\calx_n}=4b_n^{-1}\sum_{\ell=1}^n\lambda_{\ell,n}^2
\zeta_{\ell,n}^2=4b_n^{-1}\mu_{2,n}+4b_n^{-1}\sum_{\ell=1}^n\lambda_{\ell,n}^2
(\zeta_{\ell,n}^2-1).
\eea
Let us introduce the random variable $\tau_n$ that will play a role
of truncation:
$$
\tau_n=-\big(2-8\mu_{2,n}(\mfc b_n)^{-1}\big)_{+}+ 8(\mfc
b_n)^{-1}\sum_{\ell=1}^N \lambda_{\ell,n}^2(\zeta_{\ell,n}^2-1).
$$
In this notation, we have
$\sigma_{\calx_n}\ge \mfc+\frac{\mfc \tau_n}2$
and, therefore, $\1_{\{|\tau_n|<1\}}|\sigma_{\calx_n}^{-1}|<2/\mfc$.
Thus, the condition (5) of Theorem~\ref{ThmSY} is obviously
fulfilled. Let us check now that $\tau_n$ satisfies conditions (3)
and (4) of the aforementioned theorem.

To verify condition (3) of Theorem~\ref{ThmSY}, we remark that
$$
D\tau_n=16(\mfc b_n)^{-1}\sum_{\ell=1}^N
\lambda_{\ell,n}^2\zeta_{\ell,n}\phi^\ell,\quad D^2\tau_n=16(\mfc
b_n)^{-1}\sum_{\ell=1}^N \lambda_{\ell,n}^2\phi^\ell\otimes\phi^\ell
$$
$D^k\tau_n\equiv 0$ for every $k\ge 3$. Therefore,
$$
\|D\tau_n\|_{H}^2=256(\mfc b_n)^{-2}\sum_{\ell=1}^N
\lambda_{\ell,n}^4\zeta_{\ell,n}^2,\quad \|D^2\tau_n\|^2_{H\otimes
H}=256(\mfc b_n)^{-2}\sum_{\ell=1}^N \lambda_{\ell,n}^4.
$$
In view of the Rosenthal inequality, we get
$$
\Ex^\Pi[\|D\tau_n\|_{H}^{p}]\le
C(p)b_n^{-p}(\mu_{4,n}^{p/2}+\mu_{2p,n}+ \mu_{8,n}^{p/4}),
$$
for every $p\ge 2$. Using the definition of $\mu_{k,n}$, one can
check that $\mu_{2k,n}\le \mu_{4,n}^{k/2}$. In view of inequality
(\ref{boundmuk}) and the obvious bound $\mu_{2,n}\le C r_n$, we get
$$
\Ex^\Pi[\|D\tau_n\|_{H}^{p}]\le C b_n^{-p}r_n^{3p/2},\qquad
\Ex^\Pi[\|D^2\tau_n\|^p_{H\otimes H}]\le C b_n^{-p}r_n^{3p/2},\quad
\forall p\geq 4.
$$
Similar arguments yield
$$
\Ex[\tau_n^p]=\Ex^\Pi[\tau_n^p]\le C(1+ b_n^{-p}\Ex[r_n^{3p/2}])\le C(1+ b_n^{-p}b_n^{9p/8}+T^{3p/2}b_n^{-p}\Pb[A_n(a)^c])<\infty.
$$
To check condition (4) of Theorem~\ref{ThmSY}, we use the inequality
$$
\Pb(|\tau_n|>1/2)\le \Pb\big(2-8\mu_{2,n}(\mfc
b_n)^{-1}>0\big)+\Pb\bigg(8(\mfc b_n)^{-1}\Big|\sum_{\ell=1}^N
\lambda_{\ell,n}^2(\zeta_{\ell,n}^2-1)\Big|>1/2\bigg).
$$
On the one hand, since the event $\{2-8\mu_{2,n}(\mfc
b_n)^{-1}>0\}=\{\bar\lambda_{2,n}-\mfc< -\mfc/2\}$ is included in $A_n(a)^c$, its probability is $o(b_n^p)$ for every $p>1$. On the other hand,
combining the Tchebychev and the Rosenthal inequalities, for every $k\ge 16$ we get
\begin{align*}
\Pb\bigg(8(\mfc b_n)^{-1}\Big|\sum_{\ell=1}^N
\lambda_{\ell,n}^2(\zeta_{\ell,n}^2-1)\Big|>1/2\bigg)&\leq
Cb_n^{-k}\Ex[\mu_{4,n}^{k/2}+\mu_{2k,n}]\le C b_n^{-k}\Ex[r_n^{3k/2}]\\
&\le C b_n^{-k+9k/8}+Cb_n^{-k}\Pb(A_n(a)^c)=O(b_n^2).
\end{align*}
Thus, we proved that conditions (2)-(5) of Theorem~\ref{ThmSY} are fulfilled and that $\sup_n \|\calx_n\|_{p,4}<\infty$.
Condition (6) is ensured by Theorem~\ref{ThmAE3}. To complete the proof, it remains to check that $\sup_n \|\caly_n\|_{p,4}<\infty$.
This inequality can be proved using the identity $\caly_n=b_n^{-1}(\Phi_n^2+\Phi_n^3)$, where $\Phi_n^2$ and $\Phi_n^3$ are the random
variables defined in the proof of Lemma~\ref{lem-decomp}. The proof is rather technical, but is based on the arguments that we have already
used several times in this and the previous sections. Therefore it will be omitted.
\end{proof}

In the case when the sampling scheme is generated by two Poisson processes, we get the following consequence of the last theorem.
\begin{proposition}
Let the sampling times of processes $X_1$ and $X_2$ be generated by two independent Poisson processes  with intensities $np_1$
and $np_2$, $p_1p_2>0$. If
\begin{itemize}
\item the sampling times are independent of the process $\mX$,
\item the functions $\sigma_1$, $\sigma_2$ and $\rho$ are Lipschitz continuous,
\item $\sup_{t\in[0,T]}\|\beta^{[l-1]}_{i,t}\|_{p,4}<\infty$ for all $p>1$, $i,l\in\{1,2\}$,
\end{itemize}
then
\begin{equation}\label{ae6}
\sup_{f\in\cale(M,\gamma)\cap\cale^0({\sf C},\eta,r_0,\mfc^*)} \l| \>\Ex[f(n^{1/2}(\hat\theta_n-\theta))]-\int_\bbR f(z)\,p_n^\circ(z)\,dz\>\r|=o(n^{-1/2}),
\end{equation}
where
\begin{align*}
p_{n}^\circ(z)&\propto \frac{e^{-z^2/(2\mfc)}}{\sqrt{2\pi\mfc}}\Bigl[1
+\frac{1}{\sqrt{n}\mfc^3}\big(2\kappa z^3-6\kappa\mfc z+{\sf A}\mfc^2z\big)\Bigr]_+
\end{align*}
is a probability density with
\begin{align*}
\mfc&=\bigg(\frac2{p_1}+\frac{2}{p_2}\bigg)\int_0^T\sigma_{1,t}^2\sigma_{2,t}^2
(1+\rho_t^2)dt-\frac{2}{p_1+p_2}\int_0^T (\sigma_{1,t}\sigma_{2,t}\rho_t)^2dt,\\
\kappa&=\bigg(\frac{1}{p_1^2}+\frac{1}{p_2^2}\bigg) \int_0^T h(t)^3\,dt
+\frac{3p_1^2+2p_1p_2+3p_2^2}
{p_1^2p_2^2} \int_0^T \sigma_{1,t}^2\sigma_{2,t}^2h(t)\, dt,\\
{\sf A}&=\bigg(\frac1{p_1}+\frac{1}{p_2}\bigg)\int_0^T \{\sigma_{1,t}\Ex(\beta^{[1]}_{21,t}
+\beta^{[1]}_{22,t}\rho_t)+\sigma_{2,t}\Ex(\beta^{[1]}_{11,t}\rho_t+\beta^{[1]}_{12,t})+2\Ex[\beta_{1,t}\beta_{2,t}]\}dt.
\end{align*}
\end{proposition}
\begin{proof}
Lemmas~\ref{lemA5}-\ref{lemA8} (cf.\ Section \ref{App1}) imply that the partitions generated by independent Poisson processes satisfy
Assumptions \textsf{P1} and {\sf P2}. Therefore, using Theorems~\ref{200325-20} and \ref{expPois}, we get the desired result.
\end{proof}

\section{Proofs of theorems and propositions}\label{App0}

\begin{proof}[Proof of Proposition \ref{propo5}]
Let us recall the relations
\begin{align*}
n\sum_{I\!,J} v_1(I)v_2(J)\KIJ&\xrightarrow[n\to\infty]{P}
2(p_1^{-1}+p_2^{-1})\int_0^T \sigma_{1,t}^2\sigma_{2,t}^2dt,\\
n\sum_{I\in\Pi_i} v(I)^2 &\xrightarrow[n\to\infty]{P}
2p_i^{-1}\int_0^T (\sigma_{1,t}\sigma_{2,t}\rho_t)^2\,dt,\quad i=1,2\\
n\sum_{I\!,J} v(I\cap J)^2 &\xrightarrow[n\to\infty]{P}
2(p_1+p_2)^{-1}\int_0^T (\sigma_{1,t}\sigma_{2,t}\rho_t)^2\,dt
\end{align*}
proved in Hayashi and Yoshida~\cite{Hay-Yos04}. The aim of the present proposition
is to show that the rate of convergence in these relations is
$1/\sqrt{n}$ and to get an exponential control of the probabilities of 
large deviations. Thus, let us denote $\calt_1=n\sum_{I\!,J}
v_1(I)v_2(J)\KIJ$ and show that
$$
\Pb\bigg(\bigg|\calt_1- 2(p_1^{-1}+p_2^{-1})\int_0^T
\sigma_{1,t}^2\sigma_{2,t}^2dt\bigg|\ge \frac{x}{\sqrt{n}}\bigg)\le
Cne^{-x/C}.
$$
Let $N(x)=\lceil nT/x\rceil$ be the smallest positive
integer such that $N(x)x >nT$ and let us set
$L_i=[iTN(x)^{-1},(i+1)TN(x)^{-1}]$. The intervals $L_i$ define a
uniform deterministic partition of $[0,T]$ with a mesh-size of order
$x/n$. Let $\cale$ be the event ``for every $i=1,\ldots,4N(x)$, the
interval $[\frac{iT}{4N(x)},\frac{(i+1)T}{4N(x)}]$ contains at least
one point from $\Pi^1_n$ and one point from $\Pi^2_n$''. The total
probability formula implies that
\begin{align*}
\Pb\bigg(&\bigg|\calt_1- \int_0^T \bar h(t)dt\bigg|\ge \frac{x}{\sqrt{n}}\bigg)
\le \Pb\bigg(\bigg|\calt_1- \int_0^T \bar h(t)\,dt\bigg|\ge
\frac{x}{\sqrt{n}}\;\bigg|\;\cale\bigg)+\Pb(\cale^c),
\end{align*}
where $\cale^c$ denotes the complementary event of $\cale$ and $\bar
h(t)=2(p_1^{-1}+p_2^{-1})\sigma_{1,t}^2\sigma_{2,t}^2$. Easy computations show that
$\Pb(\cale^c)\le C nx^{-1} e^{-x/C}$ for some $C>0$.

Let now $l_i$ be a point in $L_i$ such that $\int_{L_i}\bar
h(t)\,dt=\bar h(l_i)|L_i|$. Let us denote by $a_I$ the left endpoint of the interval $I$ and
define the random variables
$$
\eta_i^\circ=n\bar h({l_i})\sum_{I\!,J} |I||J|\KIJ\1_{\{a_I\in
L_i\}},\quad i=1,\ldots,N(x).
$$
In what follows, we denote by $\E^\cale$ the conditional expectation given $\cale$. It holds that $\calt_1-\int_0^T\bar
h(t)\,dt=\calt_{11}+\calt_{12}+\calt_{13}+\mO(n |L_1|^2)$
on $\cale$, where
\begin{align*}
\calt_{11}&=\Ex^\cale\bigg[\sum_{i=1}^{N(x)}\eta_i^\circ\bigg]-\int_0^T\bar
h(t)\,dt,\quad
\calt_{1s}=\sum_{i=1}^{[N(x)/2]}(\eta_{2i+s-2}^\circ-\Ex^\cale[\eta_{2i+s-2}^\circ]),\quad s=2,3.
\end{align*}
For evaluating the remainder term in $\calt_1$, we have used the
Lipschitz continuity of $\sigma_1$ and $\sigma_2$, as well as the
fact that $r_n\le |L_1|/2$ on $\cale$.

Remark that in view of Lemma~\ref{lem2}, for any $p>0$, we have
\begin{align}\label{r_np}
\Ex[r_n^p]&=n^{-p}\int_0^\infty \Pb((nr_n)^p\ge t)\,dt\le C n^{-p}
\int_0^\infty (ne^{-t^{1/p}})\wedge 1\,dt=C n^{-p}\mO(\log^p n).
\end{align}

On the one hand, since $|\sum_{i=1}^{N(x)}\eta^\circ_i|\le Cnr_n$,
we have
$$
\bigg|\Ex^\cale\Big[\sum_{i=1}^{N(x)}\eta_i^\circ\Big]-\Ex\Big[\sum_{i=1}^{N(x)}\eta_i^\circ\Big]\bigg|
\le \frac{n\Ex[r_n\1_{\cale^c}]}{\Pb(\cale)}.
$$
Using the inequality of Cauchy-Schwarz, as well as the bounds
$\Pb(\cale^c)\le Cne^{-x/C}$ and (\ref{r_np}), we get $\big|
\Ex^\cale \big[ \sum_{i=1}^{N(x)} \eta_i^\circ\big] - \Ex\big
[\sum_{i=1}^{N(x)} \eta_i^\circ\big] \big| \le Cne^{-x/C}$, for some
constant $C$ and for every $x>C\log n$.

On the other hand, in view of Lemma~\ref{lemA4} presented in Section~\ref{App1} below, we have
$$
\Ex[\eta_i^\circ]\le n\bar h(l_i)\Ex\bigg[\sum_{I:a_I\in
L_i}\Big(|I|^2+\frac{2|I|}{np_2}\Big)\bigg]\le C
n\Ex[(r_n+n^{-1})(|L_i|+r_n)].
$$
Therefore, using (\ref{r_np}), we get
$\Ex[\eta_i^\circ]=\mO(n^{-1}\log^3 n)$ for every $i\le N(x)$. Using
once again Lemma~\ref{lemA4}, we get
\begin{align*}
\Ex\Big[\sum_{i=1}^{N(x)}\eta_i^\circ\Big]&=\sum_{i=2}^{N(x)-1}
n\bar h(l_i)\Ex\Big[\sum_{I:a_I\in
L_i}|I|\cdot\Ex^{\Pi^1}\Big(\sum_{J\in\Pi^2}|J|\KIJ\Big)\Big]+\mO\Big(\frac{\log^3
n}{n}\Big)\\
&=\sum_{i=2}^{N(x)-1} n\bar h(l_i)\Ex\Big[\sum_{I:a_I\in
L_i}\big(|I|^2+2|I|/(np_2)\big)\Big]+\mO\Big(\frac{\log^3
n}{n}\Big).
\end{align*}
Wald's equality yields
$
\Ex\Big[\sum_{I:a_I\in L_i} |I|^k\Big]=\Ex[N_1(L_i)]\cdot
\Ex[\zeta^k/(np_1)^k]+\mO(e^{-\log^2n/C}),
$ for every $k>0$ and for every $i\le N(x)-1$. Here, $N_1(L_i)$ is the
number of points of $\mP^{1,n}$ lying in $L_i$ and $\zeta\sim\msE(1)$, the exponential 
distribution with parameter one. Putting all these estimates together, we get
\begin{align*}
\Ex\Big[\sum_{i=1}^{N(x)}\eta_i^\circ\Big]&=\sum_{i=2}^{N(x)-1}
n\bar
h(l_i)\Big(\frac{2|L_i|}{np_1}+\frac{2|L_i|}{np_2}\Big)+\mO\Big(\frac{\log^3
n}{n}\Big)\\
&=\Big(\frac2{p_1}+\frac2{p_2}\Big)\sum_{i=1}^{N(x)} \bar
h(l_i)|L_i|+\mO\Big(\frac{\log^3 n}{n}\Big).
\end{align*}
Since $l_i$ has  been chosen such that $\bar h(l_i)|L_i|=\int_{L_i}
\bar h(t)\,dt$, the last relation implies that
$\calt_{11}=\mO(n^{-1}\log^3n)$.

The advantage of working with $\eta^\circ_i$s is that, conditionally
to $\cale$, the random variables $\eta^\circ_{2i}$,
$i=1,\ldots,[N(x)/2]$, are independent. Indeed, one easily checks
that conditionally to $\cale$, $\eta^\circ_{2i}$ depends only on the
restrictions of $\mP^{1,n}$ and $\mP^{2,n}$ onto the interval
$[\frac{(4i-1)T}{2N(x)},\frac{(4i+3)T}{2N(x)}]$. Since these
intervals are disjoint for different values of $i\in\bbN$, the
restrictions of Poisson processes $\mP^{k,n}$, $k=1,2$, onto these
intervals are independent. Therefore, $\eta_{2i}^\circ$,
$i=1,\ldots,[N(x)/2]$, form a sequence of random variables that are
independent conditionally to $\cale$. Moreover, they verify
$|\eta_i^\circ|\le Cn|L_i|^2=\frac{C\log^4 n}{n}$.

These features enable us to use the Bernstein inequality in order to
bound the probabilities of large deviations of $\calt_{12}$ as follows:
\begin{align*}
\Pb^\cale\big(|\calt_{12}|\ge x/\sqrt{n}\big)&\le
2\exp\Big({-\frac{x^2}{C(1+xn^{-1/2}\log^4 n)}}\Big)\le 2e^{-x/C},
\end{align*}
for every $x>1$. Obviously, the same inequality holds true for the
term $\calt_{13}$. These inequalities combined with the bound on the
error term $\calt_{11}$ complete the proof of (\ref{exp_in_mu}).

Moreover, since $\calt_{12}$ and $\calt_{13}$ are zero mean random variables,
conditionally to $\cale$, and $\cale^c$ has a probability bounded by
$Cne^{-x/C}$, it follows from the computations above that
$$
\Ex[\calt_1]=2(p_1^{-1}+p_2^{-1})\int_0^T
\sigma_{1,t}^2\sigma_{2,t}^2\,dt+\mO(n^{-1}\log^3 n).
$$
Similar arguments entail that
$\Ex[2n\mu_{2,n}]=\mfc+\mO(n^{-1}\log^3 n)$.
\end{proof}

\begin{proof}[Proof of Proposition \ref{prop6}]
The assertion of the theorem follows from the following relations:
\begin{align*}
&\Ex\Big[\sum_{I\in\Pi_i} v(I)^3\Big]=\frac{6}{n^2p_i^2}\int_0^T h(t)^3\,dt
+\mO\Big(\frac{\log^3 n}{n^3}\Big),\quad i=1,2,\\
&\Ex\Big[\sum_{I\!,J} v(I\cap J)^3\Big]=\frac{6}{n^2(p_1+p_2)^2}
\int_0^T h(t)^3\,dt+\mO\Big(\frac{\log^3 n}{n^3}\Big),\\
&\Ex\Big[\sum_{I\!,J} v(I\cap J)^2v(I)\Big]=\frac{18p_1+12p_2}{n^2p_1(p_1+p_2)^2}
\int_0^T h(t)^3\,dt+\mO\Big(\frac{\log^3 n}{n^3}\Big),\\
&\Ex\Big[\sum_{I\!,J} v(I\cap J)^2v(J)\Big]=\frac{18p_2+12p_1}{n^2p_2(p_1+p_2)^2}
\int_0^T h(t)^3\,dt+\mO\Big(\frac{\log^3 n}{n^3}\Big),\\
&\Ex\Big[\sum_{I\!,J} v(I\cap J)v(I)v(J)\Big]=\frac{4}{n^2p_2p_1}
\int_0^T h(t)^3\,dt+\mO\Big(\frac{\log^3 n}{n^3}\Big),\\
&\Ex\Big[\sum_{I\!,J} v(I\cup J)v_1(I)v_2(J)\Big]=\frac{6p_1^2+4p_1p_2+6p_2^2}
{n^2p_1^2p_2^2} \int_0^T \frac{h(t)^3}{\rho_t^2}\, dt
+\mO\Big(\frac{\log^3 n}{n^3}\Big).
\end{align*}
Let us prove in detail the fifth relation. The proofs of the other relations
are based on similar arguments and are easier than that of fifth relation.

Using the Lipschitz continuity of the function $h$, one can check that
$v(I\cap J)v(I)v(J)=h(a_I)^3|I|\cdot|J|\cdot|I\cap J|+\mO(r_n^3)|I\cap J|$,
where $a_I$ is the left endpoint of the interval $I$.

In view of (\ref{r_np}), we have 
$\Ex\Big[ \sum_{I\!,J}r_n^3|I\cap J|\Big] \le T\Ex[r_n^3]
=\mO\Big(\frac{\log^3 n}{n^3}\Big).
$ On the other hand
$$
\Ex\Big[\sum_{I\in \Pi^1}h(a_I)^3|I| \sum_{J\in\Pi^2}
|J| \,|I\cap J| \Big]=
\Ex\Big[\sum_{I\in \Pi^1}h(a_I)^3|I|\Ex^{I}
\Big( \sum_{J\in\Pi^2} |J|\,|I\cap J|\Big)\Big],
$$
where $\Ex^I$ is the conditional expectation given $I$.
According to Lemmas~\ref{lemA2} and \ref{lemA4}, presented in Section~\ref{App1} below,
$$
\Ex^I\Big( \sum_{J\in\Pi^2} |J|\,|I\cap J|\Big)\Big]=
\frac{2|I|}{np_2}-\frac{(1- e^{-np_2|I|})(e^{-np_2a_I}+ e^{-np_2(T-b_I)})}
{n^2p_2^2}.
$$
Now, let us show  that
\begin{align*}
\calt_1&:=\frac{2}{np_2}\;\Ex\Big[\sum_{I\in \Pi^1}h(a_I)^3|I|^2
\Big]=
\frac{4}{n^2p_1p_2}\int_0^T h^3(t)\,dt+\mO(n^{-3}),\\
\calt_2&:=\Ex\Big[\sum_{I\in \Pi^1}h(a_I)^3|I|\frac{(1-
e^{-np_2|I|})e^{-np_2a_I}}
{n^2p_2^2}\Big]=\mO(n^{-3}),\\
\calt_3&:=\Ex\Big[\sum_{I\in \Pi^1}h(a_I)^3|I|\frac{(1-
e^{-np_2|I|}) e^{-np_2(T-b_I)}}{n^2p_2^2}\Big]=\mO(n^{-3}).
\end{align*}
To this end, we use the characterization of a Poisson process as a
renewal process with exponential waiting times. Let $(\zeta_k, k\ge
1)$ be a sequence of i.i.d.\ random variables drawn from the exponential distribution 
with mean $1/(np_1)$. Then $N_1$, $S^i$ can be defined by
$N_1=\inf\{k\ge 1: \zeta_1+\ldots+\zeta_{k}\ge T\}$ and
$S^i=(\zeta_1+\ldots +\zeta_{i})\wedge T$ for $i=1,\ldots,N_1$. In
this notation,
\begin{align*}
\calt_1&=\frac{2}{np_2}\;\Ex\Big[\sum_{i=1}^{N_1-1} h(S^i)^3
\zeta_{i+1}^2
\Big]+\mO(n^{-3}),\quad
|\calt_2|\le \|h\|_\infty^3\Ex\Big[\sum_{i=1}^{N_1-1}
\frac{\zeta_{i+1}e^{-np_2S^i}} {n^2p_2^2}\Big]+\mO(n^{-3}),
\end{align*}
where $\|h\|_\infty=\max_{t\in[0,T]}|h(t)|$. Remark that $N_1$ is a
stopping time with respect to the filtration
$\calf_k=\sigma(\zeta_1,\ldots,\zeta_k)$, $k\ge 1$. It is easily
seen that
\begin{align*}
M_k&=\sum_{i=1}^{k-1} h(S^i)^3 (\zeta_{i+1}^2-\Ex[\zeta_{i+1}^2]),\quad
M_k'=\sum_{i=1}^{k-1}\;\big( \zeta_{i+1}-\Ex[\zeta_{1}]\big)
e^{-np_2S^i}
\end{align*}
are $\calf_k$-martingales for which the conditions of the optional
stopping theorem are fulfilled. Therefore
\begin{align*}
\calt_1&=\frac{2}{np_2}\;\Ex[\zeta_{1}^2]\,\Ex\Big[\sum_{i=1}^{N_1-1}
h(S^i)^3
\Big]+\mO(n^{-3}),\\
\calt_2&\le
\frac{\|h\|_\infty^3}{n^2p_2^2}\;\Ex[\zeta_{1}]\,\Ex\Big[\sum_{i=1}^{N_1-1}
e^{-np_2S^i}\Big]+\mO(n^{-3}).
\end{align*}
These relations imply that
\begin{align*}
\calt_1=\frac{4}{n^2p_1p_2}\int_0^T h(t)^3\,dt+\mO(n^{-3}),\ 
|\calt_2|\le \frac{\|h\|_\infty^3}{n^2p_2^2}\,\int_0^T
e^{-np_2t}\,dt+\mO(n^{-3})=\mO(n^{-3}).
\end{align*}
In the above inequalities we used the fact that for any integrable
function $f$ on $[0,T]$, the equality $\Ex[\sum_{i=1}^{N_1-1} f(S^i)]=
n p_1 \int_0^T f(t) \,dt$ holds true.

The term $\calt_3$ can be bounded in the same way as $\calt_2$ by
using the fact that if $\{t_1,\ldots,t_{N}\}$ is a realization of a
homogeneous Poisson point process in $[0,T]$, then
$\{T-t_1,\ldots,T-t_{N}\}$ can be seen as a realization of the same
Poisson point process. This completes the proof of the proposition.
\end{proof}

\begin{proof}[Proof of Theorem~\ref{ThmSY}]
Let $\psi_n$ be some truncation functional to be defined later and let $\zeta(x)=1+|x|^{2m}$ $(x\in\bbR^d)$, where
$m$ is an integer such that $2m>\gamma+d$. We have
\beas
\E[f(\calz_n)]=\E[f(\calz_n)\psi_n]+\E[f(\calz_n)(1-\psi_n)]
=\int_{\bbR^d}f(x)\tilde{p}_n(x)\,dx+\E[f(\calz_n)(1-\psi_n)],
\eeas
where
$\tilde{p}_n(x)
=\frac{1}{(2\pi)^d}\int_{\bbR^d}\>e^{-{\ii}u\cdot x}\>
\hat{g}^0_n(u)
\>du,$
with
$\hat{g}^0_n(u)=\E[e^{{\ii}u\cdot \calz_n}\psi_n]$.

We will show below (cf.\ (\ref{1-psi})) that the term $\E[f(\calz_n)(1-\psi_n)]$ is $\bar{o}(s_n)$ and is negligible with respect to
$\E[f(\calz_n)\psi_n]$. To deal with this latter term, let us introduce the notation
\begin{align*}
h_n^0(x)&=
\frac{1}{(2\pi)^d}\int_{\bbR^d}\>e^{-{\ii}u\cdot x}
\hat{h}_n^0(u)\>du,\\
\hat{h}_n^0(u)
&=
\Psi_n[e^{{\ii}u\cdot x}]
+s_n\E\Bigl[e^{{\ii}u\cdot \calx_\infty}\>{\ii}u\cdot \caly_\infty
\Bigr],\\
\hat{g}_n(u)&=
\E[e^{{\ii}u\cdot\calz_n}\psi_n\zeta(\calz_n)],\\
\hat{h}_n(u)&=
\zeta(-\ii\partial_u)\hat{h}_n^0(u)
=\Psi_n[e^{{\ii}u\cdot x}\zeta(x)]
+s_n\E\Bigl[\zeta(-\ii\partial_u)(e^{{\ii}u\cdot y}{\ii}u)\Big|_{y=\calx_\infty}\cdot
\caly_\infty\Bigr].
\end{align*}
Using the Integration By Parts (IBP) formula, we get
\begin{align*}
\zeta(x)\tilde{p}_n(x)=
\frac{1}{(2\pi)^d}\int_{\bbR^d}\>e^{-{\ii}u\cdot x}\hat{g}_n(u)\>du,\quad
\zeta(x)h_n^0(x)
=\frac{1}{(2\pi)^d}\int_{\bbR^d}\>e^{-{\ii}u\cdot x}\hat{h}_n(u)\>du.
\end{align*}
Further, there is a linear form $\zeta_2(x,y)[\cdot]$
of polynomial elements such that
\beas
\zeta(x+y)=\zeta(x)+\partial\zeta(x)[y]+\zeta_2(x,y)[y^{\otimes2}]
\eeas
for $x,y\in\bbR^d$.
We also notice that, for all $u,y\in\bbR^d$,
\beas
\zeta(-\ii\partial_u)(e^{{\ii}u\cdot y}{\ii}u)
&=&
\zeta(-\ii\partial_u)\partial_y e^{{\ii}u\cdot y}
=\partial_y(\zeta(-\ii\partial_u) e^{{\ii}u\cdot y})
\\&=&
\partial_y(e^{{\ii}u\cdot y}\zeta(y))
=e^{{\ii}u\cdot y}\zeta(y){\ii}u+e^{{\ii}u\cdot y}\partial\zeta(y).
\eeas
Let $\varphi(x)=f(x)/\zeta(x)$ and $\Lambda_n=\{u\in\bbR^d;\>|u|\leq s_n^{-1}\}$.
Then
\beas
(2\pi)^d\int_{\bbR^d}f(x)\bigl\{\tilde{p}_n(x)-h_n^0(x)\bigr\}\>dx
&=&A(n)+s_nB(n)+s_nC(n)+s_n^2D(n)+E(n),
\eeas
where
\begin{align*}
A(n)&=
\int_{\bbR^d}dx\,\varphi(x)
\int_{\Lambda_n}e^{-{\ii}u\cdot x}
\Bigl\{\E\Bigl[e^{{\ii}u\cdot\calx_n}\psi_n\zeta(\calx_n)\Bigr]
-\Psi_n\Bigl[e^{{\ii}u\cdot x}\zeta(x)\Bigr]\Bigr\}\>du,\\
B(n)&=
\int_{\bbR^d}dx\,\varphi(x)
\int_{\Lambda_n}e^{-{\ii}u\cdot x}
\Bigl\{\E\Bigl[e^{{\ii}u\cdot\calx_n}{\ii}u\cdot\caly_n
\>\int_0^1\exp({\ii}s_nu\cdot\caly_ns)ds\>\psi_n\zeta(\calx_n)\Bigr]
\\
&\qquad\qquad-\E\Bigl[e^{{\ii}u\cdot\calx_\infty}{\ii}u\cdot\caly_\infty\zeta(\calx_n)\Bigr]
\Bigr\}\>du,\\
C(n)&=\int_{\bbR^d}dx\,\varphi(x)
\int_{\Lambda_n}e^{-{\ii}u\cdot x}
\Bigl\{\E\Bigl[e^{{\ii}u\cdot\calz_n}\psi_n\partial\zeta(\calx_n)[\caly_n]\Bigr]
-\E[e^{{\ii}u\cdot\calx_\infty}\partial\zeta(\calx_\infty)[\caly_\infty]\Bigr]
\Bigr\}\>du,\\
D(n)&=
\int_{\bbR^d}
dx\,\varphi(x)
\int_{\Lambda_n}e^{-{\ii}u\cdot x}\E\Bigl[e^{{\ii}u\cdot\calz_n}\psi_n
\zeta_2(\calx_n,s_n\caly_n)[\caly_n^{\otimes 2}]\Bigr]
\>du\\
E(n)&=\int_{\bbR^d}dx\,\varphi(x)
\int_{\Lambda_n^c}e^{-{\ii}u\cdot x}
\Bigl(\hat{g}_n(u)-\hat{h}_n(u)\Bigr)\>du.
\end{align*}
Since
$$
\int_{\bbR^d}dx\varphi(x)
\int_{\bbR^d}e^{-{\ii}u\cdot x}
\E\Bigl[e^{{\ii}u\cdot\calx_n}\psi_n\zeta(\calx_n)\Bigr]
\>du
=
(2\pi)^d
\E\Bigl[\varphi(\calx_n)\psi_n\zeta(\calx_n)\Bigr]
$$
and $
\int_{\bbR^d}dx\varphi(x)
\int_{\bbR^d}e^{-{\ii}u\cdot x}
\Psi_n\bigl[e^{{\ii}u\cdot x}\zeta(x)\bigr]\>du
=
(2\pi)^d
\Psi_n[\varphi\zeta]$, we have
\beas
|A(n)|
&\leq&
(2\pi)^d\Bigl| \E\Bigl[\varphi(\calx_n)\psi_n\zeta(\calx_n)\Bigr]
-
\Psi_n[\varphi\zeta]\Bigr|
+F(n)
\\&\leq&
(2\pi)^d\Bigl|\E\Bigl[\varphi(\calx_n)(1-\psi_n)\zeta(\calx_n)\Bigr]\Bigr|
+F(n)+\bar{o}(s_n)
\eeas
from condition (6) of Theorem~\ref{ThmSY},
where
\beas
F(n)&=&
(2\pi)^d
\int_{\bbR^d}|\varphi(x)|\>dx \times
\int_{\Lambda_n^c}
\Bigl\{\Bigl|\E\Bigl[e^{{\ii}u\cdot\calx_n}\psi_n\zeta(\calx_n)\Bigr]\Bigr|
+\Bigl|\Psi_n\Bigl[e^{{\ii}u\cdot x}\zeta(x)\Bigr]\Bigr|
\Bigr\}\>du.
\eeas
In what follows ${\sf C}$ denotes a generic constant independent of $n$ and $u$
and it varies from line to line.

To evaluate $F(n)$, we need the explicit form of $\psi_n$. Let us denote by
$\psi$ a smooth function from $\bbR$ into $[0,1]$ such that $\psi(t)=1$ if $|t|\leq 1/2$
and $\psi(t)=0$ if $|t|\geq1$.
We can write
$$
\det\Big[I_d+s_n\sigma_{\calx_n}^{-1}(\langle\calx_n,\caly_n\rangle
+\langle\caly_n,\calx_n\rangle)\Big]
=
1+s_n\det\sigma_{\calx_n}^{-d}\>K_n
$$
with a certain functional $K_n$ satisfying, for every $p>1$, the condition
$\sup_n\|K_n\|_{p,\ell-1}<\infty$.
Let $\psi_n=\psi(\tau_n)\psi\Bigl(2s_n\det\sigma_{\calx_n}^{-d}\>K_n\Bigr)$. Obviously, $\psi_n\in\cap_{p>1}\bbD_{p,\ell-1}$;
in order to prove it,
replace $\sigma_{\calx_n}$ by $\sigma_{\calx_n}+k^{-1}I_d$,
differentiate, and take limits in $L^p$-spaces as $k\iku\infty$.
Furthermore, we infer that $\sup_n\|\psi_n\|_{p,\ell-1}<\infty$ for
every $p>1$.
If $\psi_n>0$, then $\det(\sigma_{\calx_n}^{-1}\sigma_{\calz_n})\geq 1/2$ leading to
\bea\label{180928-1}
\det\sigma_{\calz_n}\geq\half\det\sigma_{\calx_n}.
\eea

By applying the IBP formula and the non-degeneracy
assumption for $\calx_n$ under truncation,
we find that
$\sup_n
\bigl|\E\bigl[e^{{\ii}u\cdot\calx_n}\psi_n\zeta(\calx_n)\bigr]\bigr|
\leq
\frac{{\sf C}}{1+|u|^{\ell-1}}$ 
for all $u\in\bbR^d$. 
Combined with condition (6), this implies that
$F(n)=\bar{O}(s_n^2)=\bar{o}(s_n)$. Besides,
\begin{equation}\label{1-psi}
\Bigl|\E\Bigl[\varphi(\calx_n)(1-\psi_n)\zeta(\calx_n)\Bigr]\Bigr|
\leq{\sf C}_q\|1-\psi_n\|_q=\bar{o}(s_n).
\end{equation}
Here $q$ is arbitrary constant such that $q\in(0,1)$.
Consequently, $A(n)=\bar{o}(s_n)$.

Taking the limit of
$
\sup_n
\bigl|
\E\bigl[\zeta(-\ii\partial_u)(e^{{\ii}u\cdot y}{\ii}u)\Big|_{y=\calx_n}\cdot
\caly_n\bigr]
\bigr|
\leq
\frac{{\sf C}}{1+|u|^{\ell-2}},
$
we get
\beas
\Bigl|
\E\Bigl[\zeta(-\ii\partial_u)(e^{{\ii}u\cdot y}{\ii}u)\Big|_{y=\calx_\infty}\cdot
\caly_\infty\Bigr]
\Bigr|
&\leq&
\frac{{\sf C}}{1+|u|^{\ell-2}}
\eeas
for all $u\in\bbR^d$.
On the other hand, from the IBP formula in view of the uniform nondegeneracy of $\calz_n$  under truncation
deduced from that of $\calx_n$ by (\ref{180928-1}),
it follows that $\sup_n|\hat{g}_n(u)|\leq\frac{{\sf C}}{1+|u|^{\ell-1}}$ for all $u\in\bbR^d$.
{From} these estimates, we have
$E(n)=\bar{O}(s_n^2)=\bar{o}(s_n)$.
Similar argument yields the estimate $\sup_n|D(n)|<\infty$.

To obtain $C(n)=\bar{o}(1)$, we apply Lebesgue's dominated convergence theorem
in conjunction with the estimate
\beas
\sup_n
\Bigl|\E\Bigl[e^{{\ii}u\cdot\calz_n}\psi_n\partial\zeta(\calx_n)[\caly_n]\Bigr]
-\E\Bigl[e^{{\ii}u\cdot\calx_\infty}\partial\zeta(\calx_\infty)[\caly_\infty]
\Bigr]
\Bigr|
\leq
\frac{{\sf C}}{1+|u|^{\ell-1}}
\eeas
for all $u\in\bbR^d$.
In the same way, we can obtain $B(n)=\bar{o}(1)$.
However, we have to use more elaborately
the estimate
\beas
\sup_n \>1_{\Lambda_n}(u)
\Bigl| \E\Bigl[e^{{\ii}u\cdot\calx_n}{\ii}u\cdot\caly_n
\>\int_0^1\exp({\ii}s_nu\cdot\caly_ns)ds\>\psi_n\zeta(\calx_n)\Bigr] \Bigr|
\leq
\frac{{\sf C}}{1+|u|^{\ell-2}}
\eeas
(${\sf C}$ is independent of $u$) and its limiting version
$\bigl| \E\bigl[e^{{\ii}u\cdot\calx_\infty}{\ii}u\cdot\caly_\infty
\zeta(\calx_\infty)\bigr] \bigr|
\leq
\frac{{\sf C}}{1+|u|^{\ell-2}}$.

Combining all the estimates, we get
$\int_{\bbR^d}f(x)\tilde{p}_n(x)\>dx
- \int_{\bbR^d}f(x)h_n^0(x)\>dx=\bar{o}(s_n)$
as $n\iku\infty$.
{From} the definition of $h_n^0(x)$,
it is easy to show that
$\calx_\infty$ has a differentiable density $p^{\calx_\infty}$ and that
$h_n^0(x)=
\frac{d\Psi_n}{dx}(x)-s_n
\mbox{div}
\Bigl\{\E[\caly_\infty\>|\>\calx_\infty=x]p^{\calx_\infty}(x)\Bigr\}$.
The existence of the integral
$\int_{\bbR^d}f(x)h_n^0(x)dx$ is ensured
as a consequence under the assumptions of Theorem~\ref{ThmSY}.
\end{proof}

\section{Convergence of martingales and quadratic variations}\label{App0.5}

This section collects the proofs of technical results stated in 
Section~\ref{071219-2}. The major part of them make use of stochastic analysis 
and aim at controlling quadratic variations and covariations of some martingales. 

\begin{proof}[Proof of Lemma~\ref{lem-decomp}]
Let us denote by $\Phi^{1}_n$ the difference $b_n^{-1/2}(\hat{\theta}_n-\theta)-M^n_T$ and write it in the form
$\Phi^{1}_n =b_n^{-1/2}(\Phi^2_n+\Phi^3_n)$, where
\begin{align*}
\Phi^2_n&=\sum_{I\!,J}\>\KIJ \Big(\{(I\beta_1)\cdot t\}_T \times
\{(J\sigma_2)\cdot B_2\}_T+\{(J\beta_2)\cdot t)\}_T \times
\{(I\sigma_1)\cdot B_1\}_T\Big),\\
\Phi^3_n&=\sum_{I\!,J}\>\KIJ \{(I\beta_1)\cdot t\}_T \times
\{(J\beta_2)\cdot t\}_T
\end{align*}
Since we will be interested in applying martingale limit theorems, it is convenient to decompose $\Phi_n^\ell$s in
a sum of a martingale and a bounded variation process. This is achieved by the It\^o formula,
\begin{align*}
\Phi^2_n &= \sum_{I\!,J}\> \KIJ \Big(\Bigl\{
\{((I\beta_1)\cdot t)(J\sigma_2)\}\cdot B_2\Bigr\}_T
+\Bigl\{ \{((J\beta_2)\cdot
t)(I\sigma_1)\}\cdot B_1\Bigr\}_T\Big)\\
&\qquad+\sum_{I\!,J}\> \KIJ \Big(\Bigr\{
\{((J\sigma_2)\cdot B_2)(I\beta_1)\}\cdot t\Bigr\}_T
+\Bigr\{ \{((I\sigma_1)\cdot
B_1)(J\beta_2)\}\cdot t\Bigr\}_T\Big).
\end{align*}
The last two terms in this expression need some further analysis. Let us introduce the notation
$
\Phi^{21}_n =\sum_{I\!,J}\>
\KIJ \big\{ \{((J\sigma_2)\cdot B_2)(I\beta_1)\}\cdot
t\big\}_T$.
Since $((J^j\sigma_2)\cdot B_2)_s=0$ for $s\in(0,T^{j-1})$, 
\bea\label{180803-1}
\Phi^{21}_n &=& \sum_{i,j}\> K_{ij}\Bigr\{
\{((J^j\sigma_2)\cdot B_2)I^i [\beta_{1,T^{j-1}}
+(1_{(T^{j-1},\infty)}\beta^{[0]}_1)\cdot t \nn\\
&&
+(1_{(T^{j-1},\infty)}\beta^{[1]}_{11})\cdot B_1
+(1_{(T^{j-1},\infty)}\beta^{[1]}_{12})\cdot B_2] \}\cdot t\Bigr\}_T
\nn\\
&=& \sum_{i,j}\> K_{ij}\big(\big\{(J^j\sigma_2)\cdot B_2\big\}I^i\beta_{1,T^{j-1}}\big)\cdot t_T +
\sum_{I\!,J}\>\KIJ \bigl(I\{(J\sigma_2\beta^{[1]}_{11})\cdot \langle B_2,B_1
\rangle\} \bigr)\cdot t_T \nn\\
&& +\sum_{I\!,J}\>\KIJ \bigl(I\{(J\sigma_2\beta^{[1]}_{12})\cdot \langle B_2,B_2
\rangle\} \bigr)\cdot t_T +o_P(b_n).
\eea
Let us explain how the last $o_P(b_n)$ is obtained. In fact, the remainder term in the last equation
contains five summands which can all be treated similarly. Let us do it for one of them, which has the form
$\Psi_n=\sum_{I\!,J}\>\KIJ  \Bigl\{
\big(\big[\big\{(J\sigma_2)\cdot B_2\big\}\beta^{[1]}_{11}\big]\cdot B_1\big)I
\Bigr\}\cdot t_T$. We first use that
\begin{align*}
\Psi_n&=\sum_{I}\Bigl\{\big[\big(\big\{(J(I)\sigma_2)\cdot B_2\big\}\beta^{[1]}_{11}\big]\cdot B_1\big)I\Bigr\}\cdot t_T\\
&=\int_0^T\Big(\int_0^s\Big[\int_s^T\sum_{I} \1_I(t)\1_{J(I)}(u)\,dt\Big]\sigma_{2,u}\,dB_{2,u}\Big)\beta_{11,s}^{[1]}\,dB_{1,s}.
\end{align*}
Then, by the Cauchy-Schwarz inequality
and the martingale property of the stochastic integral, we get
\begin{align*}
\Ex^\Pi[\Psi_n^2]&\le \int_0^T\Big(\int_0^s\Big[\int_s^T\sum_{I} \1_I(t)\1_{J(I)}(u)\,dt\Big]^2\sigma_{2,u}^2\,du\Big)\big(\Ex^\Pi[(\beta_{11,s}^{[1]})^4]\big)^{1/2}\,ds\\
&\le C\int_0^T\int_0^s\Big[\sum_{I} |I|\1_{J(I)}(u)\1_{J(I)}(s)\Big]^2\,du\,ds\le C\sum_{I,I'} |I|\,|I'|\,|J(I)\cap J(I')|^2\le C r_n^3
\end{align*}
under the assumption that $\max_{t\in[0,T]}\Ex[(\beta_{11,t}^{[1]})^4]$ and $\max_{t\in[0,T]} \sigma_{2,t}$ are finite.
Now, interchanging the order of integrations, the first summand in the RHS of (\ref{180803-1}) can be rewritten as follows
\begin{align}\label{180803-2}
\{((J\sigma_2)\cdot B_2)I\beta_{1,T^{j-1}}\}\cdot t_T &= \Bigl\{
\{(S^{i}-S^{i-1}\vee\cdot)_+J\sigma_2\beta_{1,T^{j-1}}\}\cdot
B_2 \Bigr\}_T.
\end{align}
Using the same kind of arguments, one can check that the term $\Phi_n^{22}=\Phi_n^{2}-\Phi_n^{21}$ admits the representation
\begin{align}\label{180803-3}
\Phi^{22}_n &:= \sum_{I\!,J} \KIJ \Bigr\{ \{((I\sigma_1)\cdot
B_1)(J\beta_2)\}\cdot t\Bigr\}_T \nn\\
&= \sum_{i,j}K_{ij}\big(\big\{\big((T^j-T^{j-1}\vee\cdot)_+I\sigma_1\beta_{2,S^{i-1}}\big)\!\cdot\! B_1\big\}_T
\nonumber\\
&\quad +\sum_{I\!,J}\>
\KIJ \bigl(J\big\{(I\sigma_1\beta^{[1]}_{22})\!\cdot\! \langle B_2,B_1
\rangle +(I\sigma_1\beta^{[1]}_{21})\cdot \langle B_1,B_1
\rangle\big\} \bigr)\cdot t_T +o_p(b_n).
\end{align}
Combining (\ref{180803-1})-(\ref{180803-3}) and using that $\langle B_1,B_1
\rangle_t=\langle B_2,B_2\rangle_t=t$, $\langle B_2,B_1\rangle_t=\int_0^t\rho_s\,ds$ we get the
desired result.
\end{proof}

\begin{proof}[Proof of Lemma~\ref{lem5.6}] We will prove only the first relation, the proof of the second being quite similar.
Consider the case $\nu=1$, the case $\nu=2$ can be treated similarly
in view of the relation $\langle B_1,B_2
\rangle_t=\int_0^t\rho_s\,ds$ and the boundedness of $\rho$. To
simplify subsequent formulae, let us denote
$\xi^{[11]}=b_n^{-1/2}\sum_{ij}\KIJ \{(J\sigma_2\cdot
B_2)I\sigma_1\}\cdot\langle B_1,B_1 \rangle$. In other words,
$\xi^{[11]}$ is a random process indexed by $t\in[0,T]$ defined by
\begin{align*}
b_n^{1/2}\xi_t^{[11]}&=\sum_{I\!,J} \KIJ \int_0^t
\1_{I}(u)\sigma_{1,u}\int_0^u\1_{J}(s)\sigma_{2,s}\,dB_{2,s}\,du\\
&=\sum_{I\!,J} \KIJ \int_0^t\1_{J}(s)\sigma_{2,s}
\int_s^t\1_{I}(u)\sigma_{1,u}\,du\,dB_{2,s}\\
&= \int_0^t\sum_{J}\1_{J}(s)\sigma_{2,s}
\int_s^t\1_{I(J)}(u)\sigma_{1,u}\,du\,dB_{2,s}.
\end{align*}
The latter expression implies that conditionally to $\Pi_{n}$,
$\xi^{[11]}$ is a Gaussian process with zero mean. Moreover,
\begin{align*}
\Ex^{\Pi}[(\xi^{[11]}_t)^2]&=b_n^{-1}\sum_J\int_0^t\1_{J}(s)\sigma_{2,s}^2
\bigg(\int_s^t\1_{I(J)}(u)\sigma_{1,u}\,du\bigg)^2\,ds\\
&\leq b_n^{-1}\|\sigma_1^2\|_\infty\|\sigma_2^2\|_\infty\sum_J
|J||I(J)|^2\leq C b_n^{-1}r_n^2,
\end{align*}
where $C$ is a positive constant. This yields the desired result.
\end{proof}

\begin{proof}[Proof of Lemma~\ref{lem8}]
One easily checks that
\begin{align}\label{26}
\int_0^t\bbH^{1,n}_s\bbH^{2,n}_s\,d\langle B_1,B_2 \rangle_s=
\int_0^t\sum_{I\!,J}h_s\tilde{K}^n_{ij}(s)(J\sigma_2\cdot
B_2)_s(I\sigma_1\cdot B_1)_s\,ds.
\end{align}
To prove the convergence of this expression, we apply the It\^o
formula to the product $(J\sigma_2\cdot B_2)_s(I\sigma_1\cdot
B_1)_s$:
$$
(J\sigma_2\cdot B_2)_s(I\sigma_1\cdot B_1)_s=\{(J\sigma_2\cdot
B_2)I\sigma_1\cdot B_1\}_s+\{(I\sigma_1\cdot
B_1)J\sigma_2\cdot B_2\}_s+\{(IJh)\cdot t\}_s.
$$
One can show that the contribution of the first two terms is
asymptotically negligible, that is
\begin{align}\label{27}
\int_0^t\sum_{I\!,J}h_s\tilde{K}^n_{ij}(s)(\{(J\sigma_2\cdot
B_2)I\sigma_1\cdot B_1\}_s+\{(I\sigma_1\cdot
B_1)J\sigma_2\cdot B_2\}_s)\,ds&\xrightarrow[n\to\infty]{p} 0,
\end{align}
Thus, the main term is
\begin{align}\label{28}
\int_{0}^t\sum_{I\!,J}h_s\tilde{K}^n_{ij}(s)\{(IJh)\cdot
t\}_s\,ds.
\end{align}
To prove (\ref{27}), we show the convergence in $L^2$. More
rigorously, using the notation
$\check{K}_{IJ}^{n}(s)=\int_s^{t}\tilde \KIJ ^n(s)h_s\,ds$ and interchanging the order of integrals, we get
\begin{align*}
\Ex^\Pi\Big(\int_0^t&\sum_{I\!,J}h_s\tilde{K}^n_{ij}(s)\{(J\sigma_2\cdot
B_2)I\sigma_1\cdot B_1\}_s\,ds\Big)^2=\Ex^\Pi\Big(\Big\{\sum_{I\!,J} \check{K}_{ij}^n(J\sigma_2\cdot
B_2)I\sigma_1\cdot B_1\Big\}_t\Big)^2\\
&\le \int_0^t\Ex^\Pi\Big[\Big(\sum_{I\!,J}
\check{K}_{I\!J}^n(u)(J\sigma_2\cdot B_2)_uI_u\sigma_{1,u}\Big)^2\Big]\,du\\
&=\int_0^t\int_0^u\Big(\sum_{I\!,J}
\check{K}_{I\!J}^n(u)I_u\sigma_{1,u}J_v\sigma_{2,v}\Big)^2dvdu\\
&\le C b_n^{-2}\int_0^T\int_0^Tr_n^2\sum_{I\!,J}\Big(K_{I\!J}\1_{I}(u)\1_{J}(v)\Big)\,dvdu
\le Cb_n^{-2}r_n^3.
\end{align*}
Let us show now that the term (\ref{28}) converges in probability.
Simple algebra allows us to rewrite that term in the form
\begin{align*}
\frac1{2b_n}\sum_{I}\Big(\int_0^t h_s\1_I(s)\,ds\Big)^2+
\frac1{2b_n}\sum_{J}\Big(\int_0^t h_s\1_J(s)\,ds\Big)^2-
\frac1{2b_n}\sum_{I\!,J}\Big(\int_0^t h_s\1_{I\cap J}(s)\,ds\Big)^2,
\end{align*}
which in turn is nothing else but
$
\int_{[0,t]^2} h_sh_{s'}\1_{\{s\vee s'\le
t\}}\,\{\calv^{I}_{n}+\calv^{J}_{n}-\calv^{I\cap
J}_{n}\}(ds,ds').
$
The weak convergence of measures stated in Assumption \textsf{P1} completes the proof of the first assertion.
The proof of the second assertion is quite similar and therefore is omitted.
\end{proof}

\begin{proof}[Proof of Lemma~\ref{lem3.8}]
Using the representations of $M^n$ and ${N}^n$ as stochastic
integrals, we get
\vglue-20pt
\begin{equation}\label{29}
\langle
M^n,N^n\rangle_t=\int_0^t\big(\bbH^{1,n}_s\bbG^{1,n}_s +
\bbH^{1,n}_s\bbG^{2,n}_s\rho_s+\bbH^{2,n}_s\bbG^{1,n}_s\rho_s+\bbH^{2,n}_s\bbG^{2,n}_s\big)\,ds.
\end{equation}
Let us denote by $\bbG^{11,n}$ the first summand
$b_n^{-1}\sum_{I\!,J}\> \KIJ \{((J\beta_2)\cdot t)(I\sigma_1)\}$
in $\bbG^{1,n}$ and let us show that
$\int_0^t\bbH^{1,n}_s\bbG^{11,n}_s\,ds$ tends to zero in probability
as $n\to\infty$. Simple algebra yields
\begin{align*}
\int_0^t\bbH^{1,n}_s\bbG^{11,n}_s\,ds&= b_n^{-3/2}\int_0^t\sum_I
I_s\sigma_{1,s}^2\int_0^s J(I)_u\beta_{2,u}\,du\int_0^s J(I)_u\sigma_{2,u}\,dB_{2,u}\,ds\\
&=b_n^{-3/2}\int_0^t\sum_I
I_s\sigma_{1,s}^2\beta_{2,a_{J(I)}}(s-a_{J(I)})\int_0^s J(I)_u\sigma_{2,u}\,dB_{2,u}\,ds\\
&\quad +b_n^{-3/2}\!\!\int_0^t\sum_I I_s\sigma_{1,s}^2\!\!\int_0^s
\!J(I)_u(\beta_{2,u}-\beta_{2,a_{J(I)}})\,du\! \int_0^s\!
J(I)_u\sigma_{2,u}\,dB_{2,u}\,ds\\
&:=\calt_{1,n}+\calt_{2,n},
\end{align*}
where we denoted by $a_{J(I)}$ the left endpoint of the interval $J(I)$.
Let us show that both $\calt_{1,n}$ and $\calt_{2,n}$ tend to
zero in probability. Indeed,
\begin{align*}
\Ex^\Pi[\calt_{1,n}^2]&=b_n^{-3}\Ex^\Pi\Big[\int_0^t\sigma_{2,u}^2\Big(\sum_I
J(I)_u\int_u^tI_s\sigma_{1,s}^2(s-a_{J(I)})\,ds\beta_{2,a_{J(I)}}\Big)^2du\Big]\\
&\le Cb_n^{-3}\Ex^\Pi\Big[\int_0^t\Big(\sum_I J(I)_u
|I||J(I)||\beta_{2,a_{J(I)}}|\Big)^2du\Big]
\le Cb_n^{-3}r_n^4\sup_{t\in[0,T]} \Ex[\beta_{2,t}^2],
\end{align*}
and, after applying the Cauchy-Schwarz inequality several times,
\begin{align*}
\Ex^\Pi[\calt_{2,n}^2]&\le b_n^{-3}\Ex^\Pi\Big[\int_0^t \sum_I
I_s\sigma_{1,s}^4\Big(\int_0^s
J(I)_u(\beta_{2,u}-\beta_{2,a_{J(I)}})\,du \int_0^s
J(I)_u\sigma_{2,u}\,dB_{2,u}\Big)^2\,ds\Big]\\
&\le Cb_n^{-3}\Ex^\Pi\Big[\int_0^t \sum_I I_s |J(I)|^{4}\,ds\Big]\le
b_n^{-3}r_n^4.
\end{align*}
Similar arguments yield the convergence to zero of the sequence
$\Ex[(\int_0^t\bbH^{1,n}_s\bbG^{12,n}_s\,ds)^2]$. Thus
$\int_0^t\bbH^{1,n}_s\bbG^{1,n}_s\,ds$ tends to zero in probability
as $n\to\infty$. The convergence to zero of the other terms of the sum in the right-hand side of (\ref{29})
can be shown similarly. 
\end{proof}

\begin{proof}[Proof of Lemma~\ref{lem3.9}]
Let us prove the first assertion, the proof of the second one being completely similar.
Since $N^n=\bbG^{1,n}\cdot B_1+\bbG^{n,2}\cdot B_2$ with $\bbG^{1,n}$ and $\bbG^{2,n}$ defined in Lemma \ref{lem-decomp}, we have
$\langle N^n,B_1\rangle_t=\int_0^t(\bbG^{1,n}_s+\bbG^{n,2}_s\rho_s)\,ds$. It is easily seen that
\begin{align*}
\int_0^t \bbG^{1,n}_s\,ds&=b_n^{-1}\sum_{i,j}\> K_{ij}\int_0^t \Big(((J^j\beta_2)\cdot t)_sI^i_s\sigma_{1,s}
+(T^{j}-T^{j-1}\vee s)_+I^i_s\sigma_{1,s}\beta_{2,S^{i-1}}\Big)\,ds\\
&=\int_0^t\int_0^t \big(\beta_{2,u}\sigma_{1,s}\1(u\le s)
+\1(u> s)\sigma_{1,s}\beta_{2,s}\big)\,\calv_{n}^{I\!,J}(ds,du)\\
&\qquad +b_n^{-1}\sum_{i,j}\> K_{ij}\int_0^t (T^{j}-T^{j-1}\vee s)_+ I^i_s\sigma_{1,s}(\beta_{2,S^{i-1}}-\beta_{2,s})\,ds.
\end{align*}
Since $\beta_2$ is an It\^ o process with $\beta_2^{[0]}$, $\beta_{21}^{[1]}$ and $\beta_{22}^{[1]}$
being uniformly bounded in $L^2$-norm, the expectation $\Ex^\Pi[|\beta_{2,S^{i-1}}-\beta_{2,s}|]$ is bounded up to a constant
factor by $|I|^{1/2}$.  This implies that the second term in the last formula is $o_p(b_n^{-1}\sum_{I\!,J} \KIJ |I|^{3/2}|J|)=o_p(r_n^{3/2}b_n^{-1})$,
while the first term converges to $\int_0^t \beta_{2,s}\sigma_{1,s}\,\calv^{I\!,J}(ds)$ in view of Assumption \textsf{P1}.

Identical arguments imply the convergence of  $\int_0^t \bbG^{n,2}_s\rho_s\,ds$ to $\int_0^t \beta_{1,s}\sigma_{2,s}\rho_s\calv^{I\!,J}(ds)$
and the assertion of the lemma follows.
\end{proof}

\begin{proof}[Proof of Lemma~\ref{lem3.10}]
Since $N^n=\bbG^{1,n}\cdot B_1+\bbG^{2,n}\cdot B_2$, its quadratic variation is given by
$\langle N^n,N^n\rangle=\big[(\bbG^{1,n})^2+2\bbG^{1,n}\bbG^{2,n}\rho+(\bbG^{2,n})^2\big]\cdot t$.
Using the semimartingale decomposition of $\beta_2$, one checks that
\begin{align*}
\int_0^t(\bbG^{1,n}_s)^2\,ds&= b_n^{-2}\int_0^t \Big(\sum_{I\!,J}\KIJ\int_0^tJ_u\beta_{2,u}\,du\,I_s\sigma_{1,s}\Big)^2ds
+o_p(r_n^3b_n^{-2})\\
&= b_n^{-2}\int_0^t \sum_I I_s\sigma_{1,s}^2\Big(\int_0^tJ(I)_u\beta_{2,u}\,du\Big)^2ds
+o_p(r_n^3b_n^{-2})\\
&= \int_{[0,t]^3} \sigma_{1,s}^2\beta_{2,u}\beta_{2,u'}\,\calv_{n}^{I\!,J,J'}(ds,du,du')
+o_p(r_n^3b_n^{-2}).
\end{align*}
Analogous computations show that
\begin{align*}
\int_0^t(\bbG^{2,n}_s)^2\,ds
&= \int_{[0,t]^3} \beta_{1,s}\beta_{1,s'}\sigma_{2,u}^2\,\calv_{t}^{I,I'\!\!,J}(ds,ds',du)
+o_p(r_n^3b_n^{-2}),\\
\int_0^t\bbG^{1,n}_s\bbG^{2,n}_s\rho_s\,ds
&= \int_{[0,t]^3} \beta_{2,u}\beta_{1,s'}\sigma_{1,s}\sigma_{2,s}\rho_s\,\calv_{t}^{J(I),I(J),I\cap J}(du,ds',ds)
+o_p(r_n^3b_n^{-2}).
\end{align*}
Now, the desired result follows from Assumption \textsf{P2}.
\end{proof}

\section{Technical results on Poisson point processes}\label{App1}
\begin{lemma}\label{lemA1}
For every $\lambda>0$ it holds that
$\sum_{k=0}^\infty \frac{\lambda^k}{k!(k+2)}=
\lambda^{-2}({\lambda e^{\lambda}-e^\lambda+1}).$
\end{lemma}
\begin{proof} It follows from the equality $1/(k!(k+2))=1/((k+1)!)-1/((k+2)!)$
and the power series expansion of the exponential function.
\end{proof}

\begin{lemma}\label{lemA 1.5}
Let $\mP$ be a homogeneous Poisson point process on $\bbR$ with
intensity $\lambda>0$ and let $a\in\bbR$.
For every $\omega$, let $I_a(w)$ be the interval that contains $a$ and that is an element of
the partition of $\bbR$ generated by $\mP$. Then
$|I_a|$ is distributed according to the law $\text{\rm Gamma}(2,\lambda)$.
\end{lemma}
\begin{proof}
W.l.o.g.\ we can assume that $a=0$. Since the restrictions of $\mP$ on $(-\infty,0)$ and $[0,\infty)$
are two independent Poisson processes, the law of $|I_a|$ coincides with the law of the sum of
two i.i.d.\ random variables exponentially distributed with parameter $\lambda$.
Thus the assertion of the lemma follows from the well known properties of the Gamma distribution.
\end{proof}

\begin{lemma}\label{lemA2}
Let $\mP$ be a homogeneous Poisson point process on $\bbR$ with
intensity $\lambda>0$ and let $I=[a,b]\subset\bbR$ be some interval.
For every $\omega$, let us denote by $N=N(\omega)$ the number of
points of $\mP(\omega)$ lying in $I$ and by $t_i=t_i(\omega)$,
$i=1,\ldots,N$ the ordered sequence of these points. Then
$$
\Ex\bigg[\sum_{i=0}^{N} (t_{i+1}-t_i)^2\bigg]=
\frac{2(|I|\lambda-1+e^{-|I|\lambda})}
{\lambda^2},
$$
where we used $t_0=a$ and $t_{N+1}=b$.
\end{lemma}
\begin{proof}Without loss of generality, we assume that $I=[0,1]$.
We use the fact that conditionally to $N(\omega)=k$, the random
vector $(t_1,\ldots,t_k)$ have the same distribution as
$(U_{(1)},\ldots,U_{(k)})$, where $U_1,\ldots,U_k$ are independent
uniformly in $[0,1]$ distributed random variables and $U_{(1)}$,
$\ldots$, $U_{(k)}$ are the corresponding order statistics. Since
the joint density of $(U_{(i)},U_{(i+1)})$ is given by
$$
f_{(U_{(i)},U_{(i+1)})}(x,y)=\frac{k!}{(i-1)!(k-i-1)!}\;x^{i-1}(1-y)^{k-i-1}
\1_{\{x\le y\}},
$$
the expectation $\Ex[(U_{(i+1)}-U_{(i)})^2]$ is equal to $2/[(k+1)(k+2)]$.
It is easily seen that $\Ex[U_{(1)}^2]=\Ex[(1-U_{(k)})^2]=2/[(k+1)(k+2)]$.
Therefore,
$$
\Ex\bigg[\sum_{i=0}^{N} (t_{i+1}-t_i)^2\bigg]=\sum_{k=0}^\infty
\bigg(\sum_{i=0}^k\frac{2}{(k+1)(k+2)}\bigg)\Pb(N=k)=
\sum_{k=0}^\infty \frac{2e^{-\lambda}\lambda^k}{k!(k+2)}.
$$
The desired result follows now from Lemma~\ref{lemA1}.
\end{proof}

\begin{lemma}\label{lemA3}
Let $\zeta_1\sim\msE(\lambda_1)$ and $\mP^2$ be a Poisson process
with intensity $\lambda_2$ independent of $\zeta_1$. Let us denote by
$\Pi^2_\zeta$ the partition of $[0,\zeta_1]$ generated by $\mP^2$.
Then
$$
\Ex\bigg[\zeta_1\sum_{J\in\Pi^2_\zeta} |J|^2\bigg]=\frac{6\lambda_1+4\lambda_2}{\lambda_1^2
(\lambda_1+\lambda_2)^2}.
$$
\end{lemma}
\begin{proof}
By rescaling and by using Lemma~\ref{lemA2}, we get
$$
\Ex\bigg[\sum_{J\in\Pi^2_\zeta} |J|^2\;\big|\,\zeta_1\bigg]=
\frac{2\zeta_1^2(\lambda_2\zeta_1-1+e^{-\lambda_2\zeta_1})}{\lambda_2^2\zeta_1^2}.
$$
Therefore,
\begin{align*}
\Ex\bigg[\zeta_1\sum_{J\in\Pi^2_\zeta} |J|^2\bigg]&=
\frac{2}{\lambda_2}\Ex[\zeta_1^2]-\frac{2}{\lambda_2^2}\;\Ex[\zeta_1]+\frac{2}{\lambda_2^2}
\;\Ex[\zeta_1e^{-\lambda_2\zeta_1}]\\
&=\frac{4}{\lambda_2\lambda_1^2}-\frac{2}{\lambda_2^2\lambda_1}+
\frac{2}{\lambda_2^2}\frac{\lambda_1}{(\lambda_1+\lambda_2)^2}=\frac{6\lambda_1+4\lambda_2}{\lambda_1^2(\lambda_1+\lambda_2)^2}\ .
\end{align*}
This completes the proof of the lemma.
\end{proof}

\begin{lemma}\label{lemA4}
Let $I=[a,b]$ be an interval of $[0,T]$. If $\mP$ is a Poisson point process with
intensity $\lambda$ and $\Pi$ is the partition of $[0,T]$ generated by $\mP$,
then
\begin{align*}
\Ex\Big[\sum_{J\in\Pi}|J|\KIJ \Big]&=|I|+2\lambda^{-1}-\lambda^{-1}(e^{-\lambda a}+
e^{-\lambda (T-b)}),\\
\Ex \Big[\sum_{J\in\Pi} |J\setminus I|\cdot |J\cap I| \Big]&=
\lambda^{-2}(1-e^{-\lambda |I|}) (2-e^{-\lambda a}- e^{-\lambda (T-b)}).
\end{align*}
\end{lemma}
\begin{proof}
We can consider the Poisson point process $\mP$ on $[0,T]$ as the
union of three independent Poisson point processes: $\mP_a$ on
$[0,a]$, $\mP_I$ on $I=[a,b]$ and $\mP_b$ on $[b,T]$. Let $t_1\le
\ldots\le t_{N_a}$ (resp. $t_1''\le\ldots\le t_{N_b}''$) be the
points of $\mP_a$ (resp. $\mP_b$). Then $\Ex[\sum_J |J|
\KIJ ]=\Ex[(a-t_{N_a})+|I|+(t_1''-b)]$. For every integer $k\ge
0$, conditionally to $N_a=k$, the random variable $t_{N_a}$ has the
same law as the last order statistic $U_{(k)}$ of a sequence
$U_1,\ldots,U_k$ of i.i.d.\ uniformly in $[0,a]$ distributed random
variables. Therefore, $\Ex[a-t_{N_a}|N_a=k]=a/(k+1)$ and
$$
\Ex[a-t_{N_a}]=\sum_{k=0}^\infty \frac{(a\lambda)^ka}{k!(k+1)}\;e^{-a\lambda}
=\frac{1-e^{-a\lambda}}{\lambda}.
$$
The same arguments yield $\Ex[t_1''-b]=\lambda^{-1}(1-e^{-(T-b)\lambda})$
and the first assertion of the lemma follows.
Using the same notation, we have $\sum_J |J\setminus I|\cdot|J\cap
I|= (a-t_{N_a})(t_1'-a)+(b-t_{N_I}')(t_1''-b)$, where
$t_1'\le\ldots\le t_{N_I}'$ are the points of $\mP$ lying in $I$.
Thanks to the conditional independence of $t_{N_a}$, $(t_1',t_{N_I}')$
and $t_1''$ given $N_a$, $N_I$ and $N_b$, as well as the
representation by means of order statistics of the uniform
distribution we get the second assertion of the lemma.
\end{proof}

\begin{lemma}\label{lemA5}
Let $t>0$ and let $\mP$ be a Poisson process on $[0,t]$ with intensity $\lambda$. We denote by $\Pi$
the random partition of $[0,t]$ generated by $\mP$.
For every continuous function $h:[0,t]^2\to\bbR$, it holds that
\begin{align*}
\lambda\sum_{I\in\Pi}\int_{I\times I} h(s,s')\,ds\,ds'\xrightarrow[\lambda\to\infty]{L^1(P)} 2\int_0^t h(s,s)\,ds.
\end{align*}
\end{lemma}

\begin{proof}
Let $K$ be a positive integer and let as denote by
\begin{equation}\label{whdelta}
w_h(\delta)=\max\{|h(s,s')-h(u,u')| : (s,s',u,u')\in [0,T]^4 \text{ and }|s-u|\le \delta, |s'-u'|\le \delta\}
\end{equation}
the modulus of continuity of $h$. Since $h$ is continuous and $[0,t]^2$ is compact, we have $w_h(t/K)\to 0$
as $K\to\infty$.

It holds that $\lambda\sum_{I\in\Pi}\int_{I\times I} h(s,s')\,ds\,ds'=2\int_0^t h(s,s)\,ds+\calt_1+\calt_3+\calt_3 $
with
\begin{align*}
\calt_1&=\lambda\sum_{I\in\Pi}\int_{I\times I} h(s,s')\,ds\,ds'-\lambda\sum_{i=1}^K h\Big(\frac{it}{K},\frac{it}{K}\Big)\sum_{I\in\Pi_i^K} |I|^2,\\
\calt_2&=\sum_{i=1}^K h\Big(\frac{it}{K},\frac{it}{K}\Big)\Big(\lambda\sum_{I\in\Pi_i^K} |I|^2-\frac{2t}{K}\Big),\\
\calt_3&=2\sum_{i=1}^K \frac{t}K\, h\Big(\frac{it}{K},\frac{it}{K}\Big)-2\int_0^t h(s,s)\,ds,
\end{align*}
where $\Pi_i^K$ is the restriction of the Poisson process $\mP$ on the interval $[(i-1)t/K,it/K]$.
For the first term, easy algebra yields
\begin{align*}
\Ex[|\calt_1|]&\le \lambda\|h\|_\infty\Ex\Big[\sum_{I\in \Pi} |I|^2-\sum_{i=1}^K\sum_{I\in\Pi_i^K}|I|^2\Big]+\lambda w_h(t/K)\sum_{i=1}^K
\Ex\Big[\sum_{I\in\Pi_i^K}|I|^2\Big].
\end{align*}
This inequality combined with Lemma \ref{lemA2} implies that
\begin{align*}
\limsup_{\lambda\to\infty} \Ex[|\calt_1|]&\le \limsup_{\lambda\to\infty} \Big(\lambda\|h\|_\infty
\frac{K}{\lambda^2}+\lambda w_h(t/K)\frac{2t}{\lambda}\Big)=2tw_h(t/K).
\end{align*}
In order to bound $\Ex[|\calt_2|]$, we evaluate $\Ex[|\lambda\sum_{I\in\Pi_i^K} |I|^2-\frac{2t}{K}|]$. The value of this term being
independent of $i$, we only evaluate the term corresponding to $i=1$. Let $\{\zeta_j,\;j\in\bbN\}$ be a family of i.i.d.\ exponentially distributed
random variables with scaling parameter one and let $N=\min\{k:\zeta_1+\ldots+\zeta_k\ge npt/K\}$. Then
$$
\Big|\lambda\sum_{I\in\Pi_i^K} |I|^2-\frac{2t}{K}\Big|\le
\frac1{\lambda}\Big|\sum_{j=1}^N (\zeta_j^2-2)\Big|+\Big|\frac{2(N-1)}{\lambda}-\frac{2t}{K}\Big|+
\frac{\zeta_N^2+2}{\lambda}.
$$
Note that $\Ex[\zeta_N^2]=6$ by virtue of Lemma~\ref{lemA 1.5}.
In view of the Cauchy-Schwarz inequality and Wald's identity \cite[Ch.\ VII, Thm. 3, Eq. (15)]{Shir}, we get
$\Ex[|\sum_{j=1}^N(\zeta_j^2-2)|]\le [\var(\zeta_1^2)\,\Ex(N)]^{1/2}=O(\lambda^{1/2})$.
Finally, it is clear that $|\calt_3|\le 2t\,w_h(t/K)$. Putting these estimates together, we
get $\limsup_{\lambda\to\infty} \Ex[|\calt_1+\calt_2+\calt_3|]\le 4tw_h(t/K)$. Using the fact that
$w_h(t/K)$ tends to zero as $K\to\infty$, we arrive at the desired result.
\end{proof}

\begin{lemma}\label{lemA6}
Let $t>0$  and let ${\mP}^i$, $i=1,2$, be two Poisson processes on $[0,t]$ with intensities $\lambda_i$, $i=1,2$.
Let  $\Pi^i$ be the random partition of $[0,t]$ generated by $\mP^i$, $i=1,2$ and let $\lambda_0=\lambda_1\lambda_2/(\lambda_1+\lambda_2)$.
For every continuous function $h:[0,t]^2\to\bbR$ there exists a constant $C>0$ such that for every
$x\in[C\log \lambda_0, C\lambda_0^{1/6}]$ the inequality
\begin{align*}
\Pb\bigg(\Big|\lambda_0\!\sum_{I\!,J}\! \KIJ\int_{I\times J} h(s,s')- 2\int_0^t h(s,s)\,ds\Big|\ge
\frac{x}{\sqrt{\lambda_0}}+Cx\Big(\frac1\lambda_0+w_h\Big(\frac{x}{\lambda_0}\Big)\!\Big)\!\bigg)
\le C\lambda_0 e^{-{x}/{C}}
\end{align*}
holds for sufficiently large $\lambda_0$, with $w_h(\cdot)$ being defined by (\ref{whdelta}).
\end{lemma}

\begin{proof} W.l.o.g.\ we assume that $t=1$.
Set $\calt=\lambda_0\sum_{I\in\Pi^1,J\in\Pi^2} \KIJ\int_{I\times J} h(s,s')\,ds\,ds'$ and $\bar h(s)=h(s,s)$.
Let us denote by $N(x)=\lceil \lambda_0/x\rceil$ the smallest positive
integer such that $N(x)x >\lambda_0$ and let us set
$L_i=[iN(x)^{-1},(i+1)N(x)^{-1}]$. The intervals $L_i$ define a
uniform deterministic partition of $[0,1]$ with a mesh-size of order
$x/\lambda_0$. Let $\cale$ be the event ``for every $i=1,\ldots,4N(x)$, the
interval $[\frac{i}{4N(x)},\frac{(i+1)}{4N(x)}]$ contains at least
one point from $\Pi^1$ and one point from $\Pi^2$''. The total
probability formula implies that
\begin{align*}
\Pb\bigg(\Big|\calt- 2\int_0^1 \bar h(s)ds\Big|\ge \frac{x}{\sqrt{\lambda_0}}\bigg)
\le \Pb\bigg(\Big|\calt- 2\int_0^1 \bar h(s)\,ds\Big|\ge
\frac{x}{\sqrt{\lambda_0}}\;\Big|\;\cale\bigg)+\Pb(\cale^c),
\end{align*}
where $\cale^c$ denotes the complementary event of $\cale$. Easy computations show that, for some $C>0$, the inequality 
$\Pb(\cale^c)\le C \lambda_0 x^{-1} e^{-x/C}$ holds true.

Let now $l_i$ be a point in $L_i$ such that $\int_{L_i}\bar
h(t)\,dt=\bar h(l_i)|L_i|$ and let $a_I$ be the left endpoint of $I$.
We define the random variables
$$
\eta_i^\circ=\lambda_0\bar h({l_i})\sum_{I\!,J} |I||J|\KIJ\1_{\{a_I\in
L_i\}},\quad i=1,\ldots,N(x),
$$
and write $\calt_1=\calt_{11}+\calt_{12}+\calt_{13}+\mO(\lambda_0 |L_1|w_h(|L_1|))$
on $\cale$, where
\begin{align*}
\calt_{11}&=\Ex^\cale\bigg[\sum_{i=1}^{N(x)}\eta_i^\circ\bigg]-2\int_0^1\bar
h(s) \,ds,\quad
\calt_{1s}&=\sum_{i=1}^{[N(x)/2]}(\eta_{2i+s-2}^\circ-\Ex^\cale[\eta_{2i+s-2}^\circ]),\quad s=2,3.
\end{align*}
Let us emphasize that for evaluating the remainder term in $\calt_1$, we have used the
fact that $r=\max_{I\in\Pi^1}|I|\vee\max_{J\in\Pi^2}|J| \le |L_1|/2$ on $\cale$.

On the one hand, since $|\sum_{i=1}^{N(x)}\eta^\circ_i|\le C\lambda_0 r$,
we have
$$
\bigg|\Ex^\cale\Big[\sum_{i=1}^{N(x)}\eta_i^\circ\Big]-\Ex\Big[\sum_{i=1}^{N(x)}\eta_i^\circ\Big]\bigg|
\le \frac{\lambda_0\Ex[r\1_{\cale^c}]}{\Pb(\cale)}.
$$
Using the inequality of Cauchy-Schwarz, as well as the bounds
$\Pb(\cale^c)\le C\lambda_0 e^{-x/C}$ and (\ref{r_np}), we get $\big|
\Ex^\cale \big[ \sum_{i=1}^{N(x)} \eta_i^\circ\big] - \Ex\big
[\sum_{i=1}^{N(x)} \eta_i^\circ\big] \big| \le C\lambda_0 e^{-x/C}$, for some
constant $C$ and for every $x>C\log \lambda_0$.

On the other hand, in view of Lemma~\ref{lemA4}, we have
$$
\Ex[\eta_i^\circ]\le \lambda_0\bar h(l_i)\Ex\bigg[\sum_{I:a_I\in
L_i}\Big(|I|^2+\frac{2|I|}{\lambda_2}\Big)\bigg]\le C
\lambda_0\bar h(l_i)(\lambda_1^{-1}+\lambda_2^{-1})|L_i|=\mO(x\lambda_0^{-1}).
$$
Using once again Lemma~\ref{lemA4}, we get
\begin{align*}
\Ex\Big[\sum_{i=1}^{N(x)}\eta_i^\circ\Big]&=\sum_{i=2}^{N(x)-1}
\lambda_0\bar h(l_i)\Ex\Big[\sum_{I:a_I\in
L_i}|I|\cdot\Ex^{\Pi^1}\Big(\sum_{j\in\Pi^2}\KIJ |J|\Big)\Big]+\mO(x\lambda_0^{-1})\\
&=\sum_{i=2}^{N(x)-1} \lambda_0\bar h(l_i)\Ex\Big[\sum_{I:a_I\in
L_i}\big(|I|^2+2|I|/\lambda_2\big)\Big]+\mO(x\lambda_0^{-1}).
\end{align*}
Wald's equality yields
\begin{align}
\label{boundIk}
\Ex\Big[\sum_{I:a_I\in L_i} |I|^k\Big]=k!|L_i|\lambda_1^{1-k}+\mO(\lambda_1^{-k}),
\end{align}
for every $k>0$ and for every $i\le N(x)-1$. Putting all these estimates together, we get
\begin{align*}
\Ex\Big[\sum_{i=1}^{N(x)}\eta_i^\circ\Big]&=\sum_{i=2}^{N(x)-1}
n\lambda_0\bar h(l_i)\Big(\frac{2|L_i|}{\lambda_1}+\frac{2|L_i|}{\lambda_2}\Big)+\mO(x\lambda_0^{-1})
=\sum_{i=1}^{N(x)} 2\bar h(l_i)|L_i|+\mO(x\lambda_0^{-1}).
\end{align*}
Since $l_i$ is chosen to verify $\bar h(l_i)|L_i|=\int_{L_i}
\bar h(t)\,dt$, we get
$\calt_{11}=\mO(x\lambda_0^{-1})$.

The advantage of working with $\eta^\circ_i$s is that, conditionally
to $\cale$, the random variables $\eta^\circ_{2i}$,
$i=1,\ldots,[N(x)/2]$, are independent. Indeed, one easily checks
that conditionally to $\cale$, $\eta^\circ_{2i}$ depends only on the
restrictions of $\mP^{1}$ and $\mP^{1}$ onto the interval
$[\frac{(4i-1)}{2N(x)},\frac{(4i+3)}{2N(x)}]$. Since these
intervals are disjoint for different values of $i\in\bbN$, the
restrictions of Poisson processes $\mP^{k}$, $k=1,2$, onto these
intervals are independent. Therefore, $\eta_{2i}^\circ$,
$i=1,\ldots,[N(x)/2]$, form a sequence of random variables that are
independent conditionally to $\cale$. Moreover, conditionally to $\cale$,
they verify $|\eta_i^\circ|\le C\lambda_0 r |L_i|\le C x^2/\lambda_0$.
One can also check that $\Ex^\cale[(\eta_i^\circ)^2]=\mO(x^2\lambda_0^{-2})$.

These features enable us to use the Bernstein inequality in order to
bound large deviations of $\calt_{12}$ as follows:
\begin{align*}
\Pb^\cale\big(|\calt_{12}|\ge x/\sqrt{\lambda_0}\big)&\le
2\exp\Big({-\frac{x^2/(2\lambda_0)}{C(N(x)x^2\lambda_0^{-2}+x^3\lambda_0^{-3/2})}}\Big)\le 2e^{-x/C},\quad
\forall\,x\in [1,\lambda_0^{1/6}]
\end{align*}
Obviously, the same inequality holds true for the
term $\calt_{13}$. These inequalities combined with the bound on the
deterministic error term $\calt_{11}$ complete the proof.
\end{proof}

\begin{lemma}\label{lemA7}
Let $T>0$  and let $\mP_n^i$, $i=1,2$, be two Poisson processes on $[0,T]$ with intensities $n p_i$, $i=1,2$.
For every continuous function $h:[0,T]^3\to\bbR$, it holds that
\begin{align*}
n^2\sum_{I\in\Pi^1_n}\int_{I\times J(I)\times J(I)} h(s,t,u)\,ds\,dt\,du\xrightarrow[n\to\infty]{P} \Big(\frac{6}{p_1^2}+\frac{8}{p_1p_2}+\frac{6}{p_2^2}\Big)
\int_0^T h(s,s,s)\,ds.
\end{align*}
\end{lemma}
\begin{proof}
Let us denote $\calt_n=n^2\sum_{I\in\Pi^1_n}\int_{I\times J(I)\times J(I)} h(s,t,u)\,ds\,dt\,du$ and let us consider the
uniform partition $\{L_i=\big[(i-1)/N, i/N\big),\,i=1,\ldots,N\}$ with $N=[n^{1-\eps}]$ slightly smaller than $n$ ($\eps$
is a small positive number). For every integer $i$ smaller than $[n^{1-\eps}]$, we define $l_i$ as the real number such that $\bar h(l_i)=|L_i|^{-1}\int_{L_i} \bar h_s\,ds$, where $\bar h_s=h(s,s,s)$. The continuity of $h$ implies that
$$
\calt_n=n^2(1+o(1))\sum_{i=1}^{N}\sum_I \bar h(l_i) |I|\,|J(I)|^2\1_{L_i}(a_I).
$$
For every $i$, we set $\eta_i^\circ=n^2\sum_I \bar h(l_i) |I|\,|J(I)|^2\1_{L_i}(a_I)$. We first remark that
$$
\Ex\Big[\sum_I \bar h(l_i) |I|\,|J(I)|^2\1_{L_i}(a_I)\Big]=N^{-1}O(\Ex[r_n^2]),\qquad\forall i=1,\ldots,N.
$$
Let now $i\in{2,\ldots,N-1}$ and $I$ be an interval of $\Pi^1$ satisfying $a_I\in L_i$, then
$\big||J(I)|-|I|-\xi_1^\circ-\xi_2^\circ\big|\le (\xi_1^\circ-N^{-1})_++(\xi_2^\circ-N^{-1})_+ $, where $\xi_1^\circ$ and $\xi_2^\circ$ are two random variables distributed according to the exponential distribution with parameters $np_2$ conditionally to $\Pi^1$. Moreover, conditionally
to $\Pi^1$, $\xi_1^\circ$ and $\xi_2^\circ$ are independent. Since $N=O(n^{1-\eps})$ and $\Ex^{\Pi^1}[(\xi^\circ_j)^4]=O(n^{-4})$, by the
Cauchy-Schwarz inequality we have $\Ex^{\Pi^1}[(\xi_j^\circ-N^{-1})_+^2]=O(n^{-2-4\eps})$ for $j=1,2$. This implies that $\Ex^{\Pi^1}[|J(I)|^2]=|I|^2+4|I|(np_2)^{-1}+6(np_2)^2+O(|I|n^{-1-2\eps})$. Combining this estimate with (\ref{boundIk}), we get
$$
\Ex[\eta_i^\circ]=\bar h(l_i)|L_i|\Big(\frac{6}{p_1^2}+\frac{8}{p_1p_2}+\frac{6}{p_2^2}\Big)+n^2|L_i|O(n^{-1-2\eps})
=\Big(\frac{6}{p_1^2}+\frac{8}{p_1p_2}+\frac{6}{p_2^2}\Big)\int_{L_i} \bar h(s)\,ds+o(1).
$$
By reasoning in a similar way, we get
$\Ex[\eta_i^\circ\eta_{j}^\circ]-\Ex[\eta_i^\circ]\Ex[\eta_j^\circ]= o(|L_i|^2)$
as soon as $|i-j|>2$. Standard arguments imply that $\var[\sum_i \eta_i^\circ]=O(N\max_i\var(\eta_i^\circ))+o(N^2|L_1|^2)$.
Since $|\eta_i^\circ|\le C(nr_n)^2 |L_1|$ for every $i$, we get $\var[\sum_i \eta_i^\circ]=O(N|L_1|^2\Ex[(nr_n)^4])+o(1)=o(1)$
and the desired convergence property follows from the convergence of $\calt_n$ in $L^2$.
\end{proof}

\begin{lemma}
\label{lemA8}
Let $T>0$  and let $\mP_n^i$, $i=1,2$, be two Poisson processes on $[0,T]$ with intensities $n p_i$, $i=1,2$.
There is a constant $\nu(p_1,p_2)$ depending only on $p_1$ and $p_2$ such that for every continuous function
$h:[0,T]^3\to\bbR$
\begin{align*}
n^2\sum_{I\in\Pi^1_n}\sum_{J\in\Pi^2_n}\int_{I(J)\times J(I)\times I\cap J} h(s,t,u)\,ds\,dt\,du\xrightarrow[n\to\infty]{P} \nu(p_1,p_2)
\int_0^T h(s,s,s)\,ds.
\end{align*}
\end{lemma}
\begin{proof}
The proof of this lemma follows from the invariance of the law of a Poisson process under scaling
and translation, as well as from the independence of disjoint sets' measures. It is similar to the
proofs of preceding lemmas and therefore will be omitted.
\end{proof}

\end{document}